\documentclass[11pt]{amsart}
\usepackage{amsmath} 
\usepackage[normalem]{ulem}
\usepackage{amsthm}
\usepackage{amssymb} 
\usepackage{amsfonts}
\usepackage{graphicx} 
\usepackage{hyperref}
\usepackage{color,comment}
\usepackage{soul,xcolor,placeins}
\usepackage{ccaption}
\usepackage[shortlabels]{enumitem}
\usepackage{varwidth}
\usepackage{url}

\setlist[enumerate]{label=\normalfont{(\roman*)}}

\usepackage[OT2,OT1]{fontenc}
\newcommand\cyr
{
\renewcommand\rmdefault{wncyr}
\renewcommand\sfdefault{wncyss}
\renewcommand\encodingdefault{OT2}
\normalfont
\selectfont
}
\DeclareTextFontCommand{\textcyr}{\cyr}

\newcommand{\red}[1]{\textcolor{red}{#1}}

\newtheorem{theorem}{Theorem}[section]
\newtheorem{lemma}[theorem]{Lemma}
\newtheorem{proposition}[theorem]{Proposition}
\newtheorem{prop}[theorem]{Proposition}
\newtheorem{corollary}[theorem]{Corollary}

\newtheorem*{maintheorem}{Theorem~\ref{maintheorem}}
\newtheorem*{htpytheorem}{Corollary~\ref{htpytheorem}}
\newtheorem*{diffeotheorem}{Corollary~\ref{diffeotheorem}}
\newtheorem*{bridgethm}{Theorem~\ref{bridgethm}}

\theoremstyle{definition}

\newtheorem{definition}[theorem]{Definition}

\newtheorem{remark}[theorem]{Remark}

\newtheorem{exercise}[theorem]{Exercise}

\newtheorem{observation}[theorem]{Observation}

\setstcolor{red}

\let\t\relax
\newcommand{\t}{\mathrm}

\newcommand{\proj}{\t{proj}}

\let\int\relax
\newcommand{\int}{\mathring}

\usepackage[margin=1.5in]{geometry}

\newcommand{\boundary}{\partial}

\newcommand{\D}{\mathcal{D}}

\newcommand{\K}{\mathcal{K}}

\newcommand{\R}{\mathbb{R}}

\newcommand{\tri}{\mathcal{T}}
\title[Band diagrams of immersed surfaces in 4--manifolds]{Band diagrams of immersed surfaces in 4--manifolds}

    \author[Mark Hughes]{Mark Hughes}
    \address{Brigham Young University\\Provo, UT, 84602 USA}
    \email{hughes@mathematics.byu.edu}
    
    \author[Seungwon Kim]{Seungwon Kim}
    \address{Institute for Basic Science -- Center for Geometry and Physics\\Postech, \mbox{Pohang-si}, Gyeongsangbuk-do, South Korea}
    \email{math751@gmail.com}
    
    \author[Maggie Miller]{Maggie Miller}
    \address{Stanford University\\Stanford, CA, 94305 USA}
    \email{maggie.miller.math@gmail.com}
    

\thanks{MM was supported by NSF Grant DGE-1656466 at Princeton University, then NSF Grant DMS-2001675 at Massachusetts Institute of Technology, and even later by a Clay Research Fellowship.}

\let\oldtocsection=\tocsection
\let\oldtocsubsection=\tocsubsection
\let\oldtocsubsubsection=\tocsubsubsection
\renewcommand{\tocsection}[1]{\hspace{0em}\oldtocsection{#1}}
\renewcommand{\tocsubsection}[2]{\hspace{1em}\oldtocsubsection{#1}{#2}}
\renewcommand{\tocsubsubsection}[2]{\hspace{2em}\oldtocsubsubsection{#1}{#2}}

\begin{document}
\newgeometry{top=1.2in}

\maketitle

\begin{abstract}
We study immersed surfaces in smooth 4--manifolds via singular banded unlink diagrams. Such a diagram consists of a singular link with bands inside a Kirby diagram of the ambient 4--manifold, representing a level set of the surface with respect to an associated Morse function.  We show that every self-transverse immersed surface in a smooth, orientable, closed 4--manifold can be represented by a singular banded unlink diagram, and that such representations are uniquely determined by the ambient isotopy or equivalence class of the surface up to a set of singular band moves which we define explicitly. By introducing additional finger, Whitney, and cusp diagrammatic moves, we can use these singular band moves to describe homotopies or regular homotopies as well.

Using these techniques, we introduce bridge trisections of immersed surfaces in arbitrary trisected 4--manifolds and prove that such bridge trisections exist and are unique up to simple perturbation moves. We additionally give some examples of how singular banded unlink diagrams may be used to perform computations or produce explicit homotopies of surfaces.
\end{abstract}

\tableofcontents
\addtocontents{toc}{\protect\setcounter{tocdepth}{1}}

\restoregeometry

\section{Introduction}\label{sec:intro}
	
Immersed surfaces are fundamental objects of in low--dimensional topology, showing up frequently in the study of 4--manifolds.  For example, immersed disks play a key role in Freedman's proof of the topological $h$--cobordism theorem and the classification of simply-connected smooth 4--manifolds \cite{freedman1982topology}.  One reason for the prominent part they play lies in how abundant they are when compared to their embedded counterparts.  In particular, maps of surfaces into smooth 4--manifolds can always be perturbed slightly to yield smooth immersions with transverse double points.

Despite their importance, immersed surfaces and their isotopies are difficult to describe explicitly outside of a few concrete examples.  While diagrammatic techniques have been developed to describe both smooth 4--manifolds and embedded surfaces (see, e.g.\ \cite{carter1993reidemeister,carter1998knotted,bandpaper,kamada19932,kamada2002braid,meier2017bridge,meier2018bridge,roseman1998reidemeister}), methods of studying immersed surfaces diagrammatically have not been established as fully in the literature, aside from a few examples (see, e.g.\ \cite{kamada2018presentation} for a diagrammatic framework for representing immersed surfaces in $\mathbb{R}^4$ via marked graph diagrams).

In this paper we introduce a new diagrammatic system for describing immersed surfaces in smooth, oriented, closed 4--manifolds called {\emph{singular banded unlink diagrams}}.   Such a diagram consists of a Kirby diagram for the ambient 4--manifold along with a decorated singular (4-valent) link with bands attached away from vertices (see Section~\ref{sec:surfaces} for details).  As a Kirby diagram of $X$ is uniquely determined by a Morse function $h$ and its gradient $\nabla h$, given two singular banded unlink diagrams in the same Kirby diagram (induced by the same Morse function on $X$), it makes sense to ask whether they determine isotopic surfaces. Even with singular banded unlink diagrams in two different Kirby diagrams of $X$, we can still ask whether they describe equivalent surfaces.  With this in mind, we define a set of moves called \emph{singular band moves} in Figures~\ref{fig:oldisomoves} and~\ref{fig:newisomoves}, which allow us to relate the diagrams of any two immersed surfaces which are ambiently isotopic.  When combined with Kirby moves to the ambient diagram, these moves are also sufficient to relate equivalent surfaces.

This work generalizes earlier results in \cite{bandpaper}, where the authors define \emph{banded unlink diagrams} of smoothly \emph{embedded} surfaces in smooth 4--manifolds, and present a family of moves (called \emph{band moves}) to describe isotopies between such surfaces.  More precisely, given a smoothly embedded surface $\Sigma$ in a smooth oriented closed 4--manifold $X$ and a self-indexing Morse function $h:X\rightarrow \mathbb{R}$, we obtain a diagram $\D(\Sigma)$ which is well-defined up to band moves and depends only on the ambient isotopy class of $\Sigma$ inside $X$.  Furthermore, given the diagram $\D(\Sigma)$ we may recover the pair $(X,\Sigma)$ up to diffeomorphism.  If, in addition, we also specify the Morse function $h:X\rightarrow \mathbb{R}$ then the surface $\Sigma \subset X$ is determined up to isotopy. In the special case that $X^4=S^4$ and $h$ is a standard (i.e.\ $h$ has no index 1, 2, or 3 critical points), these results are originally due to Swenton \cite{swenton} and Kearton--Kurlin  \cite{KK}.

 Unless otherwise stated we will assume that $X$ is a closed, smooth, oriented 4--manifold.  Our main theorems are as follows.
 
 \begin{maintheorem}
 Let $\Sigma$ be a smoothly immersed, self-transverse surface in a 4--manifold $X$.  Then any choice of a self-indexing Morse function $h:X\to \R$ (with one index 0 point) and a gradient-like vector field $\nabla h$ on $X$ induces a singular banded unlink diagram $\D(\Sigma)$ of $(X,\Sigma)$ that is well-defined up to singular band moves.  
 
Furthermore, let $\D (\Sigma)$ and $\D(\Sigma')$ be singular banded unlink diagrams of immersed surfaces $\Sigma$ and $\Sigma'$ in $X$.  

 \begin{enumerate}
     \item\label{main1} The diagrams $\D (\Sigma)$ and $\D(\Sigma')$ are related by band moves and Kirby moves if and only if there is a diffeomorphism $(X,\Sigma)\cong (X,\Sigma')$.
     
\item\label{main2} If $\D (\Sigma)$ and $\D(\Sigma')$ are induced by the same self-indexing Morse function $h$ and gradient-like vector field $\nabla h$, then $\D(\Sigma)$ and $\D(\Sigma')$ are related by band moves if and only if $\Sigma$ and $\Sigma'$ are ambiently isotopic.
 \end{enumerate}
 \end{maintheorem}
 
Note that part~\ref{main2} of Theorem~\ref{maintheorem} clearly implies part~\ref{main1}, so we will focus on proving part~\ref{main2}.  Furthermore, since Kirby diagrams of two 4--manifolds can be related by a sequence of Kirby moves if and only if they are diffeomorphic, we obtain the following corollary:
 
 \begin{diffeotheorem}
 Let $\D$ and $\D'$ be singular banded unlink diagrams of surfaces $\Sigma$ and $\Sigma'$ immersed in diffeomorphic 4--manifolds $X$ and $X'$. There is a diffeomorphism taking  $(X,\Sigma)$ to $(X',\Sigma')$ if and only if there is a sequence of singular band moves and Kirby moves taking $\D$ to $\D'$.
 \end{diffeotheorem}
 
 Without much extra work, we may also extend Theorem~\ref{maintheorem} to consider homotopy instead of isotopy.
 
 \begin{htpytheorem}
 Let $\Sigma$ and $\Sigma'$ be self-transverse surfaces smoothly immersed in $X$, and let $\D(\Sigma)$ and $\D(\Sigma')$ be singular banded unlink diagrams induced by the same self-indexing Morse function and gradient-like vector field on $X$.  
 \begin{enumerate}
\item The surfaces $\Sigma$ and $\Sigma'$ are regularly homotopic if and only if $\D(\Sigma)$ and $\D(\Sigma')$ can be related by a sequence of singular band moves and the finger/Whitney moves illustrated in Figure~\ref{fig:newhtpymoves}.
 
 \item  The surfaces $\Sigma$ and $\Sigma'$ are homotopic (without specifying regularity) if and only if $\D(\Sigma)$ and $\D(\Sigma')$ are related by singular band moves, finger/Whitney moves, and cusp moves as illustrated in Figure~\ref{fig:newhtpymoves}.
 \end{enumerate}
 \end{htpytheorem}
 
One application of the authors' results in \cite{bandpaper} was to prove the uniqueness of bridge trisections of surfaces in arbitrary trisected 4--manifolds up to perturbation. In Section~\ref{sec:trisectionimmerse} we define the notions of \emph{bridge position} and \emph{bridge trisections} for immersed surfaces in trisected 4--manifolds, and in Section~\ref{sec:trisectionunique} we prove an analogous uniqueness statement.

\begin{bridgethm}
Let $(X^4,\tri)$ be a trisected 4--manifold. Let $\Sigma$ be a self-transverse immersed surface in $X^4$. Then $\Sigma$ can be isotoped into bridge position with respect to $\tri$, yielding a bridge trisection of $\Sigma$ with respect to $\tri$. Moreover, any two bridge trisections of $\Sigma$ with respect to $\tri$ are related by $\tri$--preserving isotopy, perturbations, and vertex pertubations (and their inverses).
\end{bridgethm}

The moves  referenced in Theorem~\ref{bridgethm} are defined in Section~\ref{sec:trisectionembed}. For experts, we will say now that the perturbation move is the standard perturbation move that increases the number of disks of $\Sigma$ in one section of the trisection, while vertex perturbation is supported in a neighborhood of the trisection surface and simply passes a self-intersection of $\Sigma$ from one piece of the trisection to another.

\subsection*{Organization}
\vspace{10pt}
\begin{itemize}
\item[$\circ$] In {\bf{Section~\ref{sec:singularbandedunlinkdiagrams}}} we lay out the framework of singular banded unlink diagrams.  
\\
\\
\noindent We begin in {\bf{Section~\ref{sec:markedlinks}}} with a discussion on marked singular banded links.
 In {\bf{Section~\ref{sec:surfaces}}}, we describe how to use these decorated singular links to obtain immersed surfaces. In {\bf{Section~\ref{sec:0and1std}}} we discuss two subclasses of immersed surfaces that will be needed to prove Theorem~\ref{maintheorem} and its corollaries in {\bf{Section~\ref{sec:isotopy}}}.   
\\
 \item[$\circ$]
In {\bf{Section~\ref{sec:bridgetrisections}}} we turn our attention to bridge trisections. 
\\
\\
\noindent We review the theory of bridge trisections of embedded surfaces in {\bf{Section~\ref{sec:trisectionembed}}}.  In {\bf{Section~\ref{sec:trisectionimmerse}}} we adapt  the notions of trivial tangles and bridge position to singular links, before defining bridge position for immersed surfaces in {\bf{Section~\ref{sec:trisectionbridgeimmerse}}} and showing that every immersed surface in a smooth 4--manifold can be arranged in this position.  It is here that we define the various moves on immersed bridge trisections referenced in Theorem~\ref{bridgethm}.  In {\bf{Section~\ref{sec:bridgesingularlink}}} we then proceed to adapt the singular banded unlinks developed in {\bf{Section~\ref{sec:singularbandedunlinkdiagrams}}} to bridge trisections, before using the uniqueness results for singular banded unlinks to prove Theorem~\ref{bridgethm} in {\bf Section~\ref{sec:trisectionunique}}.
\\
\item[$\circ$] In {\bf{Section~\ref{sec:examples}}} we give some additional sample applications of the usefulness of singular banded unlink diagrams.
\\
\\
\noindent In {\bf Section~\ref{sec:kirk}} we show how one may compute the Kirk invariant (see \cite{schneiderman}) of a spherical link using these diagrams.  In {\bf Section~\ref{sec:stabilization}} we prove that homologous immersed oriented surfaces with the same number of positive and negative self-intersections are stably isotopic (i.e.\ become isotopic after surgery along some collection of arcs).  Finally, in {\bf Section~\ref{sec:cassonwhitney}} we show that certain 2--spheres embedded in $S^4$ can be trivialized by a single finger and Whitney move (recovering a fact originally proved in \cite{joseph2020unknotting}).
\end{itemize}

\subsection*{Acknowledgements}

This project began in Spring 2019 when the second and third authors visited the first at Brigham Young University.

MM thanks Peter Teichner and Mark Powell for useful discussions in Fall 2020, which influenced the viewpoint of this paper, particularly in Sections~\ref{sec:0and1std} and~\ref{sec:isotopy}.

\addtocontents{toc}{\protect\setcounter{tocdepth}{2}}

\newpage
\section{Singular banded unlink diagrams}\label{sec:singularbandedunlinkdiagrams}

\subsection{Marked singular banded links}\label{sec:markedlinks}
\label{section:diagramsexistanceanduniqueness}

In this section we introduce marked singular banded links, which are the combinatorial objects we will use to describe self-transverse immersed surfaces in 4--manifolds.  In what follows, all manifolds and maps between them should be assumed to be smooth.  All isotopies of immersed (or embedded) submanifolds are assumed to be ambient isotopies unless otherwise specified. Note that we are isotoping the images of immersions rather than immersions themselves.   

\subsubsection{Marked singular links} We begin by defining special singular links with additional data recorded at their double points.
\begin{definition}
Let $M^3$ be an orientable 3--manifold.  A \emph{singular link} $L$ in $M$ is the image of an immersion $\iota: S^1 \sqcup \cdots \sqcup S^1 \rightarrow M$ which is injective except at isolated double points that are not tangencies. At every double point $p$ we include a small disk $v\cong D^2$ embedded in $M$ such that $(v,v \cap L) \cong (D^2,\{(x,y)\in D^2\,|\,xy=0\})$.  We refer to these disks as the {\emph{vertices}} of $L$. 
 \end{definition}
 (Equivalently, a singular link is a 4--valent fat-vertex graph smoothly embedded in $M$.) For now, our motivating idea is that $M$ will correspond to some level set of a 4--manifold $X$, and the double points of a singular link $L$ in $M$ will correspond to the isolated double points an immersed surface in $X$. Because these double points are isolated, we expect the singularities of $L$ to be resolved away from the level set $M$. We must make a choice of how to resolve each double point.

\begin{definition}
A \emph{marked singular link} $(L, \sigma)$ in $M$ is a singular link $L$ along with decorations $\sigma$ on the vertices of $L$, as follows: say that $v$ is a vertex of $L$, with $\boundary v\cap \overline{(L\setminus v)}$ consisting of the four points $p_1, p_2, p_3, p_4$ in cyclic order. Choose a co-orientation of the disk $v$. On the positive side of $v$, add an arc connecting $p_1$ and $p_3$. On the negative side of $v$, add an arc connecting $p_2$ and $p_4$. See Figure~\ref{fig:singularlink}, left. A choice of $\sigma$ involves making a fixed choice of decoration on $v$, for all vertices $v$ of $L$.
\end{definition}

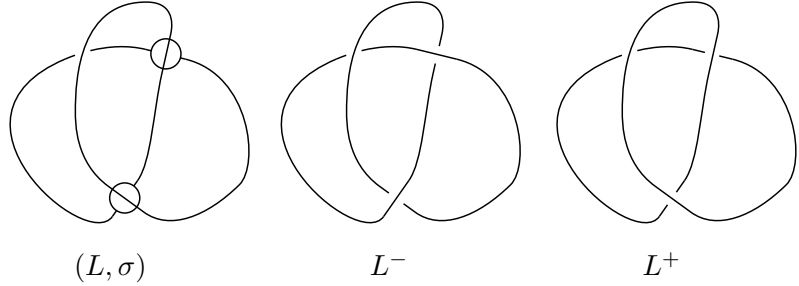
\begin{figure}
%% Accompanies image file 'singularlink.pdf' (pdf, eps, ps)
\begingroup%
  \makeatletter%
  \providecommand\color[2][]{%
    \errmessage{(Inkscape) Color is used for the text in Inkscape, but the package 'color.sty' is not loaded}%
    \renewcommand\color[2][]{}%
  }%
  \providecommand\transparent[1]{%
    \errmessage{(Inkscape) Transparency is used (non-zero) for the text in Inkscape, but the package 'transparent.sty' is not loaded}%
    \renewcommand\transparent[1]{}%
  }%
  \providecommand\rotatebox[2]{#2}%
  \newcommand*\fsize{\dimexpr\f@size pt\relax}%
  \newcommand*\lineheight[1]{\fontsize{\fsize}{#1\fsize}\selectfont}%
  \ifx\svgwidth\undefined%
    \setlength{\unitlength}{296.64333794bp}%
    \ifx\svgscale\undefined%
      \relax%
    \else%
      \setlength{\unitlength}{\unitlength * \real{\svgscale}}%
    \fi%
  \else%
    \setlength{\unitlength}{\svgwidth}%
  \fi%
  \global\let\svgwidth\undefined%
  \global\let\svgscale\undefined%
  \makeatother%
  \begin{picture}(1,0.35553509)%
    \lineheight{1}%
    \setlength\tabcolsep{0pt}%
    \put(0,0){\includegraphics[width=\unitlength,page=1]{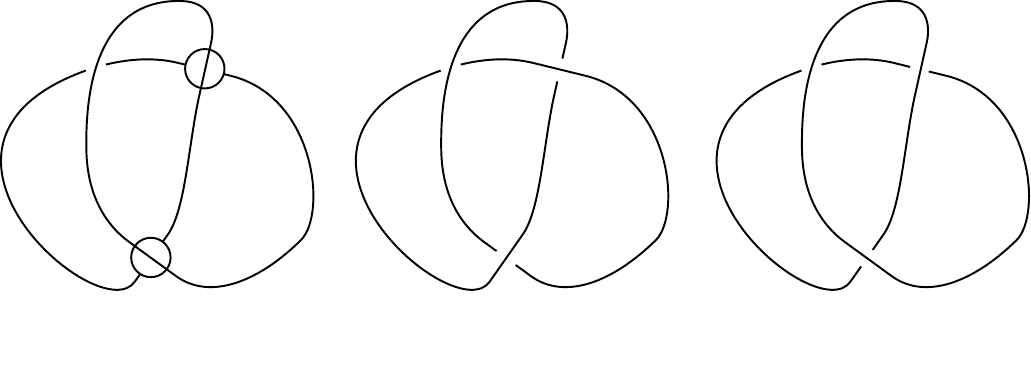}}%
    \put(0.08211034,0.00755746){\color[rgb]{0,0,0}\makebox(0,0)[lt]{\lineheight{1.25}\smash{\begin{tabular}[t]{l}$(L,\sigma)$\end{tabular}}}}%
    \put(0.45795165,0.00755746){\color[rgb]{0,0,0}\makebox(0,0)[lt]{\lineheight{1.25}\smash{\begin{tabular}[t]{l}$L^-$\end{tabular}}}}%
    \put(0.80426262,0.00755746){\color[rgb]{0,0,0}\makebox(0,0)[lt]{\lineheight{1.25}\smash{\begin{tabular}[t]{l}$L^+$\end{tabular}}}}%
  \end{picture}%
\endgroup%

\caption{{\bf{Left:}} A marked singular link $(L,\sigma)$. {\bf{Middle and {\bf{Right:}}}} The negative and positive resolutions of $(L,\sigma)$, respectively.}\label{fig:singularlink}
\end{figure}

Note that if $L$ has $n$ vertices, there are $2^n$ choices of decorations $\sigma$ so that $(L,\sigma)$ is a marked singular link.

\begin{definition}
Let $(L,\sigma)$ be a marked singular link in a 3--manifold $M$. Let $v$ be a vertex of $L$; say that on the positive side of $v$, there is an arc with endpoints $p_1$ and $p_3$ and on the negative side of $v$ there is an arc with endpoints $p_2$ and $p_4$.

Let $L^+$ denote the link in $M$ obtained from $(L,\sigma)$ by pushing the arc of $L$ between $p_1$ and $p_3$ off $v$ in the positive direction, and repeating for each vertex in $L$. We call $L^+$ the {\emph{positive resolution}} of $(L,\sigma)$ (see Figure~\ref{fig:singularlink}).

Similarly, let $L^-$ denote the link in $M$ obtained from $(L,\sigma)$ by pushing the arc of $L$ between $p_1$ and $p_3$ off $v$ in the negative direction, and repeating for each vertex in $L$. We call $L^-$ the {\emph{negative resolution}} of $(L,\sigma)$ (see Figure~\ref{fig:singularlink}).

Informally, $L^+$ is obtained from $(L,\sigma)$ by turning the decorations of $\sigma$ into new overstrands while $L^-$ is obtained by turning the decorations of $\sigma$ into new understrands.
\end{definition}

To ease notation, from now on we will always take singular links to be marked. We will generally not specify the decorations $\sigma$, and will instead write ``$L$ is a marked singular link", with $\sigma$ implicitly fixed.

 \subsubsection{Banded singular links}

 Let $L$ be a singular link, and let $\Delta_L$ denote the union of the vertices of $L$.  A \emph{band} $b$ attached to $L$ is the image of an embedding $\phi : I \times I \hookrightarrow M\backslash \Delta_L$, where $I = [-1,1]$, and $b \cap L = \phi (\{ -1,1\} \times I)$.  We call $\phi(I \times \{\tfrac{1}{2}\})$ the \emph{core} of the band $b$.  Let $L_b$  be the singular link defined by 
\[
L_b = (L \backslash \phi (\{-1,1\} \times I)) \cup \phi (I \times \{-1,1\}).
\]
Then we say that $L_b$ is the result of performing \emph{band surgery} to $L$ along $b$.  If $B$ is a finite family of pairwise disjoint bands for $L$, then we will let $L_B$ denote the link we obtain by performing band surgery along each of the bands in $B$.  We say that $L_B$ is the result of \emph{resolving} the bands in $B$.  Note that the self-intersections of $L_B$ naturally correspond to those of $L$, so a choice of markings for $L$ yields markings for $L_B$. A triple $(L,\sigma, B)$, where $(L,\sigma)$ is a marked singular link and $B$ is a family of disjoint bands for $L$ is called a \emph{marked singular banded link}. To ease notation, we may refer to the pair $(L,B)$ as a {\emph{singular banded link}} and implicitly remember that $L$ is actually a {\emph{marked}} singular link.

\subsection{Singular banded links describing surfaces}\label{sec:surfaces}
In this section, we use marked singular banded links  to describe surfaces in 4--manifolds.   Thinking of $M$ as a level set of the 4--manifold $X$, we'll begin by defining what the surface looks like in a product tubular neighborhood of $M$.
\subsubsection{Realizing surfaces in  $M^3\times [0,1]$}

Let $(L, B)$ be a marked singular banded link in the oriented 3--manifold $M$.  We will describe how to construct a surface $\Sigma$ in $M \times [0,1]$ using  $(L,B)$.

Recall first that $L^-$ is the (nonsingular) link obtained by negatively resolving each vertex of $L$.  Also notice that $L^-$ differs from $L^+$ only in a neighborhood of the vertices of $L$, where at each vertex a single strand of $L$ is pushed in the positive direction to give $L^+$, and the negative direction to give $L^-$.  For each vertex $v$ of $L$, these two opposite push-offs form a bigon in a neighborhood of $v$, which bounds an embedded disk $D_v$.  This disk $D_v$ can be chosen so that its interior intersects $L$ transversely in a single point near $v$.  For each vertex $v$ select such a disk $D_v$ (ensuring that all of these disks are pairwise disjoint), and let $D_L$ denote the union of all of these embedded disks.

We can then define the surface $\Sigma \subset M \times [0,1]$ as follows:

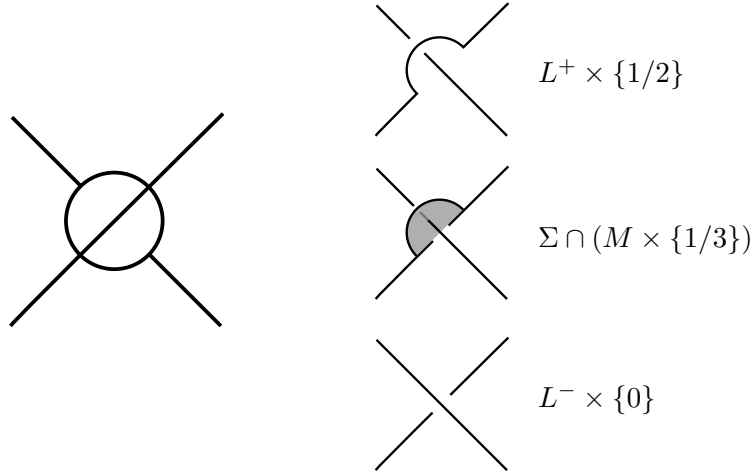
\begin{figure}
    \centering
%% Accompanies image file 'getsurface.pdf' (pdf, eps, ps)
\begingroup%
  \makeatletter%
  \providecommand\color[2][]{%
    \errmessage{(Inkscape) Color is used for the text in Inkscape, but the package 'color.sty' is not loaded}%
    \renewcommand\color[2][]{}%
  }%
  \providecommand\transparent[1]{%
    \errmessage{(Inkscape) Transparency is used (non-zero) for the text in Inkscape, but the package 'transparent.sty' is not loaded}%
    \renewcommand\transparent[1]{}%
  }%
  \providecommand\rotatebox[2]{#2}%
  \newcommand*\fsize{\dimexpr\f@size pt\relax}%
  \newcommand*\lineheight[1]{\fontsize{\fsize}{#1\fsize}\selectfont}%
  \ifx\svgwidth\undefined%
    \setlength{\unitlength}{188.53657159bp}%
    \ifx\svgscale\undefined%
      \relax%
    \else%
      \setlength{\unitlength}{\unitlength * \real{\svgscale}}%
    \fi%
  \else%
    \setlength{\unitlength}{\svgwidth}%
  \fi%
  \global\let\svgwidth\undefined%
  \global\let\svgscale\undefined%
  \makeatother%
  \begin{picture}(1,0.9375165)%
    \lineheight{1}%
    \setlength\tabcolsep{0pt}%
    \put(0,0){\includegraphics[width=\unitlength,page=1]{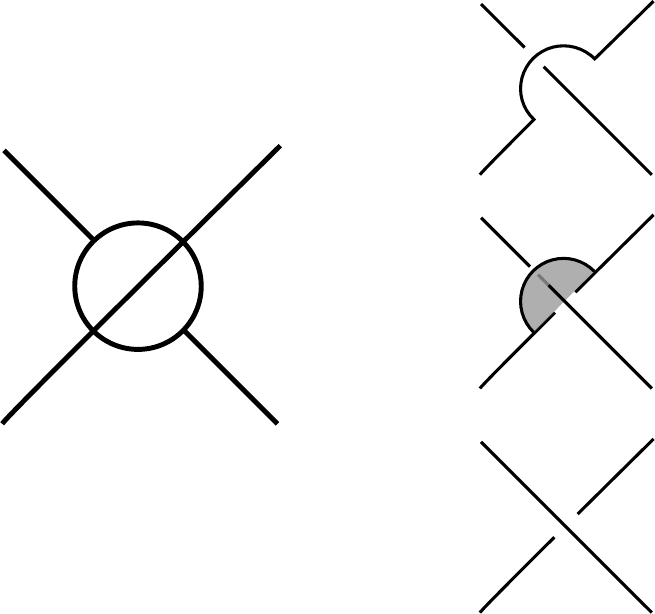}}%
    \put(1.05830315,0.12820036){\color[rgb]{0,0,0}\makebox(0,0)[lt]{\lineheight{1.25}\smash{\begin{tabular}[t]{l}$L^-\times\{0\}$\end{tabular}}}}%
    \put(1.05830315,0.44777915){\color[rgb]{0,0,0}\makebox(0,0)[lt]{\lineheight{1.25}\smash{\begin{tabular}[t]{l}$\Sigma\cap(M\times\{1/3\})$\end{tabular}}}}%
    \put(1.05830315,0.77745389){\color[rgb]{0,0,0}\makebox(0,0)[lt]{\lineheight{1.25}\smash{\begin{tabular}[t]{l}$L^+\times\{1/2\}$\end{tabular}}}}%
  \end{picture}%
\endgroup%

    \caption{{\bf{Left:}} A vertex $v$ of a marked singular link $(L,B)$. {\bf{Right:}} Part of the surface $\Sigma$ built from $(L,B)$ near $v$.}
    \label{fig:getsurface}
\end{figure}

\begin{align*}
\Sigma\cap (M\times[0,1/3)) &=L^-\times[0,1/3) \\
\Sigma\cap (M\times\{1/3\})&=(L^-\cup D_L)\times\{1/3\}\\
\Sigma\cap (M\times(1/3,2/3))&=L^+\times(1/3,2/3)\\
\Sigma\cap (M\times\{2/3\})&=(L^+\cup B)\times\{2/3\}\\
\Sigma\cap (M\times(2/3,1])&=L^+_B\times(2/3,1].
\end{align*}

In total, $\Sigma$ is a surface properly immersed in $M\times [0,1]$ with boundary $(L^-\times\{0\})\sqcup({L^+}_B\times\{1\})$, and with isolated transverse self-intersections all contained in $M\times \{1/3\}$ corresponding to the vertices of $L$. 

\begin{definition}\label{def:realizingsurfacesegment}
Let $\overline{\Sigma}{(L,B)}$ be a surface properly immersed in $M\times [0,1]$ obtained from $\Sigma$ by smoothing corners. We refer to $\overline{\Sigma}{(L,B)}$ as a {\emph{surface segment realizing $(L,B)$}}. 
\end{definition}

\begin{proposition} \label{prop:surfacesegment}
Up to ambient isotopy of $M\times [0,1]$, the surface segment $\overline{\Sigma}(L,B)$ depends only on the singular banded link $(L,B)$. 
\end{proposition}

\begin{proof}
There is a unique way (up to isotopy) to smooth the corners of $\Sigma$ in a neighborhood of $M \times \{1/3,2/3\}$. The disks $D_v$ in $M\times\{1/3\}$ are determined up to isotopy by the links $L^-$ and $L^+$, which are well-defined up to isotopy in $M$. No other choices were made in constructing $\overline{\Sigma}(L,B)$.
\end{proof}

Note that by rescaling the interval parameter, we can similarly define a surface segment realizing $(L,B)$ in any product of the form $M \times [t_1,t_2]$.  As above, the ambient isotopy class of $\overline{\Sigma}(L,B)$ will depend only on $(L,B)$.

\subsubsection{Morse functions and projections between level sets} Before describing how to construct a closed realizing surface in a 4--manifold from a singular banded unlink, it will be convenient to take a brief detour to set up some useful notation.   Let $X$ be a closed, oriented, $4$--manifold equipped with a self-indexing Morse function $h:X\to[0,4]$, where $h$ has exactly one index $0$ critical point.  In what follows it will be helpful to have a way of identifying subsets of distinct level sets $h^{-1}(t_1)$ and $h^{-1}(t_2)$.  

Suppose then that $t_1\leq t_2$, and let $x_1, \ldots, x_p$ denote the critical points of $h$ which satisfy $t_1\leq h(x_j) \leq t_2$.  Let $X_{t_1,t_2}$ denote the complement in $X$ of the ascending and descending manifolds of the critical points $x_1, \ldots, x_p$.  Then the gradient flow of $h$ defines a diffeomorphism $\rho_{t_1,t_2}:h^{-1}(t_1) \cap X_{t_1,t_2} \rightarrow h^{-1}(t_2) \cap X_{t_1,t_2}$.

\begin{definition}
We call $\rho_{t_1,t_2}$ the \emph{projection of $h^{-1}(t_1)$ to $h^{-1}(t_2)$}.  Similarly, we call $\rho^{-1}_{t_1,t_2}$ the \emph{projection of $h^{-1}(t_2)$ to $h^{-1}(t_1)$}, which we likewise denote by $\rho_{t_2,t_1}$.  
 \end{definition}
 
Note that despite calling $\rho_{t_1,t_2}$ the projection from $h^{-1}(t_1)$ to $h^{-1}(t_2)$, it is only defined on the complement of the ascending and descending manifolds of the critical points that sit between $t_1$ and $t_2$.  These projection maps allow us to define local product structures away from the ascending and descending manifolds of the critical points of $h$.

\subsubsection{Singular banded unlinks and closed realizing surfaces}\label{subsec:SBUDSandrealizingsurfaces}
We are now able to define a closed realizing surface associated to a given singular banded {\emph{unlink}}, which we define below. As above, let $X$ be a closed, oriented, $4$--manifold equipped with a self-indexing Morse function $h:X\to[0,4]$, with exactly one index $0$ critical point.

\begin{definition}\label{def:realizing}
Let $(L,B)$ be a marked singular banded link in the 3--manifold $M:=h^{-1}(3/2)$, such that $L, B \subset X_{1/2,5/2}$.  Suppose furthermore that $\rho_{3/2,1/2}(L^-)$ bounds a collection of disjoint embedded disks $D_-$ in $h^{-1}(1/2)$, and that $\rho_{3/2,5/2}(L^+_B)$ bounds a collection of disjoint embedded disks $D_+$ in $h^{-1}(5/2)$.  Then we say that $(L,B)$ is a \emph{singular banded unlink} in the manifold $X$.
\end{definition}

In plain English, $(L,B)$ is a singular banded unlink when both
\begin{enumerate}[1.]
    \item $L^-$ is an unlink when viewed as a link in $h^{-1}(3/2)$ (``below the $2$-handles"),
    \item $L^+_B$ is an unlink when viewed as a link in $h^{-1}(5/2)$ (``above the $2$-handles").
    \end{enumerate}

Fix $\varepsilon\in(0,1/2)$.  Given a singular banded unlink $(L,B)$ in $M=h^{-1}(3/2)$, and families of disks $D_+$ and $D_-$ as in Definition~\ref{def:realizing}, we can construct an immersed surface with corners $\Sigma \subset X$ as follows.

\begin{enumerate}
\item $\Sigma \cap h^{-1}(t)= \emptyset$ for  $t < 1/2$ or $t>5/2$, 
\item $\Sigma \cap h^{-1}(1/2)= D_-$,
\item $\Sigma \cap h^{-1}(t) = \rho_{1/2,t} (\partial D_-)$ for $t \in (1/2, 3/2 - \varepsilon)$,
\item \label{identification}  $\Sigma\cap h^{-1}((3/2-\varepsilon,3/2+\varepsilon))$ is a realizing surface segment in the product $h^{-1}((3/2-\varepsilon,3/2+\varepsilon)) \cong M\times (3/2-\varepsilon,3/2+\varepsilon)$ for the singular banded link $(L,B)$ in $M$,
\item $\Sigma \cap h^{-1}(t) = \rho_{5/2,t} (\partial D_+)$ for $t \in (3/2+\varepsilon, 5/2)$,
\item $\Sigma \cap h^{-1}(5/2)= D_+$.
\end{enumerate}

That is, $\Sigma$ consists from bottom to top of minimum disks, a realizing surface segment (which we recall has self-intersections and index 1 critical points), and maximum disks.

Note that the identification of $h^{-1}((3/2-\varepsilon,3/2+\varepsilon))$ with $M\times (3/2-\varepsilon,3/2+\varepsilon)$ in part~\ref{identification} above is made using the projection maps $\rho_{3/2,t}:h^{-1}(3/2) \rightarrow h^{-1}(t)$, which is a diffeomorphism for $t \in (3/2-\varepsilon,3/2+\varepsilon)$ and small $\varepsilon$.  Under this identification the boundary of the realizing surface segment will be precisely $\rho_{5/2,3/2+\varepsilon} (\partial D_+) \sqcup \rho_{1/2,3/2-\varepsilon} (\partial D_-)$, and hence the surface $\Sigma$ constructed above will be closed.

\begin{definition}\label{def:closedrealizingsurface}
Let $\Sigma(L,B)$ be an immersed surface in $X$ obtained from $\Sigma$ by smoothing corners. We refer to $\Sigma(L,B)$ as a 
\emph{(closed) realizing surface} for the singular banded unlink $(L,B)$ in $X$.
\end{definition}

The surface $\Sigma(L,B)$ is an immersed surface in $X$ with isolated, transverse self-intersections.  Note that $\Sigma(L,B)$ is obtained (up to isotopy) by smoothing the result of capping off the boundary components of $\overline{\Sigma}(L,B)$ by horizontal disks, which is possible exactly when $(L,B)$ is a singular banded \emph{unlink}.

\begin{proposition} \label{prop:uniquerealizingsurfaces}
Any two realizing surfaces for the singular banded unlink $(L,B)$ are smoothly isotopic.
\end{proposition}

\begin{proof}
We first note that choosing a different value for $\varepsilon$ changes $\Sigma$ by an isotopy through realizing surfaces.  Second, by Proposition~\ref{prop:surfacesegment} any two choices of surface segment $\overline{\Sigma}(L,B) \subset h^{-1}([3/2-\varepsilon ,3/2+\varepsilon])$ are isotopic, and this isotopy can be extended to the rest of $\Sigma \cap h^{-1}((1/2,5/2))$ using the projection maps $\rho_{t_1,t_2}$.  Finally, any choice of embedded disks $\Sigma \cap h^{-1}(1/2)$ and $\Sigma \cap h^{-1}(5/2)$ are isotopic rel boundary in $h^{-1}([0,1/2])$ and $h^{-1}([5/2,4])$ respectively, which follows from the fact that $h^{-1}([0,1/2])\cong B^4$ and $h^{-1}([5/2,4])\cong\natural^k (S^1 \times B^3)$.
\end{proof}

As the realizing surface $\Sigma(L,B)$ is determined by the singular banded unlink $(L,B)$ up to isotopy, we will often think of $\Sigma(L,B)$ as representing an isotopy class of immersed surfaces, rather than a particular representative.

\subsubsection{Singular banded unlink diagrams and Kirby diagrams}\label{sec:kirby}

We now make sense of how to describe a realizing surface as in Section~\ref{subsec:SBUDSandrealizingsurfaces} via a Kirby diagram. If one is comfortable with these diagrams, then the contents of this subsection are clear from Definition~\ref{def:realizing}: simply draw $L$ and $B$ inside a diagram for $X$ in a natural way. We now review some basic notions about Kirby diagrams.

Let $h:X\to\R$ be a self-indexing Morse function with a unique index 0 critical point, and let $n$ be the number of index 1 critical points of $h$. Fix a gradient-like vector field $\nabla h$ for $h$.  Let $M=h^{-1}(3/2)$, and let $L_2$ be the intersection of $M$ with the descending manifolds of the index 2 critical points of $h$. Perturb $\nabla h$ slightly if necessary so that this intersection is transverse, so that $L_2$ is a link in the 3--manifold $M\cong\#_n S^1\times S^2$. To each component of $L_2$, assign the framing induced by the descending manifold of the associated index 2 critical point, so that $L_2$ is actually a framed link in $M$.

Fix an $n$--component unlink $L_1$ in $S^3$. Let $V$ denote the complement of the unique (up to isotopy rel boundary) boundary-parallel disks bounded by $L_1$ in $B^4$.  Then $V$ is diffeomorphic to $\natural_n S^1\times B^3$, and we can therefore find a diffeomorphism $\phi:V\to h^{-1}([0,3/2])$. By Laudenbach--Poenaru \cite{lp} and Laudenbach \cite{laudenbachsurles}, the choice of $\phi$ is natural up to isotopy and moves that correspond to slides of $L_1$ (as a 0--framed link) in $S^3$.   Moreover, $\partial V$ can be naturally identified with the result of performing 0--surgery on $S^3$ along $L_1$, which we denote by $S_0^3(L_1)$.  By perturbing $\nabla h$ we may assume that $\phi^{-1}(L_2) \subset \partial V \cong S_0^3(L_1)$ is disjoint from the surgery solid tori, and hence we can think of $\phi^{-1}(L_2)$ as a link in $S^3$.  By abuse of notation, we will also refer to this link as $L_2$.

\begin{definition}
Let $\K:=(L_1,L_2)$ be a pair of disjoint links in $S^3$ with $L_1$ an unlink and $L_2$ framed. Suppose there is a 4--manifold $X$, a Morse function $h:X\to\R$, and a gradient-like vector field $\nabla h$ for $h$, so that $h^{-1}(3/2)$ may be identified with $S^3_0(L_1)$, and the descending manifolds of the index 2 critical points of $h$ meet $h^{-1}(3/2)$ in the framed link $L_2$. Then we call $\K$ a \emph{Kirby diagram of $X$ corresponding to $(h,\nabla h)$}.
\end{definition}

\begin{remark}
In \cite{maggiepatrick}, the third author and Naylor showed that a smooth, closed, {\emph{non-orientable}} 4-manifold $X^4$ is also determined up to diffeomorphism by  (framed) attaching regions of 0, 1, and 2-handles. If desired, one could thus make sense of diagrams of closed (immersed) surfaces in Kirby diagrams of non-orientable 4-manifolds. We choose not to pursue this explicitly in this paper for sake of simplicity.
\end{remark}

\begin{remark}
Given $h$ and $\nabla h$, a Kirby diagram $\K$ corresponding to $(h,\nabla h)$ is well-defined up to isotopy and slides over $L_1$ as long as there is no flow line of $\nabla h$ between two index 2 critical points of $h$. That is, generically we expect $h$ and $\nabla h$ to determine a Kirby diagram.

Conversely, given $\K$, the triple $(X,h,\nabla h)$ is determined up to diffeomorphism. 
\end{remark}

Let $E(\K)$ denote the complement $S^3 \backslash \nu(\mathcal{K}) $ of a small tubular neighborhood of the links $L_1,L_2$ comprising a Kirby diagram $\K$. Then given a link $L \subset E(\K)$ we may think of $L$ as describing a link in $h^{-1}(t)$ for any $t\in(0,3)$ via the projection map $\rho_{3/2,t}$.

\begin{definition}\label{def:sbud}
A \emph{singular banded unlink diagram} in the Kirby diagram $\mathcal{K}=(L_1,L_2)$ is a triple $(\mathcal{K}, L, B)$, where $L \subset E(\K)$ is a marked singular link and $B \subset E(\K)$ is a finite family of disjoint bands for $L$, such that $L^-$ bounds a family of pairwise disjoint embedded disks in $h^{-1}(1/2)$, and $L_B^+$ bounds a family of pairwise disjoint embedded disks in $h^{-1}(5/2)$. 
\end{definition}

By comparing Definition~\ref{def:sbud} to Definition~\ref{def:realizing}, we see that a singular banded unlink diagram describes an immersed realizing surface as follows.  We first note that we can identify $E(\K)$ with a subset of $h^{-1}(3/2)$ in a natural way (i.e.\ via $\nabla h)$. 
Since the banded link $L\cup B$ is disjoint from $L_1$, it can be identified with a subset of $h^{-1}(3/2)$, which we denote by $L'\cup B'$. This subset avoids the descending manifolds of the index 2 critical points of $h$.

Since $L'^-$ is disjoint from $L_1$, we can isotope it vertically downwards via the projection map $\rho_{3/2,t}$ from $h^{-1}(3/2)$ to $h^{-1}(1/2)$, where it can be capped off by a family of disjoint embedded disks in $h^{-1}(1/2)$.  Similarly, we can extend the surgered link $L'^+_{B'}$ vertically upwards from $h^{-1}(3/2)$ to $h^{-1}(5/2)$, where it can be capped off by disks. As these families of disks are unique up to isotopy rel boundary, the surface we obtain in this way from the banded unlink diagram $(\mathcal{K}, L, B)$ is well-defined up to isotopy. (See also Proposition~\ref{prop:uniquerealizingsurfaces}.)  We denote this surface by $\Sigma(\mathcal{K},L,B)$. 
\begin{definition}
We say that $\Sigma(\K,L,B)$ is a {\emph{realizing surface for}} $(\K,L,B)$, or that $(\K,L,B)$ {\emph{describes the surface}} $\Sigma(\K,L,B)$.  
\end{definition}

\begin{definition}
If $\Sigma$ is a realizing surface of a singular banded unlink diagram $(\K,L,B)$, then we say that $(\K,L,B)$ is a {\emph{singular banded unlink diagram}} for $\Sigma$, and we write $\D(\Sigma):=(\K,L,B)$. (In practice, we might drop the word ``singular", since this will be clear when $\Sigma$ is immersed.)  Note that $\Sigma$ determines $\D(\Sigma)$ uniquely up to isotopy, assuming that $\Sigma$ is a realizing surface for some diagram.
\end{definition}

\begin{definition}
Let $\Sigma$ be a subset of $X$.  Then we say that $h|_\Sigma$ is \emph{Morse} if there is a surface $F$ and an immersion $f:F\rightarrow X$ such that $\Sigma=f(F)$, and such that $h \circ f$ is a Morse function on $F$.  An index $k$ critical point of $h|_\Sigma$ is a point of the form $f(p)$, where $p$ is an index $k$ critical point of $h \circ f$.
\end{definition}

\begin{lemma} \label{thm:realizingsurface}
Let $X$ be a closed 4-manifold, and $\K$ a Kirby diagram for $X$.  Then any immersed surface $\Sigma$ in $X$ is ambient isotopic to a realizing surface $\Sigma (\K,L,B)$ for some singular banded unlink diagram $(\K,L,B)$.

\end{lemma}

\begin{proof}
After a small ambient isotopy we may assume that $h|_{\Sigma}$ is Morse.  Isotope all of the maxima of $\Sigma$ vertically upward into $h^{-1}((5/2,4))$ (generically, maxima of $\Sigma$ do not lie in the descending manifolds of index $1$ or $2$ critical points of $h$). Similarly isotope the minima of $\Sigma$ vertically downward into $h^{-1}((0,3/2))$. Isotope all of the index 1 critical points of $h|_{\Sigma}$ vertically into $h^{-1}((3/2,5/2))$ (again, index 1 critical points of $h|_{\Sigma}$ generically do not lie in the ascending manifolds of index 3 critical points or the descending manifolds of index 1 critical points). Finally, isotope the self-intersections of $\Sigma$ to lie in $h^{-1}((3/2,5/2))$ in such a way that they do not coincide with index 1 critical points of $h|_{\Sigma}$.

Now flatten $\Sigma$ as in \cite{kss}. In words, notice that $h$ and $-\nabla h$, when restricted to $\Sigma$, generically induce a CW decomposition of $\Sigma$ in which 0--cells are the index 0 critical points of $h|_{\Sigma}$, one point in the interior of each 1--cell is an index 1 critical point of $h|_{\Sigma}$, and one point in the interior of each 2--cell is an index 2 critical  point of $h|_{\Sigma}$. Perturb, if necessary, so that self-intersections of $\Sigma$ all lie outside the descending and ascending manifolds in $\Sigma$ of index 1 critical points of $h|_{\Sigma}$.

The family of gradient flow lines of $\nabla h$ in $X$ which originate on the ascending manifolds of an index 1 critical point of $h|_\Sigma$ is 2--dimensional, as is the family of gradient flow lines of -$\nabla h$ in $X$ which originate on the descending manifolds of an index 1 critical point of $h|_\Sigma$. Thus, we may generically take them all to be disjoint and also disjoint from ascending and descending manifolds of index 2 points of $h$. (We discuss this more in Section~\ref{sec:0and1std}. While this condition is generic, it is not natural -- this lack of generality precisely corresponding to the singular band moves of Theorem~\ref{maintheorem}.)

Fix $\varepsilon > 0$, and let $L^-=\Sigma \cap h^{-1}(3/2-\varepsilon)$.  Isotope $\Sigma$ near height $3/2$ so that the intersection $\Sigma\cap h^{-1}([3/2-\varepsilon,3/2+\varepsilon])$ is of the form $L^- \times [3/2-\varepsilon,3/2+\varepsilon]$. A neighborhood of each 1--cell of $\Sigma$ can be isotoped via $-\nabla h$ to a band in $h^{-1}(3/2)$ that is attached to a parallel copy of $L^-$. Let $B$ be the collection of all such bands (one for each 1-cell in $\Sigma$).

Now isotope $\Sigma$ near each self-intersection $s$ of $\Sigma$ as in the right-hand side of Figure~\ref{fig:getsurface}, i.e.\ make one of the sheets of $\Sigma$ at $s$ include a small region that is horizontal with respect to $h$, and which contains $s$. Isotope this sheet via $-\nabla h$ to push this horizontal region to $h^{-1}(3/2)$, where it can be interpreted as a marked fat vertex as in Figure~\ref{fig:getsurface} (left). Repeating for every self-intersection of $\Sigma$, we obtain a marked singular banded link $L$ in $h^{-1}(3/2)$ whose negative resolution is $L^-$.

Now $\Sigma$ intersects regions of $X$ in the following way:
\begin{align*}
h^{-1}([0,3/2-\varepsilon])&\text{ in boundary parallel disks with boundary $L^-$,}\\
h^{-1}([3/2-\varepsilon,3/2+\varepsilon])&\text{ in the realizing surface segment for $(L,B)$,}\\ 
h^{-1}([3/2+\varepsilon,5/2])&\text{ in an embedded surface on which $h$ has no critical points,}\\
h^{-1}([5/2,4])&\text{ in boundary parallel disks with boundary $L^+_B$.}
\end{align*}
We conclude that $\Sigma$ is isotopic to $\Sigma(\K,L,B)$.\end{proof}

\begin{remark}
In the proof of Proposition~\ref{thm:realizingsurface}, we made several references to genericity. That is, we made several choices of how to perturb $\Sigma$ in order to obtain $(\K,L,B)$. It may be helpful to imagine the lower-dimensional analogue of knots in $S^3$: every knot in $S^3$ is isotopic to one that projects to a knot diagram. However, not every knot in $S^3$ actually projects to a knot diagram. An arbitrary knot may, for example, have a projection that includes a cusp, self-tangency, or triple point. These conditions are not generic and can be corrected by a slight perturbation, but therein involves a choice that can yield diagrams differing by a Reidemeister move (RI, RII, RIII, respectively).  There are, of course, even ``worse" conditions, such as a knot whose projection involves a quadruple intersection. However, this condition is even ``less" generic, by which we mean:
\begin{itemize}
    \item[$\circ$] A generic knot in $S^3$ admits a projection with no triple points.
    \item[$\circ$] A generic 1--parameter family of smoothly varying knots in $S^3$ admit projections with finitely many triple points but no quadruple points.
    \item[$\circ$] A generic 2-parameter family of smoothly varying knots in $S^3$ admit projections with 1-dimensional families of triple points and finitely many quadruple points.
\end{itemize}

Thus in a 1--parameter family of knots (i.e.\ a knot isotopy), we expect to obtain diagrams that differ by an RIII move (and similarly for RI and RII), but need never consider moves involving quadruple intersections.

    Moving back to the 4-dimensional world, in order to understand to what extent a singular banded unlink diagram is well-defined up to isotopy of an immersed surface, we must understand which nongeneric behaviors of projections we expect to see a finite number of times in a 1--dimensional family of immersed surfaces. We discuss this more formally in Sections~\ref{sec:0and1std} and~\ref{sec:isotopy}.

\end{remark}

\subsubsection{Singular band moves}

The Kirby diagram $\K$ only determines the described 4--manifold $X$ up to diffeomorphism. Therefore, $(\K,L,B)$ only determines the pair $(X,\Sigma(\K,L,B))$ up to diffeomorphism; it does not make sense to say that $(\K,L,B)$ determines $\Sigma(\K,L,B)$ up to isotopy. However, if we have already identified $X$ with the manifold described by $\K$, then we can consider $\Sigma(\K,L,B)$ up to isotopy. In particular, given another singular banded unlink diagram $(\K,L',B')$ in the same Kirby diagram $\K$, there is a natural (up to isotopy) diffeomorphism between the 4--manifolds containing $\Sigma(\K,L',B')$ and $\Sigma(\K,L,B)$. Therefore, it {\emph{does}} make sense to ask whether $\Sigma(\K,L,B')$ and $\Sigma(\K,L,B)$ are ambiently isotopic, regularly homotopic, or homotopic in $X$. In this section, we define moves of singular banded unlink diagrams that describe ambient isotopies of immersed surfaces; in Sections~\ref{sec:0and1std} and~\ref{sec:isotopy} we show that indeed these moves are sufficient.

\begin{definition}\label{def:singbandmove}

Let $\D:=(\K,L,B)$ and $\D':=(\K,L',B')$ be singular banded unlink diagrams. Suppose that $\D'$ is obtained from $\D$ by a finite sequence of the moves in Figures~\ref{fig:oldisomoves} and~\ref{fig:newisomoves}. We call these moves {\emph{singular band moves}}, and say that $\D'$ {\emph{is related to $\D$ by singular band moves}}. (This relationship is clearly symmetric.)

Specifically, the singular band moves (illustrated in Figures \ref{fig:oldisomoves} and \ref{fig:newisomoves}) are:

\vspace{.1in}

\par\smallskip\noindent
\centerline{\begin{minipage}{75mm}
\begin{enumerate}
 \item \label{s41} Isotopy in $E(\K)$,
 \item Cup/cap moves,
 \item Band slides,
 \item \label{s4n}Band swims,
 \item \label{k1}Slides of bands over components of $L_2$,
 \item Swims of bands about $L_2$,
 \item \label{k1n}Slides of unlinks and bands over $L_1$,
    \item Sliding a vertex over a band,
    \item Passing a vertex past the edge of a band,
    \item Swimming a band through a vertex.
\end{enumerate}
\end{minipage}}
\par\smallskip

\vspace{.1in}

We may refer to moves~\ref{s41}--\ref{k1n} (illustrated in Figure~\ref{fig:oldisomoves}) as {\emph{band moves}} (omitting the word ``singular") since they do not involve the self-intersections of $L$.  The remainder of the moves are illustrated in Figure~\ref{fig:newisomoves}.

\end{definition}

\begin{figure}
    \centering
    \vspace{.2in}
    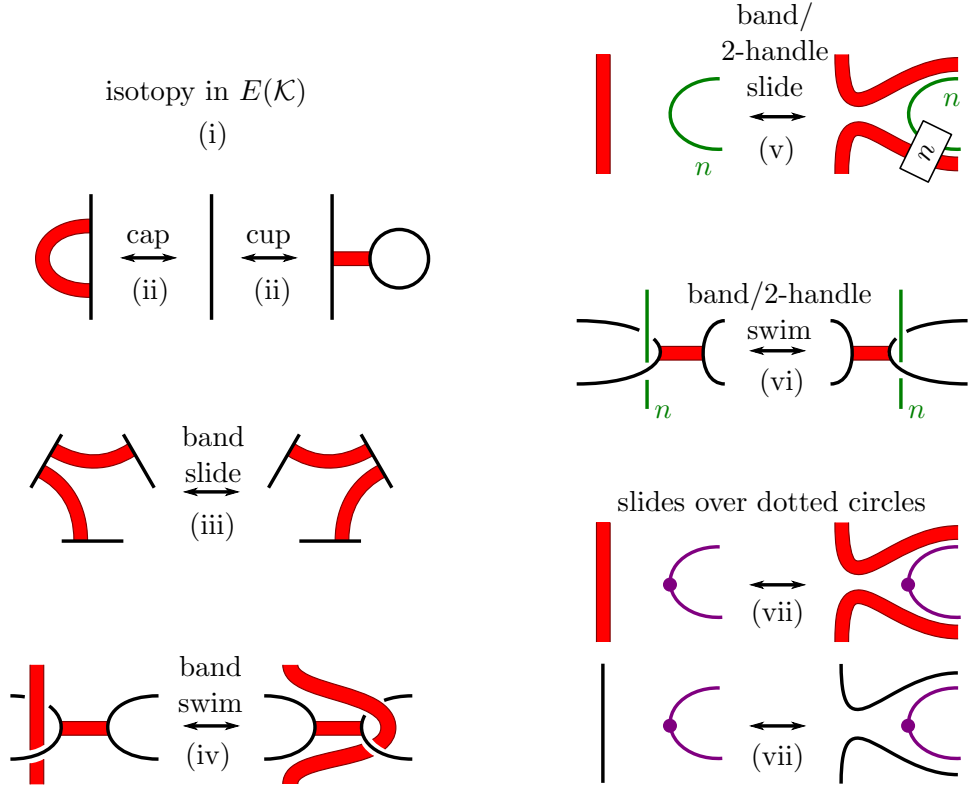
    \caption{The band moves that do not involve the self-intersections of the described surface.}
    \label{fig:oldisomoves}
    \end{figure}

\begin{figure}
    \centering
    \vspace{.05in}
    %% Creator: Inkscape 1.0.2-2 (e86c870879, 2021-01-15), www.inkscape.org
%% Accompanies image file 'newisomoves.pdf' (pdf, eps, ps)
\begingroup%
  \makeatletter%
  \providecommand\color[2][]{%
    \errmessage{(Inkscape) Color is used for the text in Inkscape, but the package 'color.sty' is not loaded}%
    \renewcommand\color[2][]{}%
  }%
  \providecommand\transparent[1]{%
    \errmessage{(Inkscape) Transparency is used (non-zero) for the text in Inkscape, but the package 'transparent.sty' is not loaded}%
    \renewcommand\transparent[1]{}%
  }%
  \providecommand\rotatebox[2]{#2}%
  \newcommand*\fsize{\dimexpr\f@size pt\relax}%
  \newcommand*\lineheight[1]{\fontsize{\fsize}{#1\fsize}\selectfont}%
  \ifx\svgwidth\undefined%
    \setlength{\unitlength}{237.40133691bp}%
    \ifx\svgscale\undefined%
      \relax%
    \else%
      \setlength{\unitlength}{\unitlength * \real{\svgscale}}%
    \fi%
  \else%
    \setlength{\unitlength}{\svgwidth}%
  \fi%
  \global\let\svgwidth\undefined%
  \global\let\svgscale\undefined%
  \makeatother%
  \begin{picture}(1,1.12549606)%
    \lineheight{1}%
    \setlength\tabcolsep{0pt}%
    \put(0,0){\includegraphics[width=\unitlength,page=1]{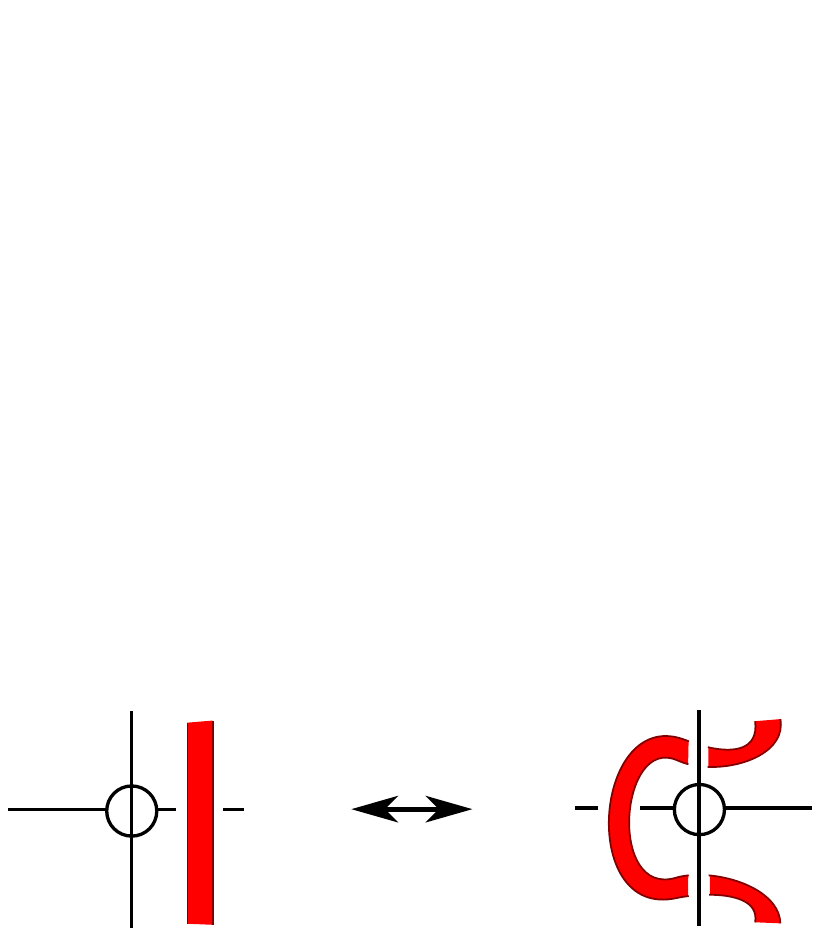}}%
    \put(0.5055278,0.26935762){\color[rgb]{0,0,0}\makebox(0,0)[t]{\lineheight{1.25}\smash{\begin{tabular}[t]{c}intersection/band\\swim\end{tabular}}}}%
    \put(0.47201341,0.05875132){\color[rgb]{0,0,0}\makebox(0,0)[lt]{\lineheight{1.25}\smash{\begin{tabular}[t]{l}(x)\end{tabular}}}}%
    \put(0,0){\includegraphics[width=\unitlength,page=2]{newisomoves.pdf}}%
    \put(0.49437438,0.65939605){\color[rgb]{0,0,0}\makebox(0,0)[t]{\lineheight{1.25}\smash{\begin{tabular}[t]{c}intersection/band\\pass\end{tabular}}}}%
    \put(0.44634761,0.45006042){\color[rgb]{0,0,0}\makebox(0,0)[lt]{\lineheight{1.25}\smash{\begin{tabular}[t]{l}(ix)\end{tabular}}}}%
    \put(0,0){\includegraphics[width=\unitlength,page=3]{newisomoves.pdf}}%
    \put(0.48322056,1.08786624){\color[rgb]{0,0,0}\makebox(0,0)[t]{\lineheight{1.25}\smash{\begin{tabular}[t]{c}intersection/band\\slide\end{tabular}}}}%
    \put(0,0){\includegraphics[width=\unitlength,page=4]{newisomoves.pdf}}%
    \put(0.42557469,0.84009854){\color[rgb]{0,0,0}\makebox(0,0)[lt]{\lineheight{1.25}\smash{\begin{tabular}[t]{l}(viii)\end{tabular}}}}%
  \end{picture}%
\endgroup%

    \caption{The singular band moves that involve self-intersections of the described surface.}
    \label{fig:newisomoves}
    \end{figure}

\begin{exercise}
If $\D$ and $\D'$ are related by singular band moves, then $\Sigma(\D)$ and $\Sigma(\D')$ are ambiently isotopic.
\end{exercise}

In the future, we will refer to moves by name rather than number to avoid confusion.

In Figures~\ref{fig:seungwonmove}--\ref{fig:intswimalt} we illustrate some other useful moves on singular banded unlink diagrams that are achievable by a combination of singular band moves. We call these moves $\star$ (Figure~\ref{fig:seungwonmove}), the upside-down intersection/band swim (Figure~\ref{fig:upsidedownswim}), the intersection pass (Figure~\ref{fig:intpass}), the intersection swim (Figures~\ref{fig:intswim} and~\ref{fig:intswimalt}), the intersection/2--handle slide (Figure~\ref{fig:int2handleslide}) and the intersection/2--handle swim (Figure~\ref{fig:int2handleswim}).

\begin{figure}
    \centering
    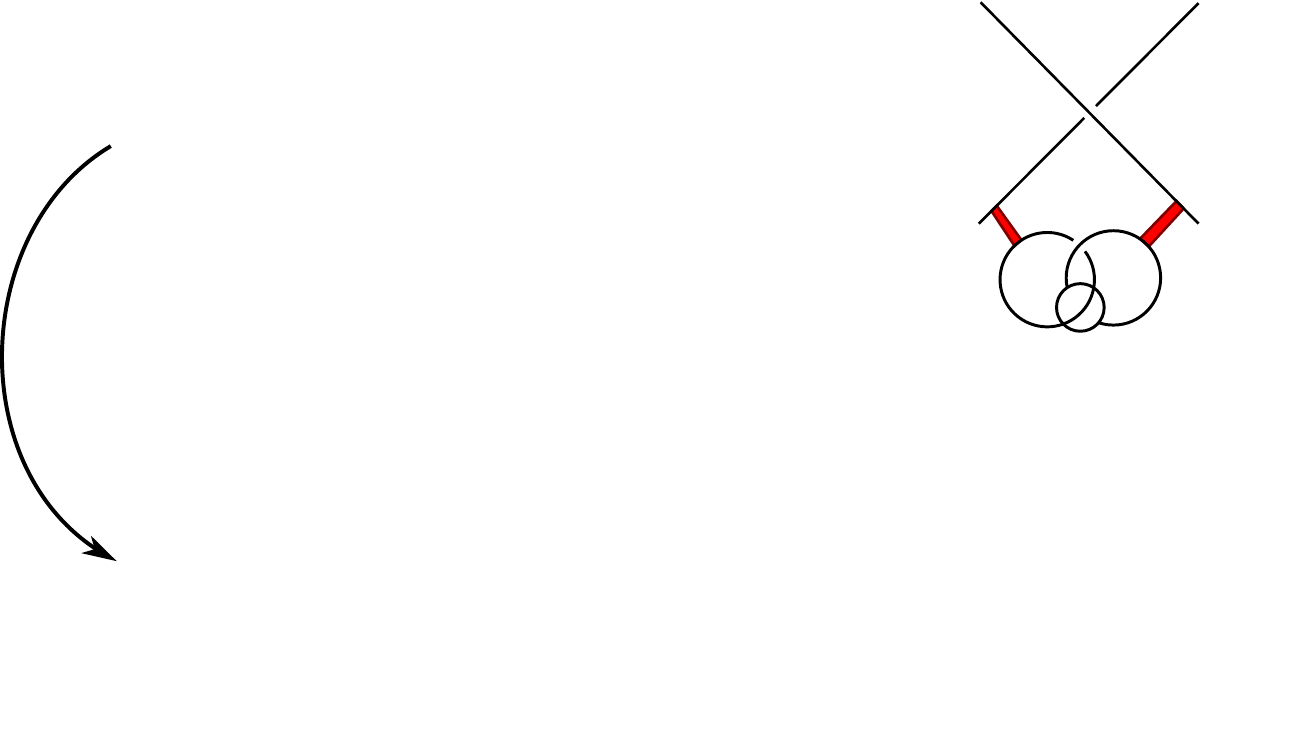
    \caption{The $\star$ move moves a vertex onto two new unlink components (or the reverse). In Figures~\ref{fig:upsidedownswim},~\ref{fig:intswim},~\ref{fig:intswimalt} we see that the $\star$-move can be used (in conjunction with singular band moves) to achieve other seemingly natural moves.}
    \label{fig:seungwonmove}
\end{figure}

\begin{figure}
    \centering
    %% Creator: Inkscape 1.0.2-2 (e86c870879, 2021-01-15), www.inkscape.org
%% Accompanies image file 'upsidedownswim.pdf' (pdf, eps, ps)
\begingroup%
  \makeatletter%
  \providecommand\color[2][]{%
    \errmessage{(Inkscape) Color is used for the text in Inkscape, but the package 'color.sty' is not loaded}%
    \renewcommand\color[2][]{}%
  }%
  \providecommand\transparent[1]{%
    \errmessage{(Inkscape) Transparency is used (non-zero) for the text in Inkscape, but the package 'transparent.sty' is not loaded}%
    \renewcommand\transparent[1]{}%
  }%
  \providecommand\rotatebox[2]{#2}%
  \newcommand*\fsize{\dimexpr\f@size pt\relax}%
  \newcommand*\lineheight[1]{\fontsize{\fsize}{#1\fsize}\selectfont}%
  \ifx\svgwidth\undefined%
    \setlength{\unitlength}{376.70414962bp}%
    \ifx\svgscale\undefined%
      \relax%
    \else%
      \setlength{\unitlength}{\unitlength * \real{\svgscale}}%
    \fi%
  \else%
    \setlength{\unitlength}{\svgwidth}%
  \fi%
  \global\let\svgwidth\undefined%
  \global\let\svgscale\undefined%
  \makeatother%
  \begin{picture}(1,0.59039263)%
    \lineheight{1}%
    \setlength\tabcolsep{0pt}%
    \put(0,0){\includegraphics[width=\unitlength,page=1]{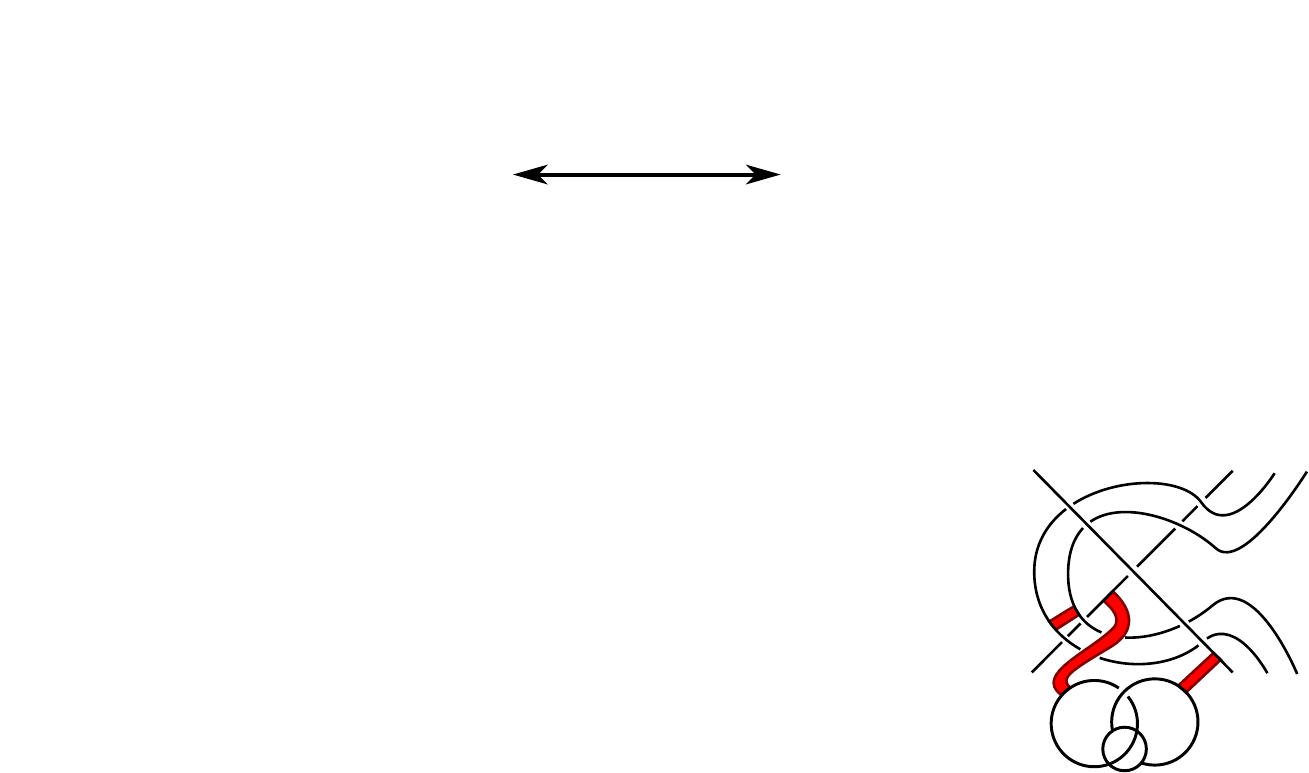}}%
    \put(0.48462013,0.5641656){\color[rgb]{0,0,0}\makebox(0,0)[t]{\lineheight{1.25}\smash{\begin{tabular}[t]{c}upside-down\\intersection/band\\swim\end{tabular}}}}%
    \put(0,0){\includegraphics[width=\unitlength,page=2]{upsidedownswim.pdf}}%
    \put(0.09016922,0.3152791){\color[rgb]{0,0,0}\makebox(0,0)[t]{\lineheight{1.25}\smash{\begin{tabular}[t]{c}$\bigstar$\end{tabular}}}}%
    \put(0,0){\includegraphics[width=\unitlength,page=3]{upsidedownswim.pdf}}%
    \put(0.25876044,0.15575662){\color[rgb]{0,0,0}\makebox(0,0)[lt]{\lineheight{1.25}\smash{\begin{tabular}[t]{l}isotopy\end{tabular}}}}%
    \put(0,0){\includegraphics[width=\unitlength,page=4]{upsidedownswim.pdf}}%
    \put(0.61524166,0.15575662){\color[rgb]{0,0,0}\makebox(0,0)[lt]{\lineheight{1.25}\smash{\begin{tabular}[t]{l}band swim\end{tabular}}}}%
    \put(0,0){\includegraphics[width=\unitlength,page=5]{upsidedownswim.pdf}}%
    \put(0.9652803,0.30522334){\color[rgb]{0,0,0}\makebox(0,0)[t]{\lineheight{1.25}\smash{\begin{tabular}[t]{c}$\bigstar$\end{tabular}}}}%
  \end{picture}%
\endgroup%

    \caption{We can achieve the upside-down intersection/band swim by performing $\star$ and singular band moves.}
    \label{fig:upsidedownswim}
\end{figure}

\begin{figure}
    \centering
    %% Creator: Inkscape 1.0.2-2 (e86c870879, 2021-01-15), www.inkscape.org
%% Accompanies image file 'intpass.pdf' (pdf, eps, ps)
\begingroup%
  \makeatletter%
  \providecommand\color[2][]{%
    \errmessage{(Inkscape) Color is used for the text in Inkscape, but the package 'color.sty' is not loaded}%
    \renewcommand\color[2][]{}%
  }%
  \providecommand\transparent[1]{%
    \errmessage{(Inkscape) Transparency is used (non-zero) for the text in Inkscape, but the package 'transparent.sty' is not loaded}%
    \renewcommand\transparent[1]{}%
  }%
  \providecommand\rotatebox[2]{#2}%
  \newcommand*\fsize{\dimexpr\f@size pt\relax}%
  \newcommand*\lineheight[1]{\fontsize{\fsize}{#1\fsize}\selectfont}%
  \ifx\svgwidth\undefined%
    \setlength{\unitlength}{372.61745183bp}%
    \ifx\svgscale\undefined%
      \relax%
    \else%
      \setlength{\unitlength}{\unitlength * \real{\svgscale}}%
    \fi%
  \else%
    \setlength{\unitlength}{\svgwidth}%
  \fi%
  \global\let\svgwidth\undefined%
  \global\let\svgscale\undefined%
  \makeatother%
  \begin{picture}(1,0.48840162)%
    \lineheight{1}%
    \setlength\tabcolsep{0pt}%
    \put(0,0){\includegraphics[width=\unitlength,page=1]{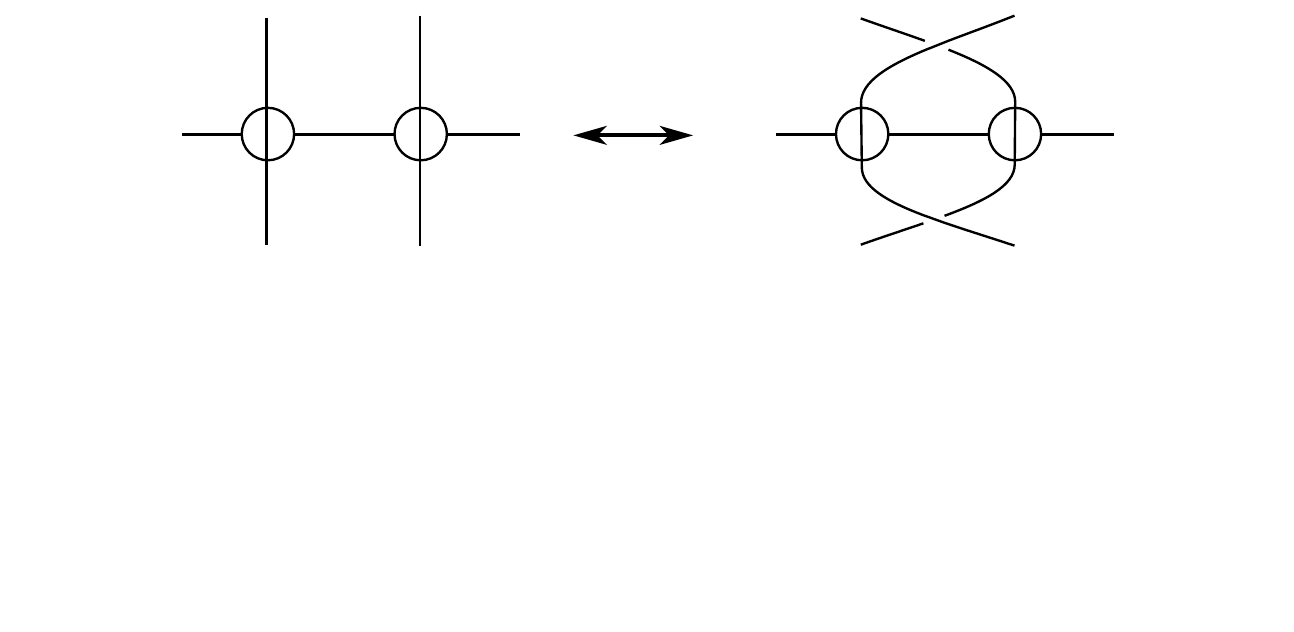}}%
    \put(0.49340024,0.46332754){\color[rgb]{0,0,0}\makebox(0,0)[t]{\lineheight{1.25}\smash{\begin{tabular}[t]{c}intersection\\pass\end{tabular}}}}%
    \put(0,0){\includegraphics[width=\unitlength,page=2]{intpass.pdf}}%
    \put(0.31721428,0.18274393){\color[rgb]{0,0,0}\makebox(0,0)[t]{\lineheight{1.25}\smash{\begin{tabular}[t]{c}int/band\\pass\end{tabular}}}}%
    \put(0,0){\includegraphics[width=\unitlength,page=3]{intpass.pdf}}%
    \put(0.68928938,0.13754842){\color[rgb]{0,0,0}\makebox(0,0)[t]{\lineheight{1.25}\smash{\begin{tabular}[t]{c}isotopy\end{tabular}}}}%
    \put(0,0){\includegraphics[width=\unitlength,page=4]{intpass.pdf}}%
    \put(0.10707907,0.25211386){\color[rgb]{0,0,0}\makebox(0,0)[t]{\lineheight{1.25}\smash{\begin{tabular}[t]{c}$\bigstar$\end{tabular}}}}%
    \put(0,0){\includegraphics[width=\unitlength,page=5]{intpass.pdf}}%
    \put(0.90621091,0.25211386){\color[rgb]{0,0,0}\makebox(0,0)[t]{\lineheight{1.25}\smash{\begin{tabular}[t]{c}$\bigstar$\end{tabular}}}}%
  \end{picture}%
\endgroup%

    \caption{We can achieve an intersection pass by performing $\star$ and singular band moves.}
    \label{fig:intpass}
\end{figure}

\begin{figure}
    \centering
    %% Creator: Inkscape 1.0.2-2 (e86c870879, 2021-01-15), www.inkscape.org
%% Accompanies image file 'intswim.pdf' (pdf, eps, ps)
\begingroup%
  \makeatletter%
  \providecommand\color[2][]{%
    \errmessage{(Inkscape) Color is used for the text in Inkscape, but the package 'color.sty' is not loaded}%
    \renewcommand\color[2][]{}%
  }%
  \providecommand\transparent[1]{%
    \errmessage{(Inkscape) Transparency is used (non-zero) for the text in Inkscape, but the package 'transparent.sty' is not loaded}%
    \renewcommand\transparent[1]{}%
  }%
  \providecommand\rotatebox[2]{#2}%
  \newcommand*\fsize{\dimexpr\f@size pt\relax}%
  \newcommand*\lineheight[1]{\fontsize{\fsize}{#1\fsize}\selectfont}%
  \ifx\svgwidth\undefined%
    \setlength{\unitlength}{372.94487901bp}%
    \ifx\svgscale\undefined%
      \relax%
    \else%
      \setlength{\unitlength}{\unitlength * \real{\svgscale}}%
    \fi%
  \else%
    \setlength{\unitlength}{\svgwidth}%
  \fi%
  \global\let\svgwidth\undefined%
  \global\let\svgscale\undefined%
  \makeatother%
  \begin{picture}(1,0.51680719)%
    \lineheight{1}%
    \setlength\tabcolsep{0pt}%
    \put(0,0){\includegraphics[width=\unitlength,page=1]{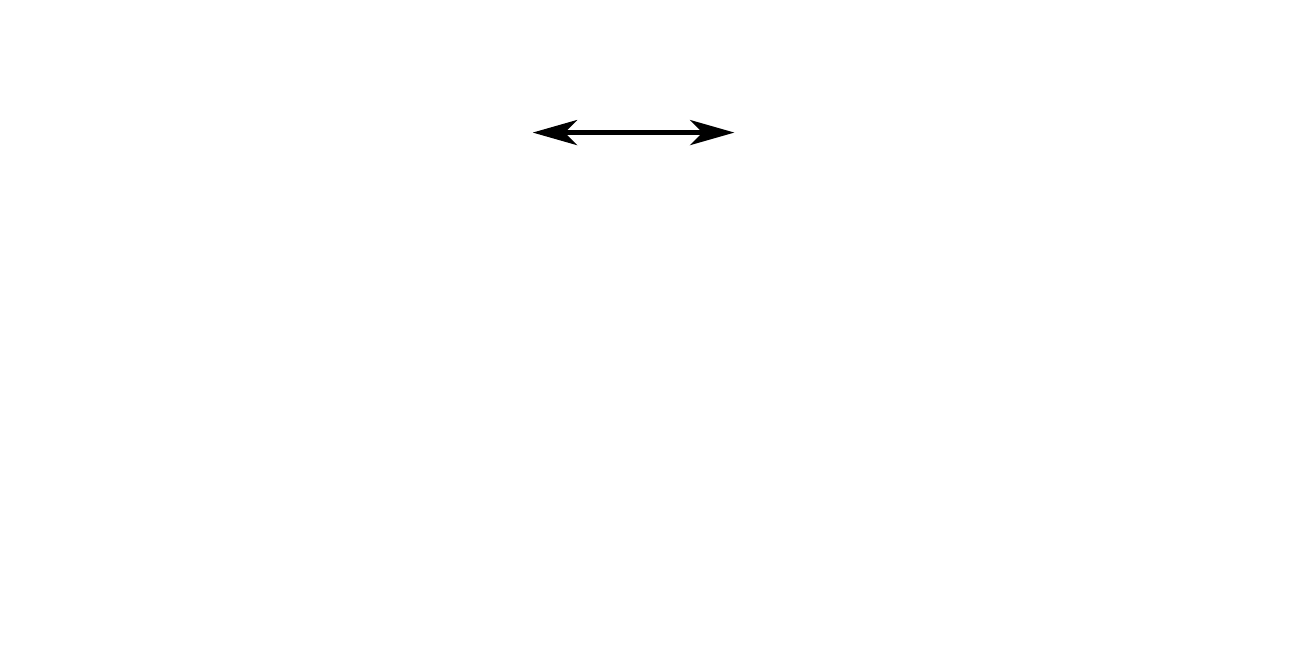}}%
    \put(0.48728899,0.49227539){\color[rgb]{0,0,0}\makebox(0,0)[t]{\lineheight{1.25}\smash{\begin{tabular}[t]{c}intersection\\swim\end{tabular}}}}%
    \put(0,0){\includegraphics[width=\unitlength,page=2]{intswim.pdf}}%
    \put(0.13838713,0.25095889){\color[rgb]{0,0,0}\makebox(0,0)[t]{\lineheight{1.25}\smash{\begin{tabular}[t]{c}$\bigstar$\end{tabular}}}}%
    \put(0,0){\includegraphics[width=\unitlength,page=3]{intswim.pdf}}%
    \put(0.31936568,0.16099454){\color[rgb]{0,0,0}\makebox(0,0)[t]{\lineheight{1.25}\smash{\begin{tabular}[t]{c}int/band\\swim\end{tabular}}}}%
    \put(0,0){\includegraphics[width=\unitlength,page=4]{intswim.pdf}}%
    \put(0.68727764,0.12027482){\color[rgb]{0,0,0}\makebox(0,0)[t]{\lineheight{1.25}\smash{\begin{tabular}[t]{c}isotopy\end{tabular}}}}%
    \put(0,0){\includegraphics[width=\unitlength,page=5]{intswim.pdf}}%
    \put(0.84498877,0.25095889){\color[rgb]{0,0,0}\makebox(0,0)[t]{\lineheight{1.25}\smash{\begin{tabular}[t]{c}$\bigstar$\end{tabular}}}}%
  \end{picture}%
\endgroup%

    \caption{We can achieve an intersection swim by performing $\star$ and singular band moves.}
    \label{fig:intswim}
\end{figure}

\begin{figure}
    \centering
    %% Creator: Inkscape 1.0.2-2 (e86c870879, 2021-01-15), www.inkscape.org
%% Accompanies image file 'intswimalt.pdf' (pdf, eps, ps)
\begingroup%
  \makeatletter%
  \providecommand\color[2][]{%
    \errmessage{(Inkscape) Color is used for the text in Inkscape, but the package 'color.sty' is not loaded}%
    \renewcommand\color[2][]{}%
  }%
  \providecommand\transparent[1]{%
    \errmessage{(Inkscape) Transparency is used (non-zero) for the text in Inkscape, but the package 'transparent.sty' is not loaded}%
    \renewcommand\transparent[1]{}%
  }%
  \providecommand\rotatebox[2]{#2}%
  \newcommand*\fsize{\dimexpr\f@size pt\relax}%
  \newcommand*\lineheight[1]{\fontsize{\fsize}{#1\fsize}\selectfont}%
  \ifx\svgwidth\undefined%
    \setlength{\unitlength}{348.53103037bp}%
    \ifx\svgscale\undefined%
      \relax%
    \else%
      \setlength{\unitlength}{\unitlength * \real{\svgscale}}%
    \fi%
  \else%
    \setlength{\unitlength}{\svgwidth}%
  \fi%
  \global\let\svgwidth\undefined%
  \global\let\svgscale\undefined%
  \makeatother%
  \begin{picture}(1,0.73448455)%
    \lineheight{1}%
    \setlength\tabcolsep{0pt}%
    \put(0,0){\includegraphics[width=\unitlength,page=1]{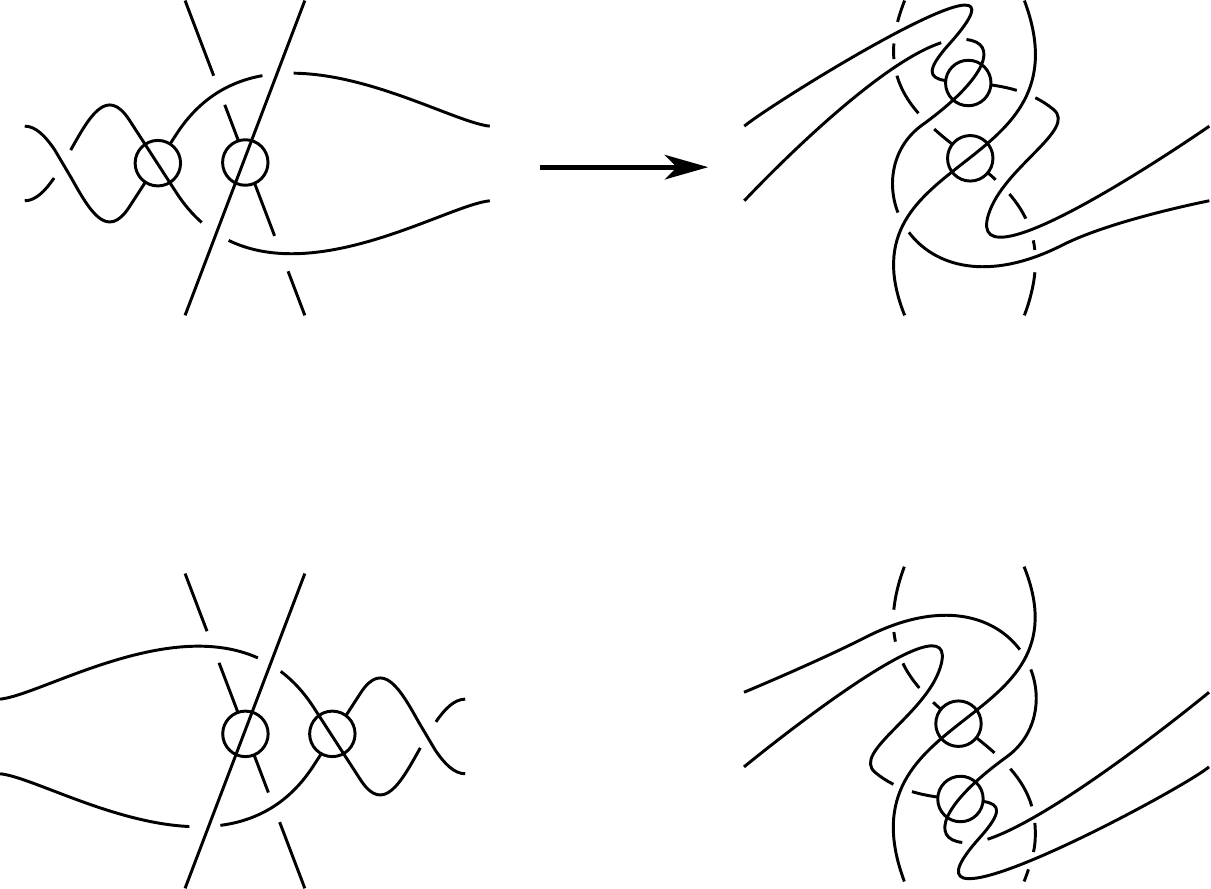}}%
    \put(0.44665803,0.62268177){\color[rgb]{0,0,0}\makebox(0,0)[lt]{\lineheight{1.25}\smash{\begin{tabular}[t]{l}isotopy\end{tabular}}}}%
    \put(0,0){\includegraphics[width=\unitlength,page=2]{intswimalt.pdf}}%
    \put(0.44665803,0.15080329){\color[rgb]{0,0,0}\makebox(0,0)[lt]{\lineheight{1.25}\smash{\begin{tabular}[t]{l}isotopy\end{tabular}}}}%
    \put(0,0){\includegraphics[width=\unitlength,page=3]{intswimalt.pdf}}%
    \put(0.88430556,0.38391431){\color[rgb]{0,0,0}\makebox(0,0)[t]{\lineheight{1.25}\smash{\begin{tabular}[t]{c}intersection\\swim\end{tabular}}}}%
  \end{picture}%
\endgroup%

    \caption{We achieve an alternate version of the intersection swim of Figure~\ref{fig:intswim}, in which one marking and one crossing are changed, via isotopy and intersection swim.}
    \label{fig:intswimalt}
\end{figure}

\begin{figure}
    \centering
    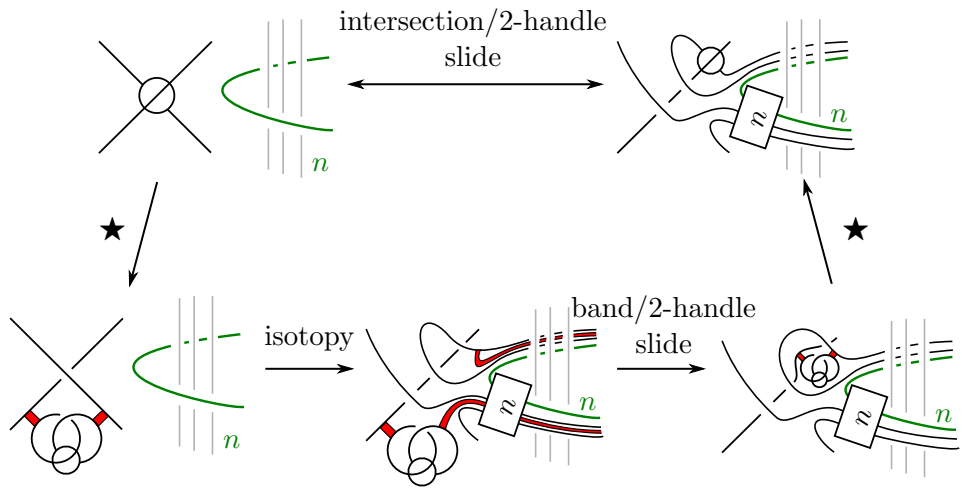
    \caption{We achieve an intersection/2--handle slide by performing $\star$ and singular band moves.}
    \label{fig:int2handleslide}
\end{figure}

\begin{figure}
    \centering
    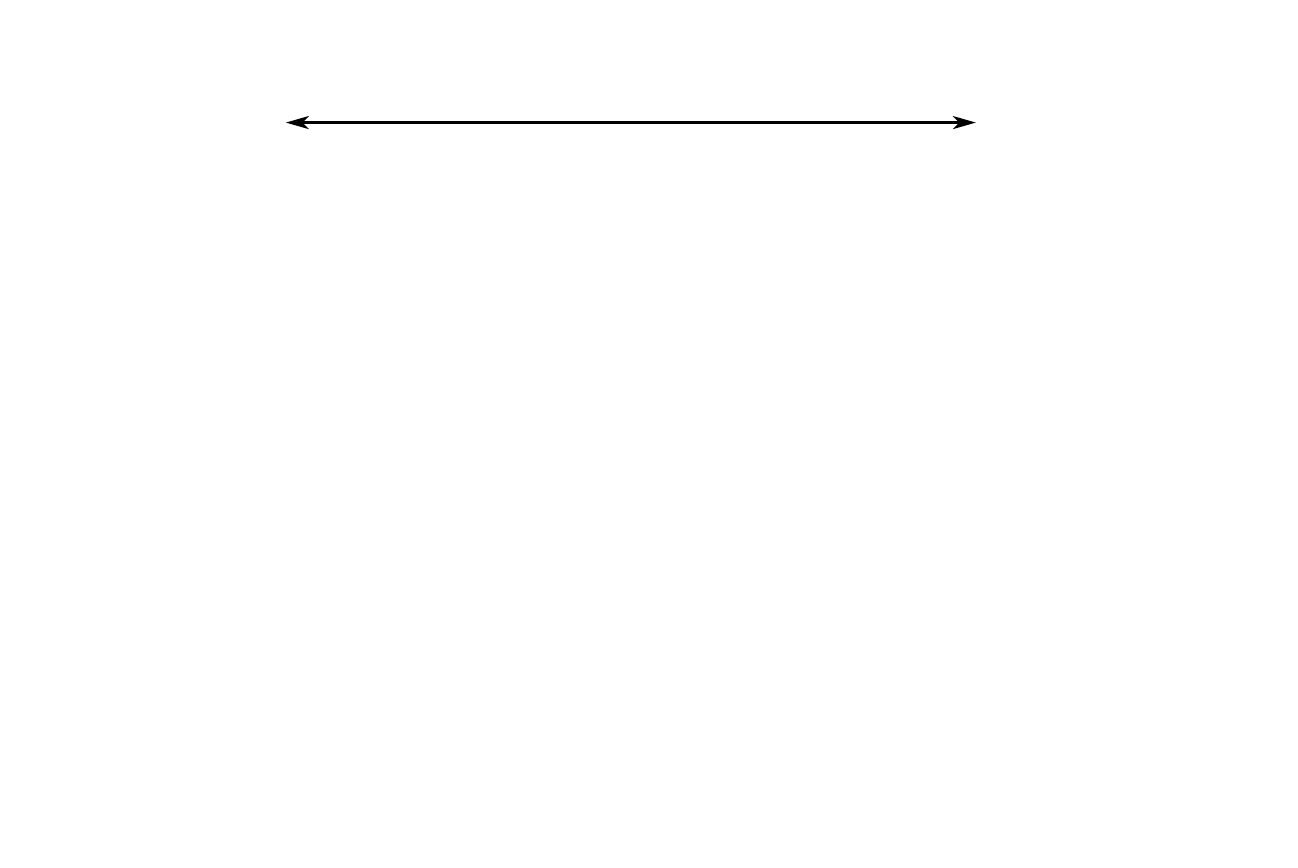
    \caption{We achieve an intersection/2--handle swim by performing singular band moves}
    \label{fig:int2handleswim}
\end{figure}

\begin{remark}
While the length of the list in Definition~\ref{def:singbandmove} may seem unwieldy, there is a general principle at play: singular band moves allow us to isotope a singular banded unlink $(L,B)$ within $\K$; or to push any vertex in $L$ or band in $B$ slightly into the past or future, do further isotopy there, and then push the vertex or band back into the present. In practice when using these diagrams, we do not explicitly break a described isotopy into a sequence of the moves of Definition~\ref{def:singbandmove}, just as how in practice one does not typically break an isotopy of a knot explicitly into a sequence of Reidemeister moves.
\end{remark}

\subsection{Ascending/descending manifolds and 0-- and 1--standard surfaces}\label{sec:0and1std}
So far, we have only used singular banded unlink diagrams to describe realizing surfaces, which are incredibly non-generic. One goal of this paper is to use singular banded unlink diagrams to describe any self-transverse immersed surface $\Sigma$. In Lemma~\ref{thm:realizingsurface}, we showed that any such $\Sigma$ is isotopic to a realizing surface. However, it is not obvious that any two realizing surfaces isotopic to $\Sigma$ have singular banded unlink diagrams that are related by singular band moves. In order to prove this, we must first restrict ourselves to understanding surfaces that intersect the ascending and descending manifolds of critical points of $h$ in prescribed ways, but yet are still more generic than realizing surfaces.

We will now consider not only the ascending/descending manifolds of critical points of $h$, but also the ascending and descending manifolds of critical points of the restricted Morse function $h|_{\Sigma}$.  From now on, fix a gradient-like vector field $\nabla h$ for the Morse function $h:X\to\R$, and let $Z$ denote $X^4\setminus\nu(\Sigma)$.

In order to obtain a gradient-like vector field on $\Sigma$ itself, we choose a splitting $TX|_\Sigma = T\Sigma \oplus N$ and let $\proj_{T{\Sigma}} : TX|_\Sigma \to T\Sigma$ be the associated bundle projection.  We can assume that the splitting is chosen so that $\proj_{T{\Sigma}}(\nabla h)|_{\Sigma}$ is a gradient-like vector field for $h|_{\Sigma}$ on $\Sigma$, which we denote by $\nabla (h|_{\Sigma})$. Note that this is actually {\emph{not}} a vector field on the immersed surface $\Sigma$ (although we could pull it back to a vector field on the abstract surface $F$), since there are two associated vectors at each point of self-intersection of $\Sigma$ (the projections of $\nabla h$ onto the tangent planes of each local sheet) -- however, we think that the language ``gradient-like vector field" is not confusing in this context.

\begin{definition}
The {\emph{ascending manifold of $p$ with respect to $\nabla h$}} consists of all points in $X$ traced out by flowing the ascending manifold of $p$ in $\Sigma$ (with respect to $\nabla (h|_\Sigma)$) along the vector field $\nabla h$. In other words, the ascending manifold of $p$ with respect to $\nabla h$ consists of all points in $X$ that lie above $p$, or that lie above a point in $\Sigma$ that lies above $p$.

Similarly, the {\emph{descending manifold of $p$ with respect to $\nabla h$}} consists of all points in $X$ traced out by flowing the descending manifold of $p$ in $\Sigma$ (with respect to $\nabla (h|_\Sigma)$) along the vector field $-\nabla h$. In other words, the descending manifold of $p$ with respect to $\nabla h$ consists of all points in $X$ that lie below $p$, or that lie below a point in $\Sigma$ that lies below $p$.
\end{definition}

Now let $p$ be a self-intersection of $\Sigma$. We will defined ascending and descending manifolds for $p$ under $\nabla h$ similar to how we defined them for critical points of $h|_\Sigma$.  To do this, note that there are two flow lines of $\nabla (h|_{\Sigma})$ in $\Sigma$ starting at $p$: one on each sheet of the self-intersection.

\begin{definition}
 The {\emph{ascending manifold of $p$ with respect to $\nabla h$}} consists of all points in $X$ traced out by flowing the two flow lines in $\Sigma$ starting at $p$ (with respect to $\nabla (h|_\Sigma)$) along the vector field $\nabla h$. In other words, the ascending manifold of $p$ with respect to $\nabla h$ again consists of all points that lie above $p$, or that lie above a point in $\Sigma$ that lies above $p$.

 Similarly, the {\emph{descending manifold of $p$ with respect to $\nabla h$}} consists of all points in $X$ traced out by flowing the two flow lines in $\Sigma$ starting at $p$ (with respect to $-\nabla (h|_\Sigma)$) along the vector field $-\nabla h$. In other words, the descending manifold of $p$ with respect to $\nabla h$ again consists of all points that lie below $p$, or that lie below a point in $\Sigma$ that lies below $p$.
\end{definition}

In what follows we will often refer to the ascending or descending manifolds of critical points of $h|_{\Sigma}$ or of self-intersections of $\Sigma$. Unless we specify otherwise, assume that this always refers to the corresponding manifolds in $X$ with respect to $\nabla h$ as defined above, rather than ascending or descending manifolds in $\Sigma$ with respect to $\nabla (h|_{\Sigma})$.

\subsubsection{1--standard surfaces}\label{subsec:1std}

Suppose that $\Sigma$ is a self-transverse immersed surface in $X$.  The following definition will be important as we consider 1--parameter families of immersed surfaces:

\begin{definition}\label{def1std}
We say that $\Sigma$ is {\emph{1--standard}} 
if the following are true:
\begin{enumerate}[(1)]
    \item\label{1std1} The surface $\Sigma$ is disjoint from the critical points of $h$.
    \item\label{1std2} The restriction $h|_{\Sigma}$ is Morse except for possibly at most one birth/death degeneracy, i.e.\ a point of $\Sigma$ about which $h|_\Sigma$ can be represented as $h|_\Sigma (x,y) = x^2 - y^3$ in some local coordinates on $\Sigma$.
    \item\label{1std3} For $k\ge n+1$, the descending manifolds of index $n$ critical points of $h$ and index $(n-1)$ critical points of $h|_{\Sigma}$ are disjoint from the ascending manifolds of index $k$ critical points of $h$ and index $(k-1)$ critical points of $h|_{\Sigma}$. 
    Moreover, self-intersections of $\Sigma$ are disjoint from the ascending manifolds of index $3$ critical points of $h$ and descending manifolds of index $1$ critical points of $h$.    In other words, we ask for $n$--dimensional descending manifolds to be disjoint from $(4-n-1)$--dimensional ascending manifolds.
    \end{enumerate}
    \end{definition}
    
\begin{remark}\label{1stdrem}

Definition~\ref{def1std} is essentially a list of all ascending/descending manifold pairs that we expect to be disjoint in a 1--parameter family of immersed surfaces by dimensional considerations, as explained in Proposition~\ref{1stdfamgeneric}. This motivates the name ``1--standard."

\end{remark}

\begin{proposition}\label{1stdfamgeneric}
Let $\Sigma_t$ be an isotopy between 1--standard surfaces $\Sigma_0$ and $\Sigma_1$. After an arbitrarily small perturbation of the isotopy $\Sigma_t$, we can assume that $\Sigma_t$ is 1--standard for all $t$.
\end{proposition}

\begin{proof}
We prove that after a small perturbation, $\Sigma_t$ satisfies each property of Definition~\ref{def1std} for all $t$.
\begin{enumerate}[(1)]
\item The critical point set of $h$ in $X\times I$ is $1$--dimensional, while the isotopy $\Sigma_t$ in $X\times I$ is $3$--dimensional. Generically, we do not expect $\Sigma_t$ to intersect a critical point of $h$ for any $t$.
\item This follows from Cerf's filtration on the space of surfaces (see, e.g.\ \cite[Chapter 1 $\S$2]{hatcher}). This is a filtration on the space $C(F)$ of all smooth maps $F\to X^4$, for $F$ a surface. The codimension--0 stratum consists of all maps $f:F\to X^4$ with $h|_{f(F)}$ Morse with critical points at distinct heights. The codimension--1 stratum  includes $f$ if either of the following is true:
\begin{itemize}
    \item[$\circ$] The restriction $h|_{f(F)}$ is Morse with exactly two critical points at the same height, but all other critical points sit at distinct heights.
    \item[$\circ$]  The restriction $h|_{f(F)}$ is Morse except for one birth or death degeneracy. This degeneracy and all critical points are at distinct heights.
\end{itemize}

Suppose $\Sigma_0$ has $n$ points of self-intersection. Fix $2n$ points $x_1$, $y_2$,$\ldots$, $x_n$, $y_n$ in $F$ and choose $f_t:F\to X$ so that $f_t(F)=\Sigma_t$ and $f_t(x_i)=f_t(y_i)$ for all $i$ and $t$. Now a small perturbation of the path $f_t$ from $f_0$ to $f_1$ in $C(F)$ yields a path $g_t$ that is completely contained in the codimension--0 and codimension--1 strata of Cerf's filtration with $g_0=f_0,g_1=f_1$. Since $g_t$ lies in these strata, $g_t(F)$ has the property~\ref{1std2} of Definition~\ref{def1std} for all $t$.  Moreover, if the perturbation is sufficiently small we may assume that $g_t(F)$ is an immersed surface with $n$ transverse double points for all $t$, all of which are contained in a fixed small tubular neighborhood of $\Sigma_t$.

While $g_t$ is a homotopy from $g_0$ to $g_1$, we may view its image as an isotopy between the singular submanifolds $\Sigma_0$ and $\Sigma_1$ in $X$. We must now check that this isotopy extends to an ambient isotopy of $X$. That is, while we have argued that we may perturb $f_t$ to achieve property~\ref{1std2}, we must explain why this perturbation may be achieved by perturbing the ambient isotopy from $\Sigma_0$ to $\Sigma_1$, since there is a distinction between the immersions $f_t$ and their images $f_t(F)=\Sigma_t$. This is relatively standard (and indeed stated without proof in e.g.\ \cite{fq}): choose small disjoint closed disks $D_{x_i}, D_{y_i}$ ($i=1,\ldots, n$) in $F$, centered at $x_i$ and $y_i$ respectively. We can fix a family of coordinates on a closed tubular neighborhood of $g_t(F)$ near the self-intersections so that centered about $g_t(x_i)=g_t(y_i)$, we have a closed ball $B_i=g_t(D_{x_i})\times g_t(D_{y_i})$ intersecting $g_t(F)$ in $$\frac{(g_t(D_{x_i})\times \{0\})\cup(\{0\}\cup g_t(D_{y_i})}{(g_t(x_i)\times 0)\sim (0\times g_t(y_i))}.$$ 

Now we may extend the isotopy $\Sigma_0\to \Sigma_1$ that is the image of $g_t$ to an isotopy $\phi_t$ of $\Sigma_0\cup B_1\cup\cdots\cup B_n$ by specifying that $\phi_t(g_0(a),g_0(b))=(g_t(a),g_t(b))$ for all $a\in D_{x_i}, b\in D_{y_i}$, since $B_i=g_0(D_{x_i})\times g_0(D_{y_i})$. Then $\phi_t(B_i)=g_t(D_{x_i})\times g_t(D_{y_i})$. Since the $B_i$'s are balls, the isotopy $\phi_t|_{\cup_i B_i}$ extends to an ambient isotopy $\psi_t$ of $X$. The composition $\psi_t^{-1}\phi_t$ then fixes $B_i$ pointwise for each $i$.

Now since $\Sigma_0\cap (X\setminus\text{int}(B_1\sqcup\cdots\sqcup B_n))$ is an embedded submanifold (whose boundary is not tangent to the boundary of $X\setminus\text{int}(B_1\cup\cdots\cup B_n)$; i.e.,\ $\Sigma_0\cap (X\setminus\text{int}(B_1\sqcup\cdots\sqcup B_n))$ is neat in the sense of \cite{hirschbook}) whose boundary is fixed by $\psi_t^{-1}\phi_t$, we may use usual isotopy extension 
to extend $\psi_t^{-1}\phi_t$ to an ambient isotopy. Then since $\psi_t^{-1}$ is an ambient isotopy (and hence a diffeotopy starting at the identity map), we conclude that $\phi_t$ extends to a diffeotopy starting at the identity map, i.e. an ambient isotopy.

We conclude that our original ambient isotopy from $\Sigma_0$ to $\Sigma_1$ may be perturbed to another ambient isotopy of $\Sigma_0$ to $\Sigma_1$ which satisfies property~\ref{1std2} of 1--standardness at all times.

\item 
 Note that both ascending and descending manifolds are parallel to $\nabla h$, so rather than counting transverse intersections, we count the dimension of the space of line intersections (parallel to $\nabla h$) of these ascending and descending manifolds. (In other words, we count the dimension of the moduli space of unparametrized flow lines of $-\nabla h$ from one critical or intersection point to another.) 
An $n$--dimensional descending manifold and a $(4-k)$--dimensional ascending manifold thus have expected dimension $$(n-1)+((4-k)-1)-(4-1)=n-k-1$$ as a space of lines. 
For $k\ge n+1$, this expected dimension is at most $-2$, so we conclude that we may perturb $\Sigma_t$ (which by the previous item we see may be obtained by perturbing a path of immersions $f_t$ in $C(F)$) to achieve property~\ref{1std3}.
\end{enumerate}
\end{proof}

\subsubsection{0--standard surfaces}\label{subsec:0std} In Remark~\ref{1stdrem}, we explained that the definition of 1--standardness comes from studying generic 1--parameter families. That is, the conditions in Definition~\ref{def1std} are generically true for 1--parameter families of surfaces.  We now define a slightly more restrictive condition on the surfaces we study, which we expect to be violated a finite number of times in a generic 1--parameter family.

\begin{definition}\label{def0std}
We say that $\Sigma$ is {\emph{0--standard}} 
if it is 1--standard and the following are true:
\begin{enumerate}[(1)]
    \item\label{0std1} The restriction $h|_{\Sigma}$ is Morse.
    \item\label{0std1} Whenever $p$ and $q$ are either index $2$ critical points of $h$, index 1 critical points of $h|_{\Sigma}$, or self-intersections of $\Sigma$ (not necessarily of the same type), and $p \neq q$, the descending manifold of $p$ is disjoint from the ascending manifold of $q$.
 In short: 2--dimensional descending manifolds are disjoint from 2--dimensional ascending manifolds.
    \end{enumerate}
    \end{definition}

    \begin{remark}
    Roughly speaking, a surface $\Sigma$ is 0--standard if its index 1 critical points (viewed as bands) and self-intersections do not lie above each other, or above or below any index 2 critical points of $h$. These conditions are what cause a projection of $\Sigma$ to a singular banded unlink diagram to not be well-defined, leading to the existence of singular band moves. The failure of 0--standardness corresponds exactly to these ``bad" projections. 
    \end{remark}

\begin{proposition}\label{0stdfamgeneric}
Let $\Sigma_t$ be an isotopy between 0--standard surfaces $\Sigma_0$ and $\Sigma_1$. After an arbitrarily small perturbation of the isotopy $\Sigma_t$, it is true that $\Sigma_t$ is 1--standard for all $t$, and 0--standard for all but finitely many $t$. 
\end{proposition}

\begin{proof}
It follows from Proposition~\ref{1stdfamgeneric} that 1--parameter families $\Sigma_t$ of surfaces are generically 1--standard for all $t$. We now consider the conditions of Definition~\ref{def0std} separately.

\begin{enumerate}[(1)]
\item This is well-known by Cerf (see, e.g.\ \cite[Chapter 1 $\S$2]{hatcher}).
\item A pair of complementary-dimension descending and ascending manifolds meet with expected dimension $-1$ (as a space of lines parallel to $\nabla h$).  Therefore, Property~\ref{0std1} is generically true at all but finitely many times during a 1--parameter family of surfaces. 
\end{enumerate}
\end{proof}

\begin{proposition}\label{prop:canonical}
Suppose $\Sigma$ is 0--standard. Then there is a singular banded unlink diagram $\D$ determined by $\Sigma$ up to isotopy and slides over the 1--handle circles $L_1$, so that $\Sigma$ is ambiently isotopic to $\Sigma(\D)$.
\end{proposition}

\begin{remark}
In Proposition~\ref{prop:canonical}, note that slides of $L$ and $B$ over $L_1$ correspond to horizontal isotopy near $h^{-1}(3/2)$. That is, if we instead defined a Kirby diagram to constitute a framed link $L_2$ of 2--handle attaching circles in $\#_k S^2\times S^1$ (rather than drawing 1-handles in $S^3$ as dotted circles), Proposition~\ref{prop:canonical} would say that if $\Sigma$ is 0--standard, then the singular banded unlink $\D$ is canonical up to isotopy.
\end{remark}

\begin{proof}[Proof of Proposition~\ref{prop:canonical}]

Since $\Sigma$ is 0--standard (and hence 1--standard), we may vertically isotope $\Sigma$ so that the minima of $h|_{\Sigma}$ lie below $h^{-1}(3/2)$, the maxima of $h|_{\Sigma}$ lie above $h^{-1}(5/2)$, and the self-intersections/bands of $\Sigma$ lie in $h^{-1}((3/2,5/2))$.

Since $\Sigma$ is 0--standard, the descending manifolds (using $\nabla (h|_{\Sigma})$) of index 1 critical points of $h|_{\Sigma}$ end at index 0 points of $h|_{\Sigma}$ without meeting any index 1 points or self-intersections of $\Sigma$. Similarly, flow lines of $-\nabla (h|_{\Sigma})$ originating at self-intersections of $\Sigma$ also end at index 0 points of $h|_{\Sigma}$ without meeting any other index 1 critical points or self-intersections of $\Sigma$.

Now let $S$ be the 1--skeleton of $\Sigma$ determined by $\nabla h$, i.e.\ the 1--complex with:
\begin{enumerate}[1)]
    \item 0--cells at index 0 points of $h|_{\Sigma}$,
    \item 1--cells along the descending manifolds of index 1 critical point of $h|_{\Sigma}$, 
    \item Additional 1--cells consisting of pairs of flow lines of $-\nabla h|_\Sigma$ glued together at self-intersections of $\Sigma$.
\end{enumerate}

Isotope $\Sigma$ vertically so that the index 1 critical points of $h|_{\Sigma}$ and self-intersections of $\Sigma$ lie disjointly in $h^{-1}(3/2)$. (Here we are implicitly using the fact that since $\Sigma$ is 0--standard, these points do not lie directly above one another nor above index 2 critical points of $h$.)  Flatten $\Sigma$ near $h^{-1}(3/2)$ to turn index 1 points of $h|_{\Sigma}$ into bands whose cores are contained in 1--cells of $S$.

Since $\Sigma$ is 0--standard, the bands and self-intersections of $\Sigma\cap h^{-1}(3/2)$ are disjoint from the descending manifolds of index 2 critical points of $h$, i.e.\ they are disjoint from the attaching circles $L_2$ of the $2$--handles in $\K$.

Then $\Sigma\cap h^{-1}(3/2)$ is a singular banded link $(L,B)$, where $L^-$ is isotopic to $\Sigma\cap h^{-1}(3/2-\varepsilon)$, and $L^+_B$ is isotopic to $\Sigma\cap h^{-1}(3/2+\varepsilon)$. We conclude that $(L,B)$ is well-defined up to isotopy in $h^{-1}(3/2)\setminus($descending manifolds of index 2 critical points of $h$). Therefore, $(\K,L,B)$ is well-defined up to slides of $L$ and $B$ over the dotted circles $L_1$ of $\K$. This completes the proof of Proposition~\ref{prop:canonical}.
\end{proof}

\begin{corollary}\label{0stdisotopy}
Let $\Sigma_0$ and $\Sigma_1$ be 0--standard surfaces. 
Suppose there is an isotopy $\Sigma_t$ from $\Sigma_0$ to $\Sigma_1$ that is 0--standard for all $t$. Then $\D_0$ and $\D_1$ are related by isotopy in $E(\K)$ and slides over $L_1$.
\end{corollary}

\subsection{Conclusion: uniqueness of singular banded unlink diagrams}\label{sec:isotopy}

\subsubsection{Singular band moves and isotopy}  We are now in a position to prove our main results.

\begin{theorem}\label{thm:getmoves}
Let $\Sigma_0,\Sigma_1$ be 0--standard self-transverse immersed surfaces. Suppose there exists an isotopy $\Sigma_t$ so that $\Sigma_t$ is 1--standard for all $t$, and 0--standard for all $t\neq 1/2$.

Set $\D_t:=\D(\Sigma_t)$. Then $\D_0$ and $\D_1$ are related by singular band moves.

\end{theorem}
We break Theorem~\ref{thm:getmoves} into Propositions~\ref{getmovesprop1}--\ref{getmovesprop5}, in which we separately consider different ways in which $\Sigma_{1/2}$ may fail to be 0--standard. 

\begin{prop}\label{getmovesprop1}
Suppose that $\Sigma_{1/2}$ would be 0--standard if not for a single birth or death degeneracy. Then $\D_0$ and $\D_1$ are related by isotopy in $E(K)$ and possibly a cup or cap move.
\end{prop}

\begin{proof}
This is a standard fact. See, e.g.,\ \cite{arnold}.
\end{proof}

\begin{prop}\label{getmovesprop2}
Suppose that $\Sigma_{1/2}$ would be 0--standard if not for the descending manifold (in $\Sigma$) of $p$ meeting the ascending manifold (in $\Sigma$) of $q$, where $p$ and $q$ are each index 1 critical points of $h|_{\Sigma}$ or self-intersections of $\Sigma$.  
 Then $\D_0$ and $\D_1$ are related by isotopy in $E(K)$ and possibly a band slide, intersection/band slide, intersection/band pass, or intersection pass. 
\end{prop}

\begin{proof}
The proof of Proposition~\ref{prop:canonical} fails for $\Sigma_{1/2}$ precisely because $p$ lying above $q$ in $\Sigma$ causes indeterminacy in the 1-skeleton $S$. There are then two choices (up to small isotopy through 0--standard surfaces) in how to perturb $\Sigma$ near $p$ to obtain a 0--standard surface. See Figure~\ref{fig:getmoves1}. The resulting two singular banded unlink diagrams differ by a:

\vspace{.1in}
\renewcommand{\arraystretch}{1.5}

\noindent\hfil\begin{tabular}{l|c|r}
\hline
\begin{tabular}{l}
band slide\\\hline
intersection/band slide\\\hline
intersection/band pass\\\hline
intersection pass\\
\end{tabular}
&if&
\begin{tabular}{r}
$p$ and $q$ are index 1 points\\\hline
$p$ is a self-intersection; $q$ is an index 1 point\\\hline
$p$ is an index 1 point; $q$ is a self-intersection\\\hline
$p$ and $q$ are self-intersections.\\
\end{tabular}\\
\hline
\end{tabular}\par

\vspace{.1in}

We conclude that $\D_{1/2-\varepsilon}$ and $\D_{1/2+\varepsilon}$ are either isotopic or isotopic after one of the above moves. The same is then true of $\D_0$ and $\D_1$ by Corollary~\ref{0stdisotopy}.

\end{proof}

\begin{figure}
    \centering
    \includegraphics[width=120mm]{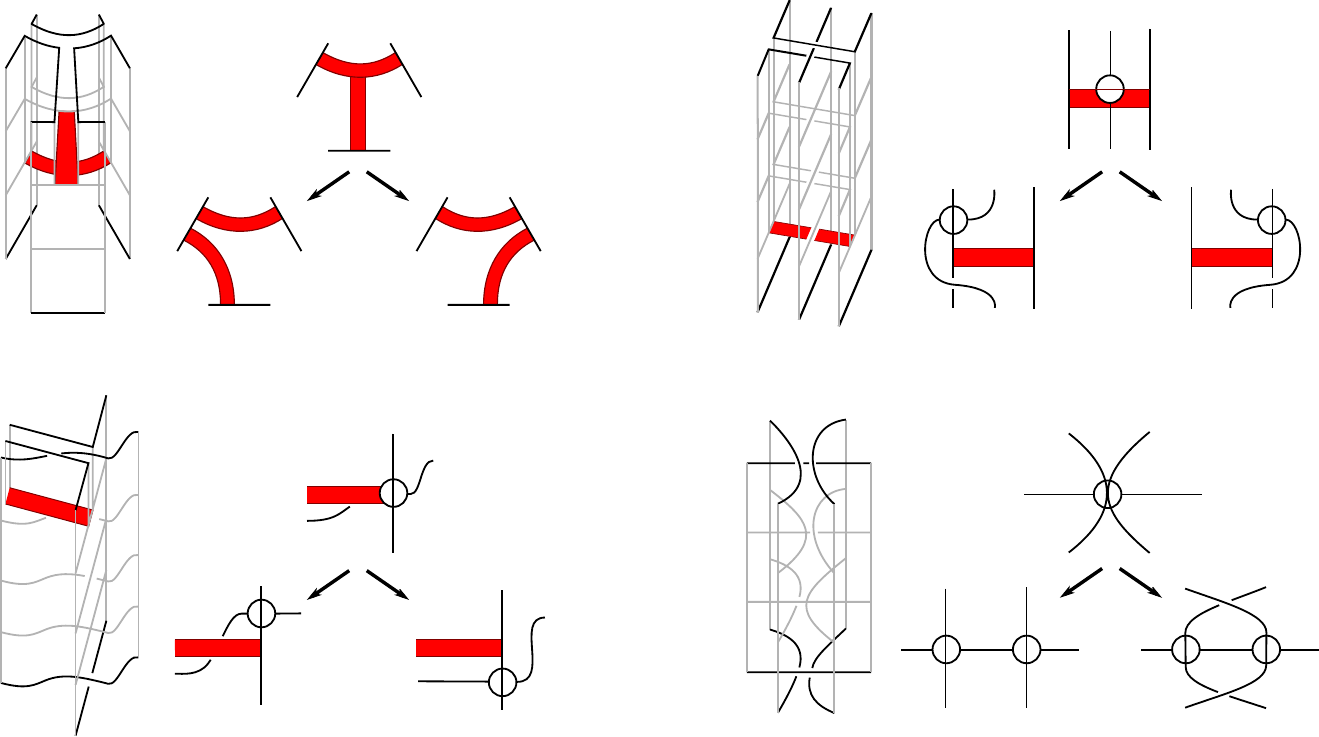}
    \caption{The cases of Proposition~\ref{getmovesprop2}. At the left of each quadrant we draw a local model about the flow line that causes $\Sigma_{1/2}$ to not be 0--standard. At the top right of each quadrant, we draw a schematic of the projection of $\Sigma_{1/2}$ to $E(\K)$, where two bands, two self-intersections, or one of each coincide. We draw arrows to indicate the two diagrams that arise if we perturb $\Sigma_{1/2}$ to be 0--standard.}
    \label{fig:getmoves1}
\end{figure}

\begin{prop}\label{getmovesprop3}
Suppose that $\Sigma_{1/2}$ would be 0--standard if not for the descending manifold of $p$ with respect to $\nabla h$ meeting the ascending manifold of $q$ with respect to $\nabla h$, where $p$ and $q$ are each index 1 critical points of $h|_{\Sigma}$ or self-intersections of $\Sigma$, and their ascending/descending manifolds intersect in their interiors (rather than in just $\Sigma$, as in Proposition~\ref{getmovesprop2}). Then $\D_0$ and $\D_1$ are related by isotopy in $E(K)$ and possibly a band swim, intersection/band swim, upside-down intersection/band swim, or intersection swim. 
\end{prop}

\begin{proof}
The proof of Proposition~\ref{prop:canonical} fails for $\Sigma_{1/2}$ because when we attempt to project the 1-skeleton of $\Sigma$ to $h^{-1}(3/2)$, the edges corresponding to $p$ and $q$ will intersect.
There are then two choices (up to small isotopy through 0--standard surfaces) in how to perturb $\Sigma$ near $p$ to obtain a 0--standard surface. See Figure~\ref{fig:getmoves2}. The resulting two singular banded unlink diagrams differ by a:

\vspace{.1in}
\renewcommand{\arraystretch}{1.5}
\noindent\hfil\scalebox{.85}{
\begin{tabular}{l|c|r}
\hline
\begin{tabular}{l}
band swim\\\hline
intersection/band swim\\\hline
upside-down intersection/band swim\\\hline
intersection swim\\
\end{tabular}
&if&
\begin{tabular}{r}
$p$ and $q$ are index 1 points\\\hline
$p$ is a self-intersection; $q$ is an index 1 point\\\hline
$p$ is an index 1 point; $q$ is a self-intersection\\\hline
$p$ and $q$ are self-intersections\\
\end{tabular}\\
\hline
\end{tabular}}
\par

\vspace{.1in}

We conclude that $\D_{1/2-\varepsilon}$ and $\D_{1/2+\varepsilon}$ are either isotopic or isotopic after one of the above moves. The same is then true of $\D_0$ and $\D_1$ by Corollary~\ref{0stdisotopy}.
\end{proof}

\begin{figure}
    \centering
    \includegraphics[width=120mm]{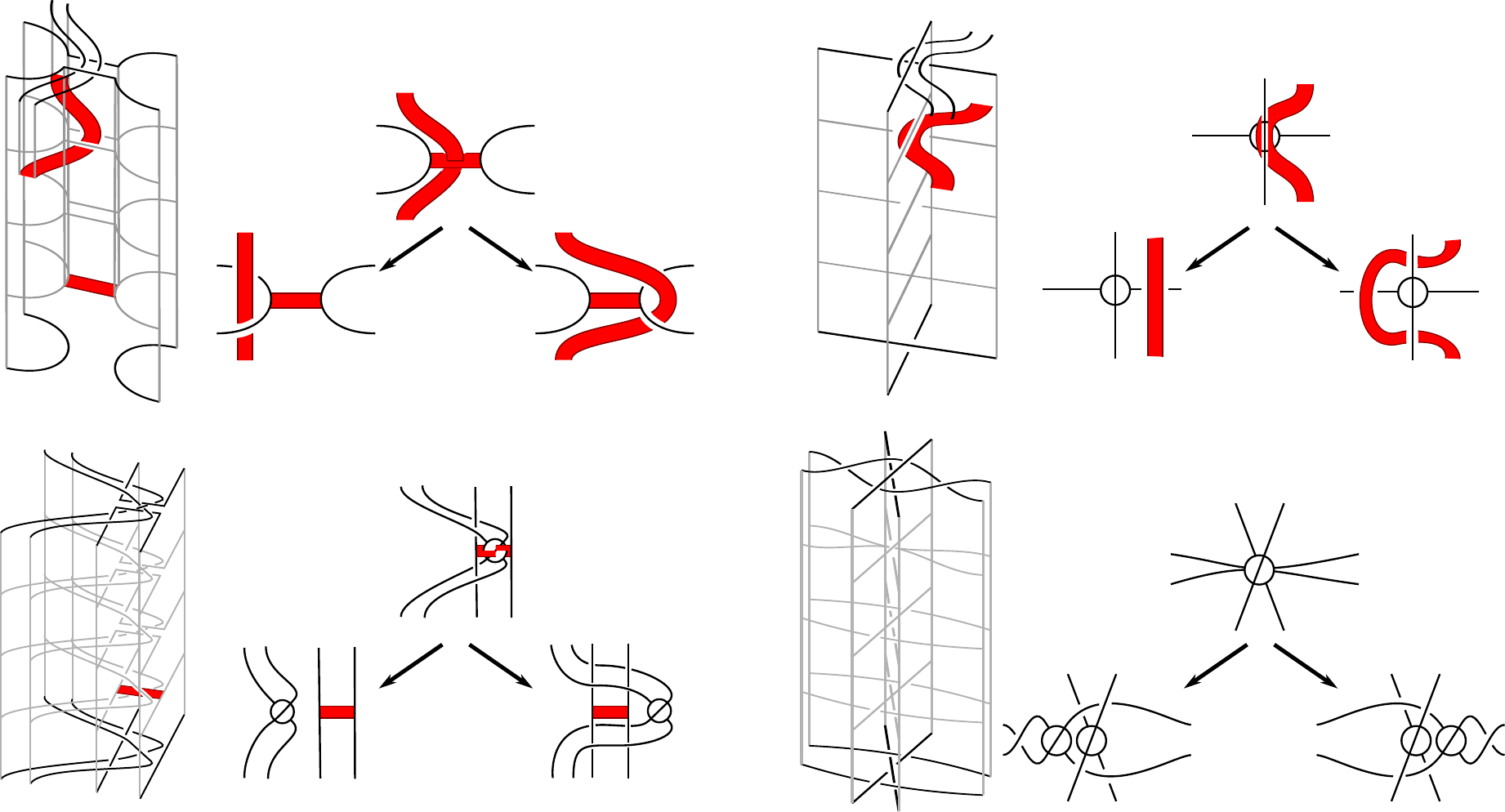}
    \caption{The cases of Proposition~\ref{getmovesprop3}. At the left of each quadrant we draw a local model about the flow line that causes $\Sigma_{1/2}$ to not be 0--standard. At the top right of each quadrant, we draw a schematic of the projection of $\Sigma_{1/2}$ to $E(\K)$, where two bands, two self-intersections, or one of each coincide. We draw arrows to indicate the two diagrams that arise if we perturb $\Sigma_{1/2}$ to be 0--standard.}
    \label{fig:getmoves2}
\end{figure}

\begin{prop}\label{getmovesprop4}
Suppose that $\Sigma_{1/2}$ would be 0--standard if not for the descending manifold of $p$ intersecting the ascending manifold of $q$, where $p$ is an index 1 critical point of $h|_{\Sigma}$ or  a self-intersection of $\Sigma$ and $q$ is an index 2 critical point of $h$. 
Then $\D_0$ and $\D_1$ are related by isotopy in $E(K)$ and possibly a band/2--handle slide or intersection/2--handle slide.

\end{prop}

\begin{proof}
The proof of Proposition~\ref{prop:canonical} fails because we cannot project the edge of the 1-skeleton of $\Sigma$ corresponding to $p$ to the level $h^{-1}(3/2)$. There are then two choices (up to small isotopy through 0--standard surfaces) in how to perturb $\Sigma$ near $p$ to obtain a 0--standard surface, with resulting singular banded unlink diagrams differing by a slide over a 2-handle. That is, the resulting two singular banded unlink diagrams differ by a:

\vspace{.1in}
\renewcommand{\arraystretch}{1.5}

\noindent\hfil
\begin{tabular}{l|c|r}
\hline
\begin{tabular}{l}
band/2-handle slide\\\hline
intersection/2-handle slide\\
\end{tabular}
&if&
\begin{tabular}{r}
$p$ is an index 1 point\\\hline
$p$ is a self-intersection
\end{tabular}\\
\hline
\end{tabular}
\par

\vspace{.1in}

We conclude that $\D_{1/2-\varepsilon}$ and $\D_{1/2+\varepsilon}$ are either isotopic or isotopic after one of the above moves. The same is then true of $\D_0$ and $\D_1$ by Corollary~\ref{0stdisotopy}.

\end{proof}

\begin{prop}\label{getmovesprop5}
Suppose that $\Sigma_{1/2}$ would be 0--standard if not for the descending manifold of $p$ intersecting the ascending manifold of $q$, where $p$ is an index 2 critical point of $h$ and $q$ is either an index 1 critical point of $h|_{\Sigma}$ or a self-intersection of $\Sigma$. 
Then $\D_0$ and $\D_1$ are related by isotopy in $E(K)$ and possibly a band/2-handle swim or intersection/2-handle swim.

\end{prop}

\begin{proof}
The proof of Proposition~\ref{prop:canonical} fails for $\Sigma_{1/2}$ because after we project the 1-skeleton of $\Sigma$ to $h^{-1}(3/2)$, the edge corresponding to $q$ will intersect the component of $L_2\subset\mathcal{K}$ corresponding to $p$. 
There are then two choices (up to small isotopy through 0--standard surfaces) in how to perturb $\Sigma$ near $p$ to obtain a 0--standard surface, with resulting singular banded unlink diagrams differing by a swim through a 2-handle attaching circle. That is, the resulting two singular banded unlink diagrams differ by a:

\vspace{.1in}
\renewcommand{\arraystretch}{1.5}

\noindent\hfil
\begin{tabular}{l|c|r}
\hline
\begin{tabular}{l}
band/2-handle swim\\\hline
intersection/2-handle swim\\
\end{tabular}
&if&
\begin{tabular}{r}
$p$ is an index 1 point\\\hline
$p$ is a self-intersection
\end{tabular}\\
\hline
\end{tabular}
\par

\vspace{.1in}

We conclude that $\D_{1/2-\varepsilon}$ and $\D_{1/2+\varepsilon}$ are either isotopic or isotopic after one of the above moves. The same is then true of $\D_0$ and $\D_1$ by Corollary~\ref{0stdisotopy}.
\end{proof}

This completes the proof of Theorem~\ref{thm:getmoves}, since Propositions~\ref{getmovesprop1}--\ref{getmovesprop5} cover all of the cases in which $\Sigma_{1/2}$ is 1--standard and not 0--standard (of course, if $\Sigma_{1/2}$ is 0--standard then Theorem~\ref{thm:getmoves} follows from Corollary~\ref{0stdisotopy}) except for the case that there are flow lines of $-\nabla h$ between index 2 critical points.  However, $h$ and $\nabla h$ are fixed during the isotopy so this does not happen.\qedsymbol

\begin{corollary}\label{cor:getmoves}
Let $\Sigma_0,\Sigma_1$ be 0--standard self-transverse immersed surfaces. Suppose there exists an isotopy $\Sigma_t$ and values $t_1<t_2<\cdots<t_n\in(0,1)$ so that $\Sigma_t$ is 0--standard for all $t\not\in\{t_1,t_2,\ldots, t_n\}$, and $\Sigma_{t_i}$ is 1--standard for each $i=1,2,\ldots, n$.

Let $\D_t:=\D(\Sigma_t)$. Then $\D_0$ and $\D_1$ are related by a sequence of singular band moves.
\end{corollary}
\begin{proof}
For each $i=1,\ldots, {n-1}$, let $s_1$ be a value in $(t_i,t_{i+1})$. By Corollary~\ref{0stdisotopy}, \begin{itemize}
    \item[$\circ$]  $\D_0$ is related to $\D_{s_1}$ by  singular band moves,
    \item[$\circ$] $\D_{s_i}$ is related to $\D_{s_{i+1}}$ by  singular band moves for $i=1,\ldots, n-1$,
    \item[$\circ$] $\D_{n-1}$ is related to $\D_1$ by  singular band moves.
    \end{itemize}
    We conclude that $\D_0$ and $\D_1$ are related by singular band moves.
\end{proof}

\subsubsection{Proof of uniqueness theorems}\label{sec:uniqueness}
We finally prove that singular banded unlink diagrams of isotopic (resp. regularly homotopic, homotopic) surfaces exist for arbitrary immersed self-transverse surfaces and are well-defined up to singular band moves. At this point, not much is left to do -- the material in Section~\ref{sec:isotopy} is essentially the whole proof that diagrams exist and are unique up to singular band moves.

\begin{theorem}\label{maintheorem}

Let $\Sigma$ be a self-transverse smoothly immersed surface in $X$. Then there is a singular banded unlink diagram $\D(\Sigma)$, well-defined up to singular band moves, so that $\Sigma$ is isotopic to the closed realizing surface for $\D(\Sigma)$. Moreover, if $\Sigma$ is isotopic to $\Sigma'$, then $\D(\Sigma)$ and $\D(\Sigma')$ are related by singular band moves.

We say that $\D(\Sigma)$ is {\emph{a singular banded unlink diagram for $\Sigma$}}, or simply that $\D(\Sigma)$ is a diagram for $\Sigma$.
\end{theorem}
\begin{proof}
Via a small perturbation, $\Sigma$ is isotopic to a 0--standard surface $\Sigma_0$. Set $\D(\Sigma):=\D(\Sigma_0)$. 
To show that $\D(\Sigma)$ is well-defined, suppose that $\Sigma_1$ is another 0--standard surface that is isotopic to $\Sigma$, and hence isotopic to $\Sigma_0$. By Proposition~\ref{0stdfamgeneric}, there is an isotopy $\Sigma_t$ from $\Sigma_0$ to $\Sigma_1$ so that $\Sigma_t$ is 1--standard for all $t$ and 0--standard for all but finitely many $t$. By Corollary~\ref{cor:getmoves}, $\D(\Sigma_0)$ and $\D(\Sigma_1)$ are related by singular band moves.

Since this argument applies to any 0--standard surface $\Sigma_1$ isotopic to $\Sigma$, we conclude that if $\Sigma$ and $\Sigma'$ are isotopic, then $\D(\Sigma)$ and $\D(\Sigma')$ are related by singular band moves.
\end{proof}

\begin{corollary}\label{diffeotheorem}
 Let $\D$ and $\D'$ be singular banded unlink diagrams of surfaces $\Sigma$ and $\Sigma'$ immersed in diffeomorphic 4--manifolds $X$ and $X'$. There is a diffeomorphism taking  $(X,\Sigma)$ to $(X',\Sigma')$ if and only if there is a sequence of singular band moves and Kirby moves taking $\D$ to $\D'$.
\end{corollary}

In addition, we can use these moves to describe homotopies of surfaces in terms of singular banded unlink diagrams.

\begin{corollary}\label{htpytheorem}
Let $\D$ and $\D'$ be singular banded unlink diagrams for surfaces $\Sigma$ and $\Sigma'$ immersed in $X$. If $\Sigma$ and $\Sigma'$ are homotopic, then $\D$ and $\D'$ are related by a finite sequence of singular band moves and the following moves (illustrated in Figure~\ref{fig:newhtpymoves}):
\begin{itemize}
    \item[$\circ$] Introducing or cancelling two oppositely marked vertices (a ``finger move" or ``Whitney move") as illustrated,
    \item[$\circ$] replacing a nugatory crossing with a vertex, or vice versa, (a ``cusp move") as illustrated.
\end{itemize}

In addition, if $\Sigma$ and $\Sigma'$ are {\emph{regularly}} homotopic, then $\D$ and $\D'$ are related by a finite sequence of singular band moves, finger moves, and Whitney moves (i.e.\ a sequence of the given moves that does not include any cusp moves.)
\end{corollary}

\begin{figure}
    \centering
%% Accompanies image file 'newhtpymoves.pdf' (pdf, eps, ps)
\begingroup%
  \makeatletter%
  \providecommand\color[2][]{%
    \errmessage{(Inkscape) Color is used for the text in Inkscape, but the package 'color.sty' is not loaded}%
    \renewcommand\color[2][]{}%
  }%
  \providecommand\transparent[1]{%
    \errmessage{(Inkscape) Transparency is used (non-zero) for the text in Inkscape, but the package 'transparent.sty' is not loaded}%
    \renewcommand\transparent[1]{}%
  }%
  \providecommand\rotatebox[2]{#2}%
  \newcommand*\fsize{\dimexpr\f@size pt\relax}%
  \newcommand*\lineheight[1]{\fontsize{\fsize}{#1\fsize}\selectfont}%
  \ifx\svgwidth\undefined%
    \setlength{\unitlength}{190.55989423bp}%
    \ifx\svgscale\undefined%
      \relax%
    \else%
      \setlength{\unitlength}{\unitlength * \real{\svgscale}}%
    \fi%
  \else%
    \setlength{\unitlength}{\svgwidth}%
  \fi%
  \global\let\svgwidth\undefined%
  \global\let\svgscale\undefined%
  \makeatother%
  \begin{picture}(1,1.13276995)%
    \lineheight{1}%
    \setlength\tabcolsep{0pt}%
    \put(0,0){\includegraphics[width=\unitlength,page=1]{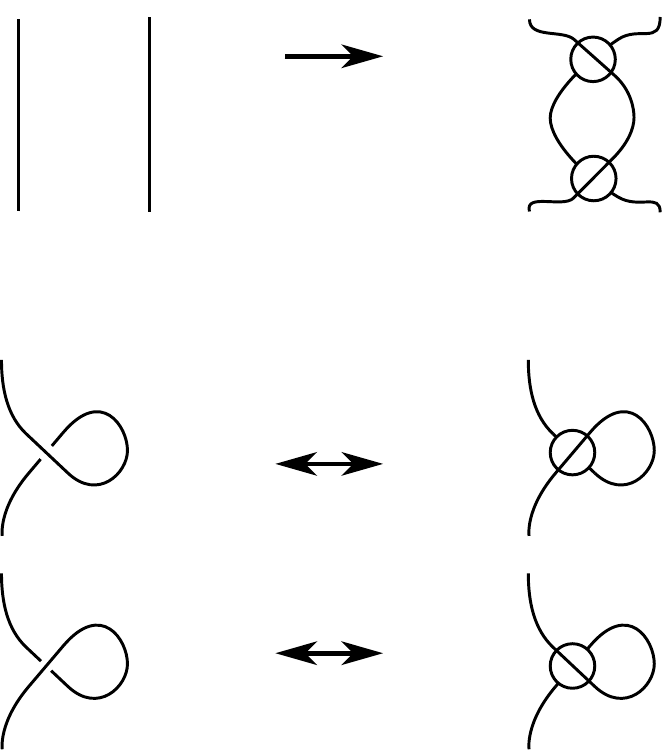}}%
    \put(0.50411577,1.09114257){\color[rgb]{0,0,0}\makebox(0,0)[t]{\lineheight{1.25}\smash{\begin{tabular}[t]{c}finger move\end{tabular}}}}%
    \put(0.50411577,0.88917145){\color[rgb]{0,0,0}\makebox(0,0)[t]{\lineheight{1.25}\smash{\begin{tabular}[t]{c}Whitney move\end{tabular}}}}%
    \put(0,0){\includegraphics[width=\unitlength,page=2]{newhtpymoves.pdf}}%
    \put(0.50411577,0.47513043){\color[rgb]{0,0,0}\makebox(0,0)[t]{\lineheight{1.25}\smash{\begin{tabular}[t]{c}cusp move\end{tabular}}}}%
    \put(0.50411577,0.19237071){\color[rgb]{0,0,0}\makebox(0,0)[t]{\lineheight{1.25}\smash{\begin{tabular}[t]{c}cusp move\end{tabular}}}}%
  \end{picture}%
\endgroup%

    \caption{The new moves describing homotopy of a surface in a 4--manifold. There are two versions of the cusp move. One involves a positive self-intersection and one involves a negative self-intersection of the described immersed surface. To describe regular homotopy we only need finger and Whitney moves. }
    \label{fig:newhtpymoves}
\end{figure}

\begin{proof}
Say $\Sigma$ and $\Sigma'$ are homotopic and have self-intersection numbers $s$ and $s'$, respectively. By work of Hirsch \cite{hirsch} and Smale \cite{smale}, $\Sigma$ and $\Sigma'$ are regularly homotopic if and only if $s=s'$.

After performing a cusp move on $\D$, a realizing surface for the resulting diagram has self-intersection $s\pm 1$, with sign depending on the choice of cusp move. Perform $|s'-s|$ cusp moves of the appropriate sign to $\D$ to obtain a diagram $\D_2$ whose realizing surface $\Sigma_2$ has self-intersection number $s'$. Now $\Sigma_2$ and $\Sigma'$ are regularly homotopic.

We recommend the reference \cite{fq} for exposition on regular homotopy of surfaces. In brief, there exists a sequence of finger moves on $\Sigma_2$ along framed arcs $\eta_1,\ldots, \eta_n$ yielding a surface $\Sigma_3$, and a sequence of finger moves on $\Sigma'$ along framed arcs $\eta'_1,\ldots,\eta'_m$ yielding a surface $\Sigma''$, so that $\Sigma_3$ and $\Sigma''$ are ambiently isotopic. 

We isotope $\eta_1$ to lie completely in $h^{-1}(3/2)$ (which may involve isotopy of $\Sigma_2$ inducing singular band moves on its singular banded unlink diagram according to Theorem~\ref{maintheorem}) and then shrink $\eta_1$ to be short and contained in a neighborhood identical to the top left of Figure~\ref{fig:newhtpymoves}. Twist the diagram as necessary so that the framing of $\eta_1$ is untwisted. Then we perform a finger move to $\D_2$ in that neighborhood. Repeat for each $i=2,\ldots, n$, and call the resulting diagram $\D_3$. A realizing surface for $\D_3$ is isotopic to $\Sigma_3$.

Now repeat for $\Sigma'$ by performing singular band moves and finger moves to its diagram $\D'$ until obtaining a diagram $\D''$ whose realizing surface is isotopic to $\Sigma''$. Since $\Sigma''$ and $\Sigma_3$ are isotopic, by Theorem~\ref{maintheorem} it follows that $\D_3$ and $\D''$ are related by singular band moves.

We thus conclude that $\D$ can be transformed into $\D'$ by a sequence of singular band moves, cusp moves, finger moves, and Whitney moves (which are the inverses to finger moves).
\end{proof}

\begin{remark}
When performing a finger move to a singular banded unlink diagram, there are seemingly two choices (related by a local symmetry) of how to mark the new vertices. However, the choices yield diagrams related by singular band moves, as shown in Figure~\ref{fig:changefingermove}.
\end{remark}

\begin{figure}
    \centering
    %% Creator: Inkscape 1.1 (c4e8f9e, 2021-05-24), www.inkscape.org
%% Accompanies image file 'changefingermove1.pdf' (pdf, eps, ps)
\begingroup%
  \makeatletter%
  \providecommand\color[2][]{%
    \errmessage{(Inkscape) Color is used for the text in Inkscape, but the package 'color.sty' is not loaded}%
    \renewcommand\color[2][]{}%
  }%
  \providecommand\transparent[1]{%
    \errmessage{(Inkscape) Transparency is used (non-zero) for the text in Inkscape, but the package 'transparent.sty' is not loaded}%
    \renewcommand\transparent[1]{}%
  }%
  \providecommand\rotatebox[2]{#2}%
  \newcommand*\fsize{\dimexpr\f@size pt\relax}%
  \newcommand*\lineheight[1]{\fontsize{\fsize}{#1\fsize}\selectfont}%
  \ifx\svgwidth\undefined%
    \setlength{\unitlength}{362.95471672bp}%
    \ifx\svgscale\undefined%
      \relax%
    \else%
      \setlength{\unitlength}{\unitlength * \real{\svgscale}}%
    \fi%
  \else%
    \setlength{\unitlength}{\svgwidth}%
  \fi%
  \global\let\svgwidth\undefined%
  \global\let\svgscale\undefined%
  \makeatother%
  \begin{picture}(1,0.83580206)%
    \lineheight{1}%
    \setlength\tabcolsep{0pt}%
    \put(0,0){\includegraphics[width=\unitlength,page=1]{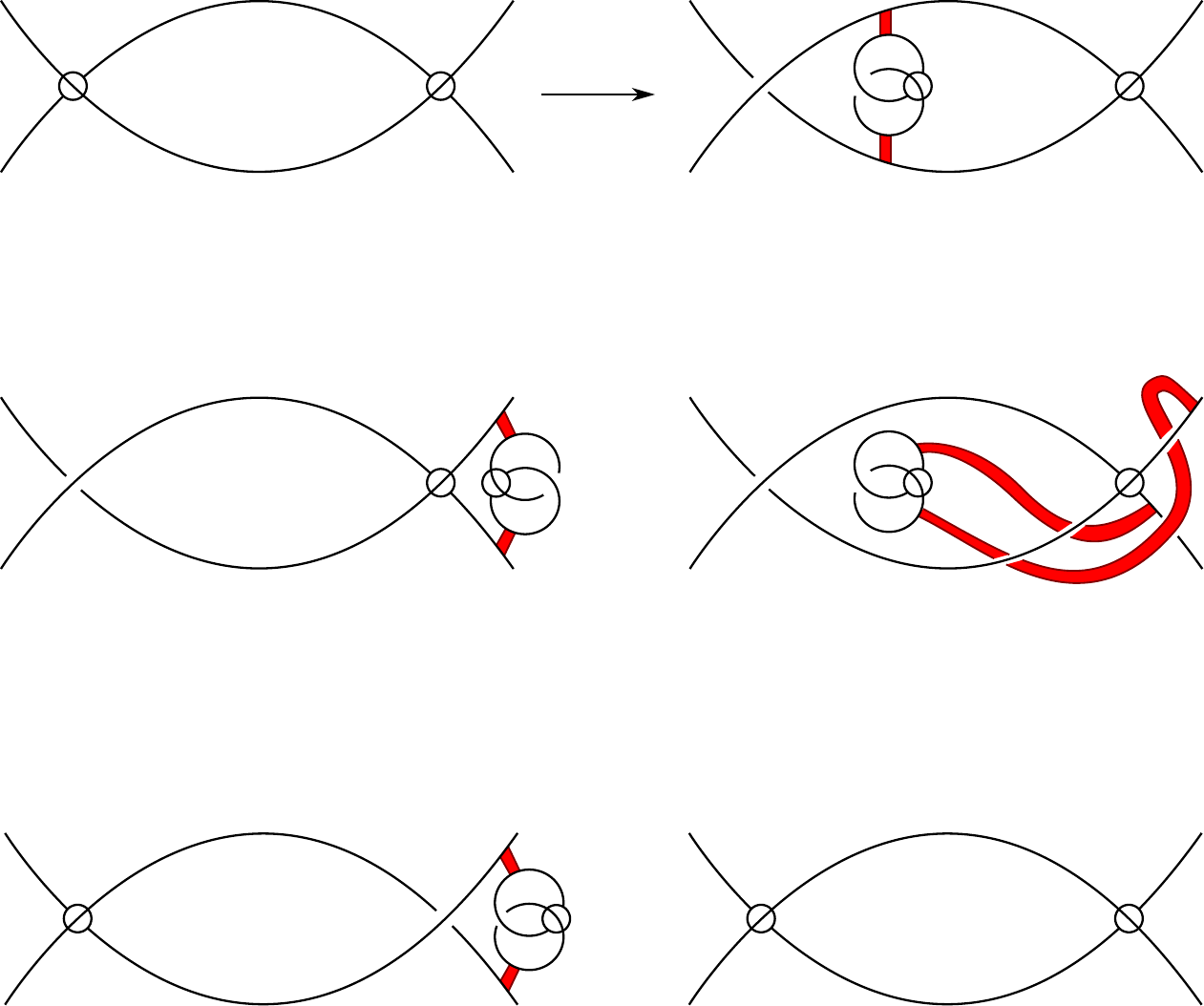}}%
    \put(0.47478702,0.77012237){\color[rgb]{0,0,0}\makebox(0,0)[lt]{\lineheight{1.25}\smash{\begin{tabular}[t]{l}$\bigstar$\end{tabular}}}}%
    \put(0,0){\includegraphics[width=\unitlength,page=2]{changefingermove1.pdf}}%
    \put(0.77509305,0.60555186){\color[rgb]{0,0,0}\makebox(0,0)[lt]{\lineheight{1.25}\smash{\begin{tabular}[t]{l}int/band passes\end{tabular}}}}%
    \put(0,0){\includegraphics[width=\unitlength,page=3]{changefingermove1.pdf}}%
    \put(0.20118407,0.28318596){\color[rgb]{0,0,0}\rotatebox{-90}{\makebox(0,0)[lt]{\lineheight{1.25}\smash{\begin{tabular}[t]{l}isotopy\end{tabular}}}}}%
    \put(0,0){\includegraphics[width=\unitlength,page=4]{changefingermove1.pdf}}%
    \put(0.51225804,0.07487688){\color[rgb]{0,0,0}\makebox(0,0)[lt]{\lineheight{1.25}\smash{\begin{tabular}[t]{l}$\bigstar$\end{tabular}}}}%
    \put(0,0){\includegraphics[width=\unitlength,page=5]{changefingermove1.pdf}}%
    \put(0.48744977,0.4320533){\color[rgb]{0,0,0}\makebox(0,0)[lt]{\lineheight{1.25}\smash{\begin{tabular}[t]{l}isotopy\end{tabular}}}}%
  \end{picture}%
\endgroup%

    \caption{There are two seemingly different finger moves, but they yield singular banded unlink diagrams that differ by singular band moves.}
   \label{fig:changefingermove}
\end{figure}

\section{Bridge trisections}\label{sec:bridgetrisections}
\subsection{Bridge trisections of embedded surfaces}\label{sec:trisectionembed}

In Section~\ref{sec:trisectionimmerse}, we prove that self-transverse immersed surfaces in 4--manifolds can be put into {\emph{bridge position}}, a notion introduced for embedded surfaces by Meier and Zupan \cite{meier2017bridge,meier2018bridge}. Meier and Zupan showed that a bridge trisection of a surface in $S^4$ (with respect to a standard trisection of $S^4$) is unique up to perturbation \cite{meier2017bridge}, using the work of Swenton \cite{swenton} and Kearton--Kurlin \cite{KK} on banded unlink diagrams in $S^4$. The authors of this paper then used a general version of this theorem in arbitrary 4--manifolds to show that bridge trisections of surfaces in any trisected manifold are unique up to perturbation. In what follows, we will apply Theorem~\ref{maintheorem} to prove an analogous uniqueness result for bridge trisections of immersed surfaces. In this section, we will review the situation where the surface is embedded.

First, we recall the definition of a trisection of a closed $4$-manifold. Similar exposition can be found in \cite{bandpaper}. We do not require much knowledge of trisections; for more detailed exposition, the interested reader may refer to~\cite{gay2016trisecting}.

\begin{definition}[\cite{gay2016trisecting}]
Let $X^4$ be a connected, closed, oriented $4$-manifold. A {\emph{$(g,k)$--trisection}} of $X^4$ is a triple $(X_1,X_2,X_3)$ where
\begin{enumerate}
\item $X_1\cup X_2\cup X_3=X^4$,
\item $X_i\cong\natural_{k_i} S^1\times B^3$,
\item $X_i\cap X_j=\boundary X_i\cap\boundary X_j\cong \natural_g S^1\times B^2$,
\item $X_1\cap X_2\cap X_3\cong \Sigma_g$,
\end{enumerate}
where $\Sigma_g$ is a closed orientable surface of genus $g$. Here, $g$ is an integer while $k=(k_1,k_2,k_3)$ is a triple of integers. If $k_1=k_2=k_3$, then the trisection is said to be {\emph{balanced}}.
\end{definition}

Briefly, a trisection is a decomposition of a $4$--manifold into three elementary pieces, analogous to a Heegaard splitting of a $3$--manifold into two elementary pieces. 

\begin{theorem}[\cite{gay2016trisecting}]
Any smooth, connected, closed, oriented 4--manifold $X^4$ admits a trisection. Moreover, any two trisections of $X^4$ are related by a stabilization operation.
\end{theorem}

Note that from the definition, $\Sigma_g$ is a Heegaard surface for $\partial X_i$, inducing the Heegaard splitting $(X_i\cap X_j, X_i\cap X_k)$. By Laudenbach and Po\'{e}naru~\cite{lp}, $X^4$ is specified by its {\emph{spine}}, $\Sigma_g\cup_{i\neq j}(X_i\cap X_j)$. Therefore, we usually describe a trisection $(X_1,X_2,X_3)$ by a {\emph{trisection diagram}} $(\Sigma_g;\alpha,\beta,\gamma)$.  Here each of $\alpha,\beta,$ and $\gamma$ consist of $g$ independent curves on $\Sigma_g$ (abusing notation to take $\Sigma_g$ as both an abstract surface and the surface $X_1\cap X_2\cap X_3$ in $X$), which bound disks in the handlebodies $X_1\cap X_2, X_2\cap X_3, X_1\cap X_3$ respectively. Given $(X_1,X_2,X_3)$, such a diagram is well-defined up to slides of $\alpha,\beta,\gamma$ and automorphisms of $\Sigma_g$.

\begin{definition}
Let $X^4$ be a $4$--manifold with trisection $\tri=(X_1,X_2,X_3)$. 
We say that an isotopy $f_t$ of $X^4$ is {\emph{$\tri$--regular}} if $f_t(X_i)=X_i$ for each $i=1,2,3$ and for all $t$.
\end{definition}

\begin{definition}
The {\emph{standard trisection of $S^4$}} is the unique $(0,0)$--trisection $(X^0_1,$ $X^0_2,$ $X^0_3)$. View $S^4=\R^4\cup\infty$, with coordinates $(x,y,r,\theta)$ on $\R^4$, where $(x,y)$ are Cartesian planar coordinates and $(r,\theta)$ are polar planar coordinates. Up to isotopy, $X^0_i=\{\theta\in[2\pi/3\cdot i, 2\pi/3\cdot(i+1)]\}\cup\infty$. Then $X^0_i\cong B^4$, $X^0_i\cap X^0_{i+1}=\{\theta=2\pi/3\cdot(i+1)\}\cup\infty\cong B^3$, and $X^0_1\cap X^0_j\cap X^0_k=\{r=0\}\cup\infty\cong S^2$.
\end{definition}

From a trisection $(X_1,X_2,X_3)$ of $X^4$, we can obtain a handle decomposition of $X^4$ in which $X_1$ contains the 0-- and 1--handles, $X_2$ is built from $(X_1\cap X_2)\times I$ by attaching the 2--handles, and $X_3$ contains the 3-- and 4--handles. The following definition encapsulates this construction. 

\begin{definition}
Let $\tri=(X_1,X_2,X_3)$ be a trisection of a 4--manifold $X^4$. Let $h:X^4\to[0,4]$ be a self-indexing Morse function. We say that $h$ is {\emph{$\tri$-compatible}} if both of the following are true.
\begin{enumerate}
    \item $X_1=h^{-1}([0,3/2])$, 
    \item $X_2\subset h^{-1}([3/2,5/2))$ contains all of the index 2 critical points of $h$,
    \item $X_1\cup X_2$ contains the descending manifolds of all index 2 critical points of $h$.
    \end{enumerate}

    Given any trisection $\tri$, there always exists a Morse function compatible with $\tri$ (see \cite{gay2016trisecting} or \cite{propr}).
\end{definition}

Meier and Zupan used trisections to give a new way of describing a surface smoothly embedded in a 4--manifold.

\begin{definition}[\cite{meier2017bridge,meier2018bridge}]\label{def:embeddedbridge}
Let $\tri=(X_1,X_2,X_3)$ be a trisection of a closed $4$--manifold $X^4$. Let $S$ be a surface embedded in $X^4$. We say that $S$ is in $(b,c)$--bridge position with respect to $\tri$ if for every $i\neq j\in\{1,2,3\}$,
\begin{enumerate}
\item $S\cap X_i$ is a disjoint union of $c_i$ boundary parallel disks,
\item $S\cap X_i\cap X_j$ is a trivial tangle of $b$ arcs.
\end{enumerate}
Here $b$ is an integer and $c=(c_1,c_2,c_3)$ is a triple of integers. Note that  $\chi(S)=\sum c_i-b$. 
\end{definition}

\begin{theorem}[\cite{meier2017bridge,meier2018bridge}]\label{thm:triunique}
Let $S$ be a surface embedded in a 4-manifold $X^4$ with trisection $\tri=(X_1,X_2,X_3)$. Then for some $c$ and $b$, $S$ can be isotoped into $(b,c)$--bridge position with respect to $\tri$. We may take $c_1=c_2=c_3$.
\end{theorem}

Because a collection of boundary parallel disks in $\natural (S^1\times B^3)$ is uniquely determined by its boundary (up to isotopy rel boundary), a surface $S$ in bridge position is determined up to isotopy by $S\cap(\cup_{i\neq j}X_i\cap X_j)$.

There is a natural perturbation of a surface in bridge position, analogous to perturbation of a knot in bridge position within a $3$-manifold. We define the simplest version of Meier--Zupan's  original perturbation operation \cite{meier2017bridge,meier2018bridge}.

\begin{definition}\label{def:embedperturb}
Let $S\subset X^4$ be a surface in $(b,c)$--bridge position with respect to $\tri=(X_1,X_2,X_3)$. Let $S'$ be the surface obtained from $S$ as in Figure~\ref{fig:embedperturb}. In words, we take a small disk $D$ contained in $S\cap X_1$ whose boundary consists of an arc $\delta_1$ in the interior of $X_1$, an arc $\delta_2$ in $X_1\cap X_2$, and an arc $\delta_3$ in $X_3\cap X_1$. We take a parallel copy $\Delta$ of $D$ pushed off $S$ away from $\delta_1$, so $\Delta$ meets $S$ in the arc $\delta_1\subset\partial\Delta$ and the remaining boundary of $\Delta$ is an arc $\delta'$ in $\partial X_1$ that meets $X_1\cap X_2\cap X_3$ transversely in one point. Using the direction from which we obtained $\Delta$ from $D$, we frame $\Delta$ and isotope $S$ along $\Delta$ to introduce two more intersection points between $S$ and $X_1\cap X_2\cap X_3$.  We call the resulting surface $S'$ and say that $S'$ is obtained from $S$ by {\emph{elementary perturbation}}.  We likewise say that $S$ is obtained from $S'$ by {\emph{elementary deperturbation}}.

We may exchange the roles of $X_1,X_2$, and $X_3$ cyclically when performing this operation, i.e.\ alternatively obtain $S'$ from this compression operation in either $X_2$ or $X_3$. We still say $S'$ is obtained from $S$ by elementary perturbation and that $S$ is obtained from $S'$ by elementary deperturbation.

\begin{figure}
    \centering
    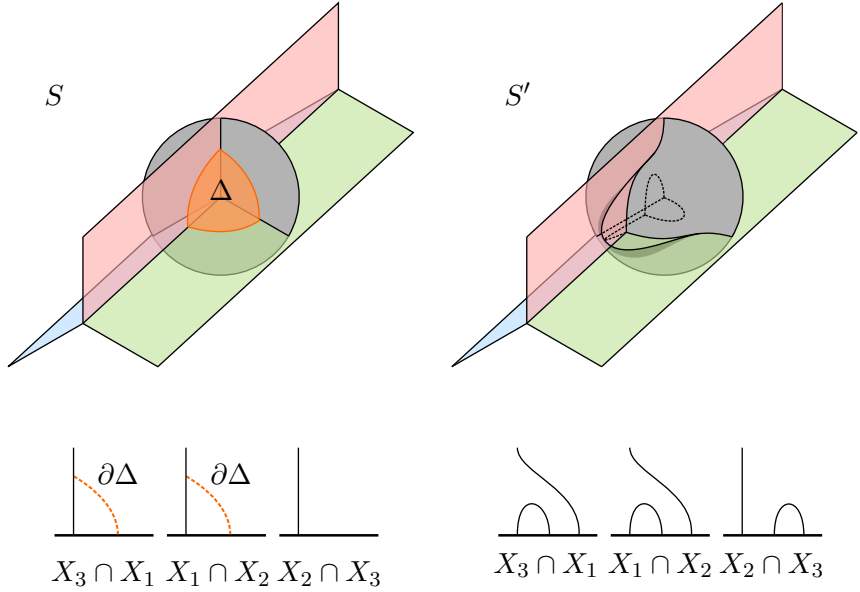
    \caption{{\bf{Left:}} A surface $S$ in $(b,c)$--bridge position with respect to a trisection $\tri$. We draw a neighborhood of an intersection of $S$ with the central surface of $\tri$. {\bf{Right:}} We perturb $S$ to obtain a surface $S'$ in $(c',b+1)$--bridge position.}
    \label{fig:embedperturb}
\end{figure}

\end{definition}

\begin{proposition}\cite[Lemma 5.2]{meier2018bridge}
Let $S$ be a surface in $(b,c)$--bridge position with respect to a trisection $\tri=(X_1,X_2,X_3)$, with $c=(c_1,c_2,c_3)$. Let $S'$ be obtained from $S$ by elementary perturbation, using a disk in $X_i$. Then $S'$ is in $(c',b+1)$--bridge position with respect to $\tri$, with $c'_j=c_j$ for $j\neq i$ and $c'_i=c_i+1$.
\end{proposition}

In previous work the authors of this paper  showed that any two bridge trisections of a surface are related by elementary perturbations.
\begin{theorem}[\cite{bandpaper}]\label{oldbridgethm}
Let $S$ and $S'$ be surfaces in bridge position with respect to a trisection $\tri$ of a 4--manifold $X^4$. Suppose $S$ is isotopic to $S'$. Then $S$ can be taken to $S'$ by a sequence of elementary perturbations and deperturbations, followed by a $\tri$--regular isotopy.
\end{theorem}

When $\tri$ is the standard trisection of $S^4$, Theorem~\ref{oldbridgethm} is a result of Meier and Zupan \cite{meier2017bridge}.

\subsection{Basic definitions for singular links and immersed surfaces}\label{sec:trisectionimmerse}

In Definition~\ref{def:embeddedbridge} of a bridge trisection of an embedded surface, we cut a 4--manifold into simple pieces so that an embedded surface is cut into systems of boundary-parallel disks. To describe immersed surfaces, we need to describe this notion with slightly different language.

\begin{definition}\label{def:trivimmersedtangle}
Let $C_1,...,C_k$ be arcs properly immersed in a $3$-manifold $M^3$. Assume that all intersections (including self-intersections) of $C_1,\ldots, C_k$ are isolated points that are not tangencies. Let $V=(\boundary M^3)\times I$ be a collar neighborhood of $\boundary M^3$ and let $h:V\to I$ be projection onto the second factor.

We say that $(C_1,\ldots, C_k)$ is a {\emph{trivial immersed tangle}} if the following are satisfied.
\begin{enumerate}
    \item Each $C_i$ is contained in $V$.
    \item All self-intersections of $C_i$ and intersections of $C_i$ with $C_j$ are contained in the interior of $M$.
    \item There is an immersed tangle $(C'_1,\ldots, C'_k)$ that is isotopic rel boundary to $(C_1,\ldots, C_k)$ so that $h|C'_i$ is Morse with a single critical point for all $i$.
    \end{enumerate}
    
\end{definition}

\begin{definition}
Let $D_1,\ldots,D_k$ be $2$--dimensional disks properly immersed in a $4$-manifold $X^4$. Assume that all intersections (including self-intersections) of $D_1$,$\ldots$, $D_k$ are isolated, transverse intersections contained in $\partial X^4$ (so $\partial (\cup D_i)$ is a singular link in $\partial X$). Let $V=\boundary X\times I$ be a neighborhood of $\boundary X$ and let $h:V\to I$ be projection onto the second factor.

We say that $(D_1,\ldots, D_k)$ is a {\emph{trivial immersed disk system}} if the following are satisfied (up to isotopy rel boundary).
\begin{enumerate}
    \item Each $D_i$ is contained in $V$.
    \item The restriction $h|D_i$ is Morse with a single critical point for all $i$.
    \end{enumerate}
\end{definition}

Trivial immersed tangles and disk systems are the immersed analogue to systems of boundary parallel embedded tangles and disks. With immersed tangles we can easily define an analogue of bridge position for singular links.

\begin{definition}
Let $L$ be a singular link in a 3--manifold $M$ with a Heegaard splitting $(H_1,H_2)$. Let $F:=H_1\cap H_2$.  

We say that $L$ is in {\emph{bridge position}} with respect to $F$ if $L\cap H_i$ is a trivial immersed tangle for $i=1,2$. See Figure~\ref{fig:immersedbridgeposition}. If $(L,\sigma)$ is a marked singular link, then we say that $(L,\sigma)$ is in {\emph{bridge position}} if $L$ is in bridge position.

\begin{figure}
    \centering
    \includegraphics[width=2in]{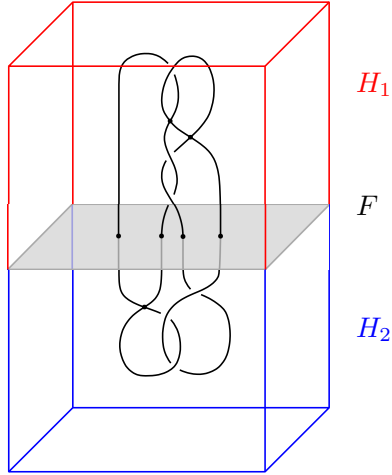}
    \caption{A singular link in bridge position.}
    \label{fig:immersedbridgeposition}
\end{figure}
\end{definition}

 We can perturb immersed tangles just as we perturb embedded tangles, but we must also account for vertices.

\begin{definition}
Let $L$ be a marked singular link in a 3--manifold $M$ with Heegaard splitting $(H_1,H_2)$. Suppose $L$ is in bridge position with respect to $\Sigma:=H_1\cup H_2$.

Let $L'$ be a marked singular link obtained from $L$ by perturbation near $\Sigma$, as in Figure~\ref{fig:singularperturb}. Note that we allow up to one vertex of $L$ to be between the original intersection of $L$ with $\Sigma$ and the newly created pair of intersections. Then we say $L'$ is obtained from $L$ by {\emph{elementary perturbation}}, and $L$ is obtained from $L'$ by {\emph{elementary deperturbation}}. 

\begin{figure}
    \centering
    \includegraphics[width=4in]{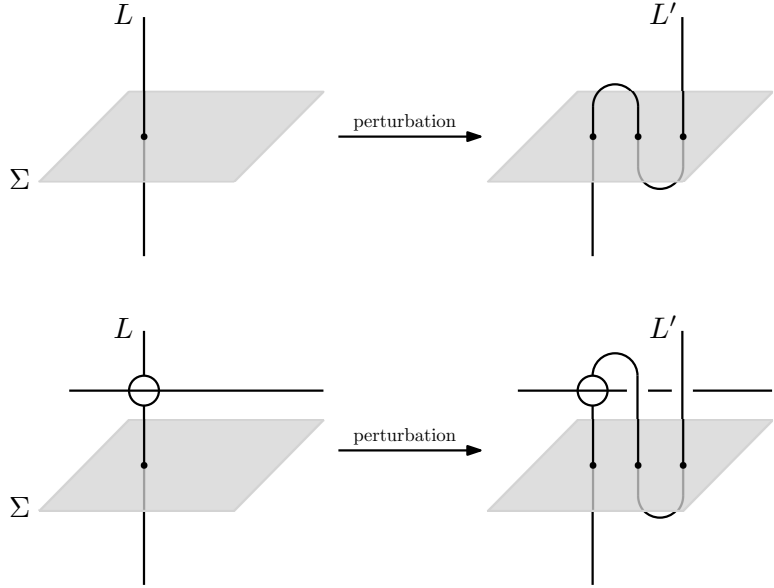}
    \caption{An elementary perturbation of a marked singular link in bridge position.}
    \label{fig:singularperturb}
\end{figure}

Let $L''$ be a marked singular link obtained from $L$ by moving a vertex in $L$ through $\Sigma$ as in the local model shown in Figure~\ref{fig:vertexperturb}. Then we say $L''$ is obtained from $L$ (and vice versa) by {\emph{vertex perturbation}}.

\begin{figure}
    \centering
    \includegraphics[width=4in]{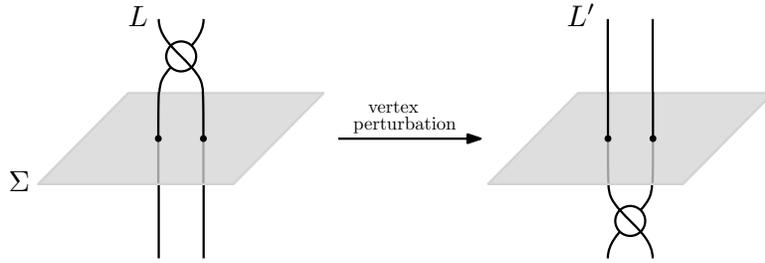}
    \caption{A vertex perturbation of a marked singular link in bridge position.}
    \label{fig:vertexperturb}
\end{figure}
\end{definition}

\begin{theorem}\label{theorem:bridgelinkisotopy}
Let $L$ and $L'$ be isotopic marked singular links in a 3--manifold $M$ with Heegaard splitting $(H_1,H_2)$. Assume $L$ and $L'$ are in bridge position with respect to $\Sigma:=H_1\cap H_2$. Then there exists a marked singular link $L''$ that can be obtained from $L$ and from $L'$ by  sequences of elementary perturbations, vertex perturbations, and isotopies fixing $\Sigma$ setwise.
\end{theorem}

\begin{proof}
When $L$ and $L'$ are nonsingular, this is a theorem of Hayashi and Shimokawa \cite{hayashi}. We will apply a version of this theorem for nonsingular banded links due to Meier and Zupan \cite{meier2017bridge,meier2018bridge} by using the following observation.  First, recall from Section~\ref{sec:markedlinks} that if $L$ is a marked singular link, then $L^+$ denotes the nonsingular link obtained by positively resolving the vertices of $L$.

\begin{observation}
There exist disjoint framed arcs $a_1,\ldots, a_n$ with endpoints on $L^+$ so that contracting $L^+$ along $a_1,\ldots, a_n$ yields $L$.
\end{observation}

Similarly let $a'_1,\ldots, a'_n$ be framed arcs with endpoints on ${L'}^+$ so that contracting ${L'}^+$ along $a'_1,\ldots, a'_n$ yields $L'$.

Now by Meier and Zupan \cite{meier2017bridge,meier2018bridge}, there exists a link $J$ that can be obtained from $L^+$ and from ${L'}^+$ by elementary perturbations and isotopies fixing $\Sigma$ setwise.  Moreover, these isotopies and perturbations may be chosen to carry $a_i$ and $a'_i$ to framed arcs $b_i$, $b'_i$ (respectively) with endpoints on $J$, so that $b_i,b'_i$  are parallel to $\Sigma$ with surface framing, and are parallel to each other (though possibly on opposite sides of $\Sigma$). 
In Meier and Zupan's construction, during this sequence of perturbations and isotopies of $L^+$ (resp. ${L'}^+$), $a_i$ (resp. $a'_i$) never intersect $\Sigma$, so these perturbations and isotopies may be achieved by perturbations and isotopies of $L$ (resp. $L'$). Let $\hat{J}$ and $\hat{J'}$ be the marked singular links obtained by contracting $J$ along $\cup b_i$ and $\cup b'_i$, respectively, and with markings induced by those of $L$ and $L'$. Then $\hat{J'}$ can be transformed into $\hat{J}$ by isotopy fixing $\Sigma$ and a vertex perturbation for each pair $a_i,a'_i$ separated in different components of $M\setminus\Sigma$. Therefore, the claim holds with $L''=\hat{J}$.
\end{proof}

\subsection{Bridge trisections of immersed surfaces}\label{sec:trisectionbridgeimmerse}

We now use the definitions from Section~\ref{sec:trisectionimmerse} to define bridge trisections of self-transverse immersed surfaces.

\begin{definition}
Let $\tri=(X_1,X_2,X_3)$ be a trisection of a closed $4$-manifold $X^4$. Let $S$ be a self-transverse immersed surface in $X^4$. We say that $S$ is in $(b,c)$--bridge position with respect to $\tri$ if for each $i\neq j\in\{1,2,3\}$,
\begin{enumerate}
    \item $S\cap X_i$ is a trivial immersed disk system of $c_i$ disks,
    \item $S\cap X_i\cap X_j$ is a trivial immersed tangle of $b$ strands. 
\end{enumerate}
Here, $b$ is a positive integer and $c=(c_1,c_2,c_3)$ is a triple of positive integers. 
\end{definition}

In Figure~\ref{fig:immersedbridgetrisection}, we give some small examples of bridge trisections of 2--spheres immersed in $S^4$.

\begin{figure}
    \includegraphics[width=5.8cm,angle=90]{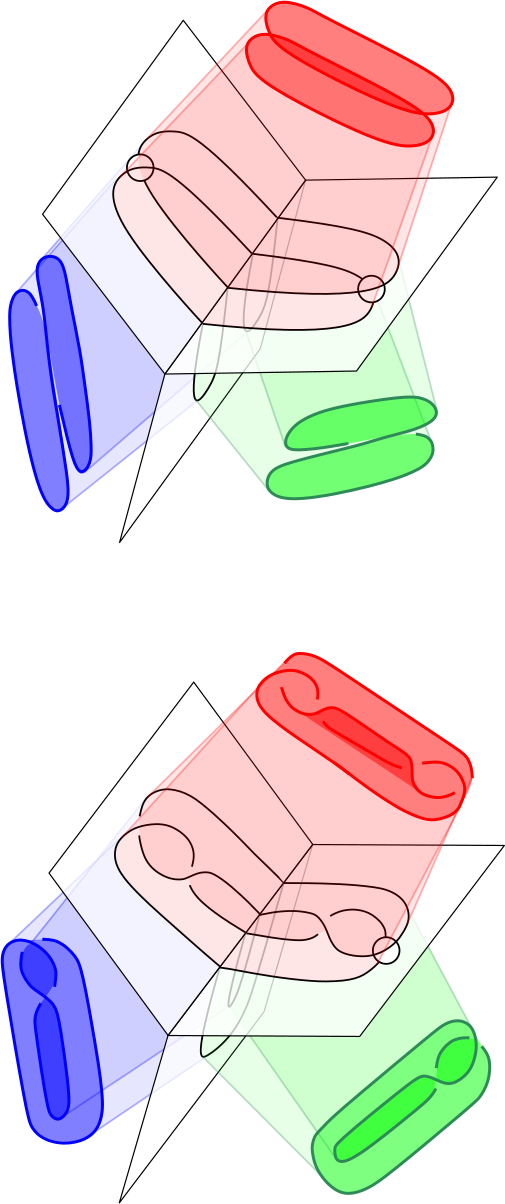}
    \caption{Two $((2,1,1),2)$--bridge trisections of immersed 2--spheres in $S^4$. {\bf{Left:}} This 2--sphere has a pair of self-intersections of opposite sign. {\bf{Right:}} This 2--sphere has a single self-intersection.}
    \label{fig:immersedbridgetrisection}
\end{figure}

There is again a natural notion of perturbing an immersed surface in $(b,c)$--bridge position. More precisely, the notion of perturbing an embedded surface in bridge position works perfectly well for an immersed surface in bridge position. We write the definition below, believing that the value of transparency outweighs the cost of redundancy.

\begin{definition}\label{def:perturbbridge}
Let $S$ be a self-transverse immersed surface in bridge position with respect to a trisection $\tri=(X_1,X_2,X_3)$. In Figure~\ref{fig:embedperturb}, we depict a small neighborhood of a point in $S\cap \Sigma$, for $\Sigma:=X_1\cap X_2\cap X_3$. Let $S'$ be the surface obtained from $S$ as in Figure~\ref{fig:embedperturb}. We say that $S'$ is obtained from $S$ by {\emph{elementary perturbation}}, and that $S$ is obtained from $S'$ by {\emph{elementary deperturbation}}.
\end{definition}

If $S$ is in bridge position with respect to a trisection $\tri=(X_1,X_2,X_3)$, then elementary perturbation and $\tri$--regular isotopy cannot move a self-intersection of $S$ from $X_i$ to $X_j$ for $i\neq j$. Thus, we introduce one new kind of perturbation for immersed surfaces in bridge position, based on the most elementary way one might move a self-intersection of $S$ from $X_i$ to $X_j$.

\begin{definition}\label{def:vertexperturb}
Let $v$ be a vertex of the singular link $S\cap X_i\cap X_{i+1}$ for some $i$ (where the indices are understood to be taken $\text{mod }3$), so that $v$ is a self-intersection of $S$. Suppose $v$ has a neighborhood as in Figure~\ref{fig:surfacevertexperturb}, so that $v$ is near $\Sigma:=X_1\cap X_2\cap X_3$. We may isotope $S$ to move $v$ into $\Sigma$ and then into either $X_{i+1}\cap X_{i+2}$ or $X_{i-1}\cap X_i$, producing a new surface $S'$ in $(b,c)$--bridge position. See Figures~\ref{fig:surfacevertexperturb} and~\ref{fig:surfacevertexperturb2}. We say that $S'$ is obtained from $S$ (and vice versa) by {\emph{vertex perturbation}}.
\end{definition}

\begin{figure}
    \centering
    \includegraphics[width=90mm]{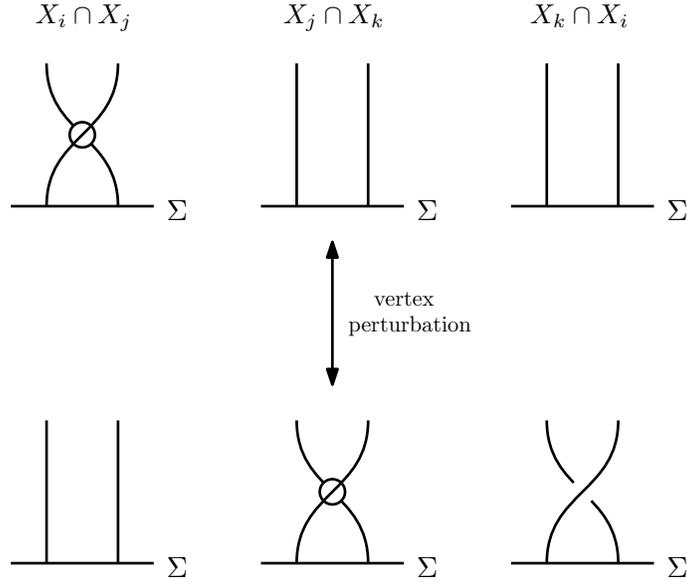}
    \caption{A vertex perturbation of a triplane diagram.}
    \label{fig:surfacevertexperturb}
\end{figure}
\begin{figure}
    \centering
    \includegraphics[width=\textwidth]{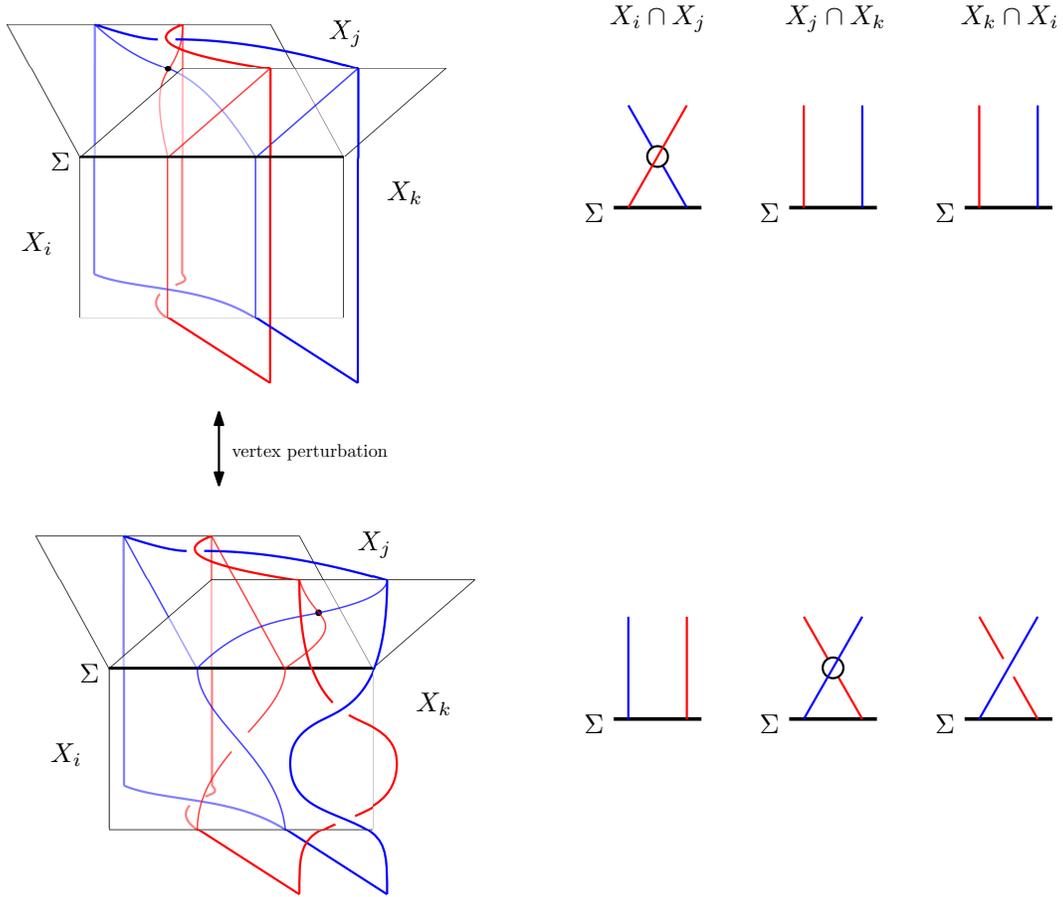}
    \caption{Pushing a self-intersection point from $X_i \cap X_j$ to $X_j \cap X_k$ during a vertex perturbation.}
    \label{fig:surfacevertexperturb2}
\end{figure}

\begin{remark}
Let $S$ be an immersed surface in $(b;c_1,c_2,c_3)$--bridge position with respect to $\tri=(X_1,X_2,X_3)$.
\begin{enumerate}[(1)]
\item If $S'$ is obtained from $S$ by elementary perturbation along a disk in $X_i$, then $S'$ is in $(b+1;c'_1,c'_2,c'_3)$--bridge position with $c'_i=c_i+1$ and $c'_j=c_j$ for $j\neq i$.
\item If $S'$ is obtained from $S$ by vertex perturbation, then $S'$ is in $(b;c_1,c_2,c_3)$--bridge position.
\end{enumerate}
\end{remark}

\begin{definition}
If a surface $S'$ in bridge position with respect to a trisection $\tri$ is obtained from a surface $S$ in bridge position with respect to $\tri$ by a sequence of elementary and vertex perturbations, then we simply say that $S'$ is obtained from $S$ by perturbation (with $\tri$ implicit). If $S'$ is obtained from $S$ by a sequence of elementary perturbations and deperturbations and vertex perturbations, then we say that $S'$ is obtained from $S$ (or ``related to $S$") by perturbation and deperturbation.
\end{definition}

\begin{theorem}\label{thm:triexist}
Let $S$ be a self-transverse immersed surface in a 4--manifold $X^4$ with trisection $\tri=(X_1,X_2,X_3)$. Then for some $c$ and $b$, $S$ can be isotoped into $(b,c)$--bridge position with respect to $\tri$. 
\end{theorem}

\begin{proof}
Let $h$ be a self-indexing Morse function of $X^4$ that is $\tri$--compatible. Let $(L,B)$ be a singular banded unlink diagram for $S$, so $L$ is a singular link in $M:=h^{-1}(3/2)$, and $B$ is a set of bands for $L$ in $M$. Let $H_1:=X_3\cap X_1$, and $H_2:=X_1\cap X_2$, so that $\Sigma:=H_1\cap H_2$ is a Heegaard surface for $M$.

By dimensionality, we may isotope $L,B$ to be contained in $\Sigma\times[-1,1]\subset M$ (i.e.\ we isotope $L$ and $B$ to avoid a 1--skeleton of $H_1$ and $H_2$), with $\Sigma\times[-1,0]\subset H_1,\Sigma\times[0,1]\subset H_2$. Isotope $L$ so that the vertices of $L$ are disjoint from $\Sigma$, and so that $B$ consists of short straight bands parallel to $\Sigma$ in $H_2$ that are far from each other, as in Figure~\ref{fig:getbridgeposition} (ii). Let $\pi:\Sigma\times[0,1]\to[0,1]$ be the projection, and perform a small isotopy of $L$ so that $\pi|_L$ is Morse.  Isotope the index 0 critical points of $\pi|_L$ vertically with respect to $\pi$ to be contained in $H_1$, and the index 1 critical points of $\pi|_L$ vertically with respect to $\pi$ to be contained in $H_2$, isotoping horizontally first if necessary to avoid introducing self-intersections of $L$ or intersections of $L$ with $B$. Now $L$ is in bridge position with respect to $\Sigma$. Perturb $L$ near again near $\boundary B$ as in Figure~\ref{fig:getbridgeposition} (iv), and isotope all bands in $B$ to lie in $H_2$.

By Theorem~\ref{maintheorem}, $S$ is isotopic to $S':=\Sigma(L,B)$. We investigate the intersections of $S'$ with the pieces of $\tri$.
\begin{enumerate}
\item $S'\cap X_1=S'\cap h^{-1}(3/2)$ consists of the minimum disks of $S'$. All self-intersections of $S'$ are contained in $\boundary X_1$.
\item $S'\cap X_2$ contains the index 1 critical points of $h|_{S'}$. This surface is built from the singular tangle $L\cap H_2$ by extending vertically and then attaching bands according to $B$. By construction, these bandings are trivial and the components of $S'\cap X_2$ are boundary-parallel away from the intersections.
\item $S'\cap X_3$ contains the maximum disks of $h|_{S'}$. In particular, $(X_3,S'\cap X_3)$ can be strongly deformation retracted to $(h^{-1}([5/2,4]),S'\cap h^{-1}([5/2,4]))$.
\item $S'\cap X_1\cap X_2=L\cap H_2$.
\item $S'\cap X_2\cap X_3$ is equivalent to the tangle obtained from $L^+\cap H_2$ by surgery on $B$.
\item $S'\cap X_3\cap X_1=\overline{L\cap H_1}$. Note the reversed orientation; this is because $H_1$ is oriented as being in the boundary of $X_1$, but $X_3\cap X_1$ is oriented as the boundary of $X_3$.
\end{enumerate}

We conclude that $S'$ is in $(b,c)$--bridge position with respect to $\tri$ for some $b,c$.
\end{proof}

\begin{figure}
    \centering
    \includegraphics[width=.95\textwidth]{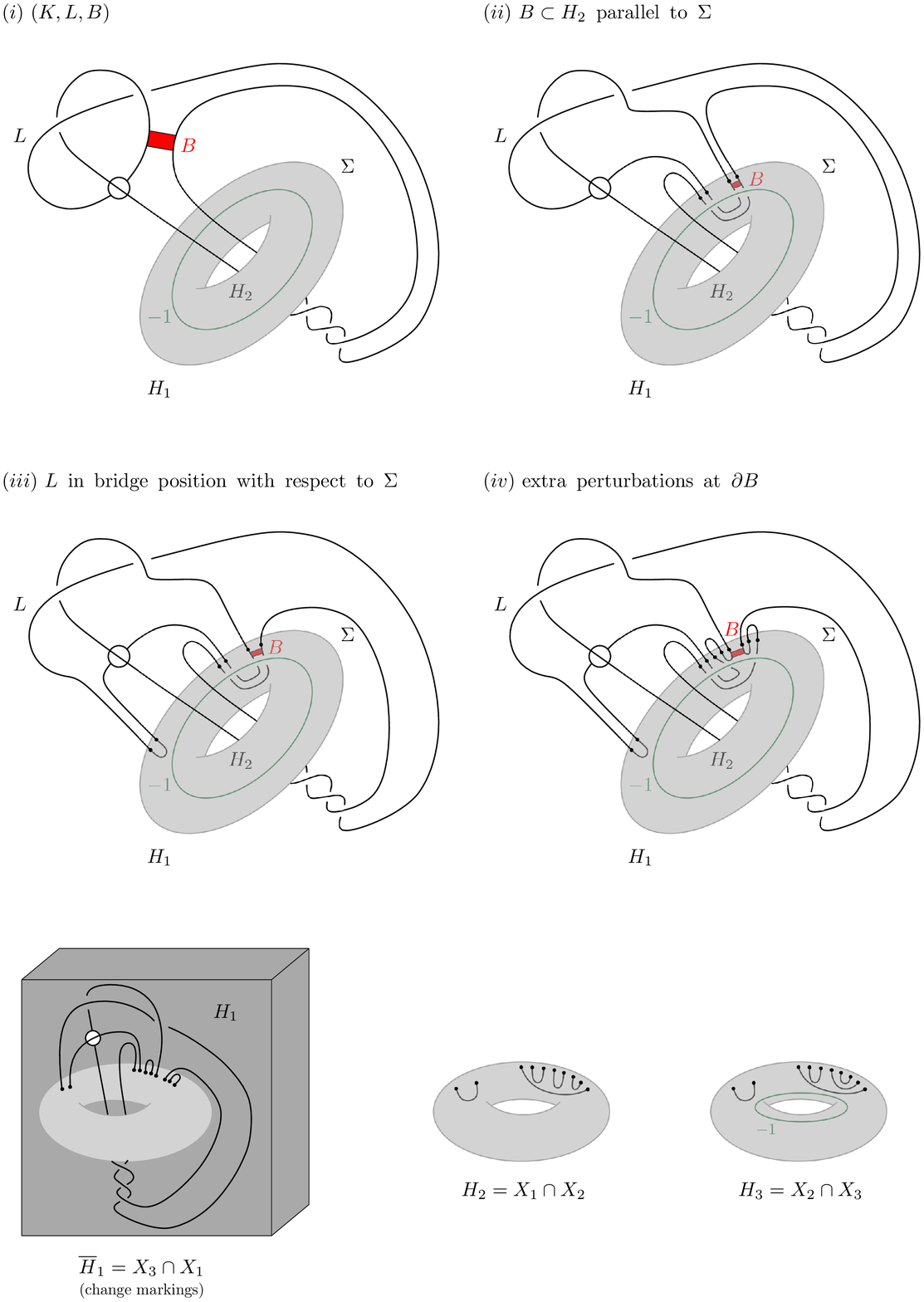}
    \caption{We illustrate how a surface that realizes a banded unlink diagram $(\mathcal{K},L,B)$ may be isotoped to lie in bridge position. See the proof of Theorem~\ref{thm:triexist}.}
    \label{fig:getbridgeposition}
\end{figure}

\subsection{Bridge splittings of singular banded links}\label{sec:bridgesingularlink}

The proof of Theorem~\ref{thm:triexist} motivates the following definition.

\begin{definition}
Let $L$ be a singular link in a 3--manifold $M$, and let $B=b_1,\ldots, b_n$ be a set of bands for $L$. Let $F$ be a Heegaard surface for $M$. We say that the singular banded link $(L,B)$ is in {\emph{bridge position}} with respect to $F$ if $L$ is in bridge position with respect to $F$, and each band $b_i$ is contained in a 3-ball $U_i$ as in Figure~\ref{fig:trivband}, with $U_i\cap U_j=\emptyset$ for $i\neq j$. 
\end{definition}

\begin{figure}
    \centering
    %% Creator: Inkscape 1.0.2-2 (e86c870879, 2021-01-15), www.inkscape.org
%% Accompanies image file 'trivband.pdf' (pdf, eps, ps)
\begingroup%
  \makeatletter%
  \providecommand\color[2][]{%
    \errmessage{(Inkscape) Color is used for the text in Inkscape, but the package 'color.sty' is not loaded}%
    \renewcommand\color[2][]{}%
  }%
  \providecommand\transparent[1]{%
    \errmessage{(Inkscape) Transparency is used (non-zero) for the text in Inkscape, but the package 'transparent.sty' is not loaded}%
    \renewcommand\transparent[1]{}%
  }%
  \providecommand\rotatebox[2]{#2}%
  \newcommand*\fsize{\dimexpr\f@size pt\relax}%
  \newcommand*\lineheight[1]{\fontsize{\fsize}{#1\fsize}\selectfont}%
  \ifx\svgwidth\undefined%
    \setlength{\unitlength}{179.55461037bp}%
    \ifx\svgscale\undefined%
      \relax%
    \else%
      \setlength{\unitlength}{\unitlength * \real{\svgscale}}%
    \fi%
  \else%
    \setlength{\unitlength}{\svgwidth}%
  \fi%
  \global\let\svgwidth\undefined%
  \global\let\svgscale\undefined%
  \makeatother%
  \begin{picture}(1,0.32344071)%
    \lineheight{1}%
    \setlength\tabcolsep{0pt}%
    \put(0,0){\includegraphics[width=\unitlength,page=1]{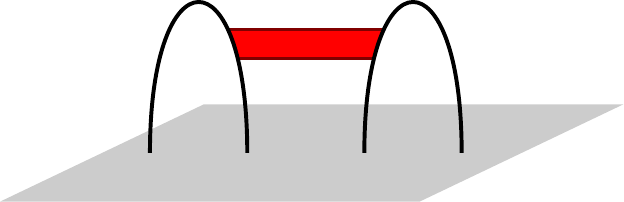}}%
    \put(-0.03850079,0.05408336){\color[rgb]{0,0,0}\makebox(0,0)[lt]{\lineheight{1.25}\smash{\begin{tabular}[t]{l}$F$\end{tabular}}}}%
  \end{picture}%
\endgroup%

    \caption{If a singular banded link $(L,B)$ is in bridge position with respect to a Heegaard surface $F$, then every band in $B$ has a neighborhood as pictured here. That is, every band in $B$ has a neighborhood $U$ containing two components $C_1,C_2$ of $L\setminus F$ (on which $B$ has ends), meeting $F$ in a disk, and not meeting any other bands in $B$ or other components of $L\setminus F$. Moreover, $\overline{C_1}\cup\overline{C_2}\cup B$ may be isotoped rel $\partial(\overline{C_1}\cup\overline{C_2})$ in $U$ to lie in $F$.}
    \label{fig:trivband}
\end{figure}

The proof of Theorem~\ref{thm:triexist} can be broken down into the following two lemmas, which are useful to state directly.

\begin{lemma}\label{lemma:bridgeposition}
Let $L$ be a singular link in a 3--manifold $M$, and let $B$ be a set of bands for $L$. Fix a Heegaard surface $F$ for $M$. Then $(L,B)$ can be isotoped to lie in bridge position with respect to $F$.
\end{lemma}
\begin{lemma}\label{lemma:bridgerealization}
Let $\tri=(X_1,X_2,X_3)$ be a trisection of a 4--manifold $X^4$. Let $h$ be a $\tri$--compatible Morse function on $X^4$, and $\K$ a Kirby diagram induced by $h$ and a gradient-like vector field $\nabla h$. Then $H_1=X_3\cap X_1$ and $H_2=X_1\cap X_2$ give a Heegaard splitting $(H_1,H_2)$ for $h^{-1}(3/2)$, in which $\Sigma:=H_1\cap H_2\subset E(\K)$ is a Heegaard surface.

Suppose a banded unlink $(\K,L,B)$ is in bridge position with respect to $\Sigma$. Then a realizing surface $\Sigma(\K,L,B)$ is in bridge position with respect to $\tri$.
\end{lemma}

\begin{definition}\label{def:bridgediagram}
Let $S$ be a self-transverse immersed surface in a 4--manifold $X^4$ with trisection $\tri=(X_1,X_2,X_3)$. Assume $S$ is in $(b,c)$--bridge position. We call the triple of singular marked tangles $(T_1,T_2,T_3)=(S\cap X_1\cap X_2,S\cap X_2\cap X_3, S\cap X_3\cap X_1)$ a {\emph{bridge trisection diagram}} of $S$. The markings of each tangle should be chosen so that:
\begin{itemize}
    \item[$\circ$] In $X_i$, cross-sections of $S$ are the negative resolution of $S\cap X_i\cap X_{i+1}$.
    \item[$\circ$] In $X_i$, cross-sections of $S$ are the positive resolution of $S\cap X_{i-1}\cap X_{i}$.
\end{itemize}

Note that we choose the marking convention to be symmetric with respect to the trisection, even though in the construction of Theorem~\ref{thm:triexist}, we used a Morse function $h$ in which the pieces $X_1, X_2, X_3$ were not symmetric. If $(L,B)$ is a singular banded unlink diagram for $S$ and we follow the construction of Theorem~\ref{thm:triexist}, then we obtain a bridge trisection diagram $(T_1,T_2,T_3)$ of $S$ with:
\begin{enumerate}
\item $T_1=L\cap H_2$ with markings agreeing with those of $L$,
\item $T_2=(L\cap H_2)^+_B$,
\item $T_3=L\cap \overline{H_1}$ with markings {\emph{opposite}} those of $L$.
\end{enumerate}

We include a local example in Figure~\ref{fig:bridgeorientations}.

\begin{figure}
    \centering
    \scalebox{1.05}{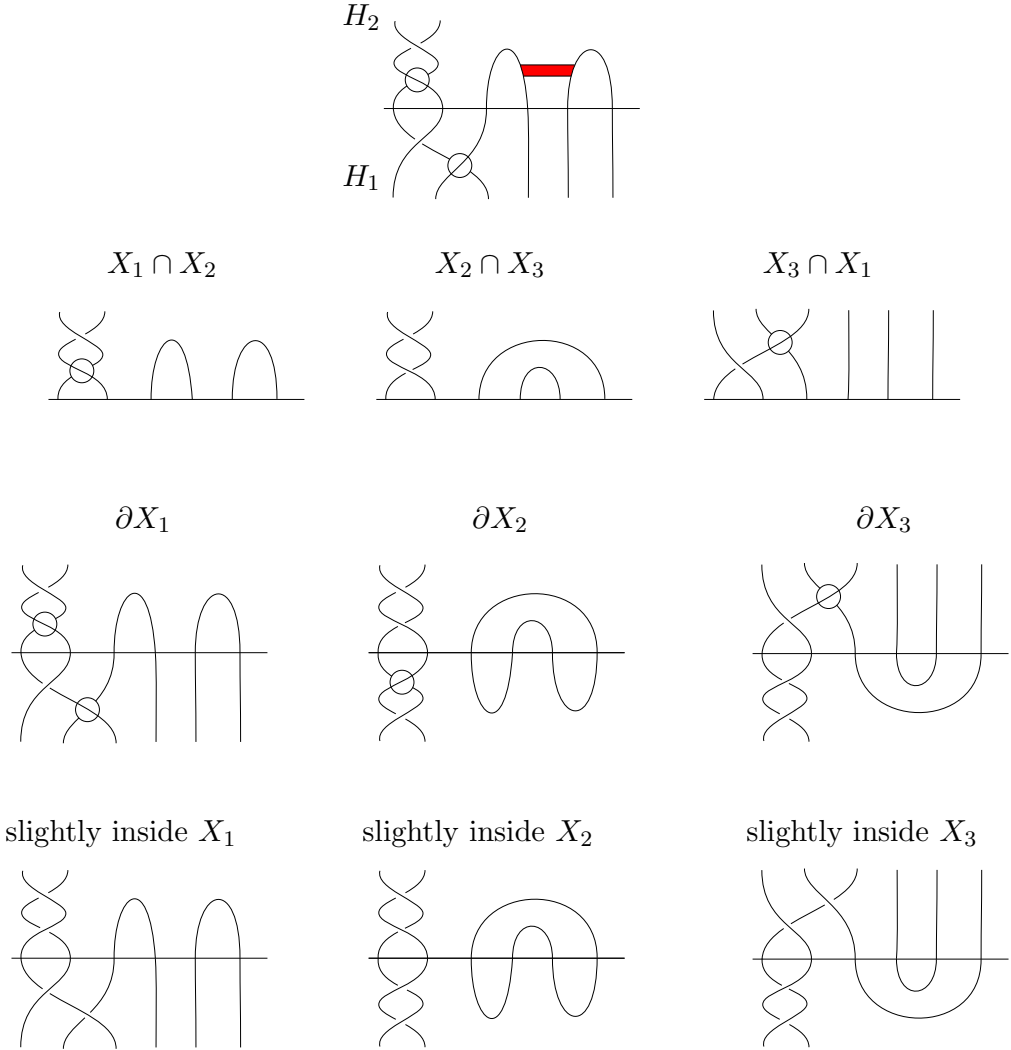}
    \caption{{\bf{Top row:}} Part of a singular banded unlink in bridge position. {\bf{Second row:}} We obtain the singular tangles $T_1, T_2, T_3$ as in Definition~\ref{def:bridgediagram}. {\bf{Third row:}} The singular links that are the intersection of the associated bridge trisected surface with $\partial X_1, \partial X_2, \partial X_3$. {\bf{Bottom row:}} We draw the resolutions of these tangles in the interiors of $X_1, X_2, X_3$. Note that vertices in $X_i\cap X_{i+1}$ are resolved negatively into $X_i$, while vertices in $X_{i-1}\cap X_i$ are resolved positively into $X_i$.}
    \label{fig:bridgeorientations}
\end{figure}

\end{definition}
From a bridge trisection diagram of $S$, we can reconstruct a surface that is ambiently isotopic to $S$ as usual. For convenience (to mirror the construction in Theorem~\ref{thm:triexist}), it is more convenient to assume all self-intersections lie in $H_1$ and $H_2$ (i.e.\ in $\partial X_1$ and not in $X_2\cap X_3$).

\begin{lemma}\label{lemma:bridgegetbandedunlink}
Let $S$ be a self-transverse immersed surface in a 4--manifold $X^4$ that is in bridge position with respect to a trisection $\tri=(X_1,X_2,X_3)$. Assume that $S$ has no self-intersections in $X_2\cap X_3$.

Let $h$ be a $\tri$--compatible Morse function on $X^4$, and fix a gradient-like vector field $\nabla h$ inducing a Kirby diagram $\K$. Then there is a singular banded unlink diagram $(\K,L,B)$ so that $(L,B)$ is in bridge position with respect to the Heegaard surface $\Sigma=X_1\cap X_2\cap X_3\subset E(\K)$, and $S$ is $\tri$--regularly isotopic to the surface $\Sigma(\K,L,B)$.

\end{lemma}

\begin{proof}
Isotope $S$ to be 0--standard (with respect to $h, \nabla h$). Since $S$ is in bridge position, we may take this isotopy to be $\tri$--regular.

Let $L:=S\cap h^{-1}(3/2)$. (Recall $h^{-1}(3/2)=\partial X_1=H_1\cup H_2$, where $H_1=X_3\cap X_1$ and $H_2=X_1\cap X_2$). Then $L$ is a singular link whose vertices are either in $H_1$ or $H_2$. Mark $L$ so that the negative resolutions of the vertices in $H_1$ and the positive resolutions of the vertices in $H_2$ correspond to the resolutions of the immersed disk system $S\cap X_1$. Then $L$ is a marked singular link and $L^-$ is an unlink.

Now $S\cap X_2$ is a trivial immersed disk system with all intersections in $X_1\cap X_2$. Let $\widetilde{X}_2$ be obtained from $X_2$ by deleting a small neighborhood of each intersection, so that $\widetilde{X}_2$ is still a 4--dimensional 1--handlebody, but $S\cap\widetilde{X}_2$ is a trivial embedded disk system $D$. Let $\widetilde{H}_2$ denote the closure of $(\partial\widetilde{X}_2)\setminus(X_2\cap X_3)$.

Now $D$ is a collection of boundary parallel disks in $\widetilde{X}_2$, and $\partial\widetilde{X}_2$ has a Heegaard splittings ($\widetilde{H}_2,X_2\cap X_3)$, which in respect to  $\partial D$ is in bridge position. We proceed as in \cite[Lemma 3.3]{meier2017bridge}: for each component $D_i$ of $D$, let $a_i$ be one component of $\overline{\partial D\setminus(X_2\cap X_3)}$. Then let $y_i$ be an arc in $\partial\widetilde{X}_2$ parallel to $\partial D_i\setminus a_i$ with endpoints on $\partial D$, with framing induced by $D_i$. Isotope $y_i$ in $\partial{X}_2$ into the Heegaard surface for $\partial\widetilde{X}_2$, twisting $y_i$ around $\partial D$ as necessary so that the framing of $y_i$ agrees with the framing induced by the Heegaard surface. Finally, project $y_i$ to $\partial X_2$, push slightly into $H_2$, and thicken (according to the framing of $y_i$) to obtain a band attached to $S\cap H_2$  (i.e.\ a band $b_i$ in $h^{-1}(3/2)$ attached to $L$, with $b_i$ in $H_2$ parallel to $H_1\cap H_2$). 

Repeat this for every component of $D$ to obtain a collection $B$ of bands for $L$. By construction, $L^+_B$ is an unlink when projected to $h^{-1}(5/2)$.  More specifically, in $\K$ the link $L^+_B$ (projected to $h^{-1}(5/2)$) can be made to agree with the link $S\cap h^{-1}(5/2)$  via an isotopy rel boundary in $H_2$ and slides in $H_2$ over curves in $\K$.

We conclude immediately that $(\K,L,B)$ is a singular banded unlink for some surface $S':=\Sigma(\K,L,B)$ in $X$. Moreover, $S'$ is in bridge position with respect to $\tri$, and by the above paragraph can be $\tri$--regularly isotoped so that it agrees with $S$ in $X_i\cap X_j$ for all $i\neq j$. Therefore, $S$ and $S'$ are $\tri$--regularly isotopic.
\end{proof}

\begin{remark}
Fix a trisection $\tri=(X_1,X_2,X_3)$ of $X$, a $\tri$--compatible Morse function $h$, and a gradient-like vector field $\nabla h$, so that $(h,\nabla h)$ induce a Kirby diagram $\K$ of $X$ in which $\Sigma:=X_1\cap X_2\cap X_3$ is a Heegaard surface. Definition~\ref{def:bridgediagram} and Lemma~\ref{lemma:bridgegetbandedunlink} can be combined to form the following equivalence.

\vspace{.1in}

\centering{
\begin{tabular}{c}
    $\displaystyle\frac{\{\text{bridge trisections w.r.t $\tri$ with no self-intersections in $X_2\cap X_3$}\}}{\text{$\tri$--regular isotopy}}$\\\\
   $\displaystyle\updownarrow$\\\\
   $\displaystyle\frac{\{\text{SBUDs in $\K$ in bridge position w.r.t $\Sigma$}\}}{\text{singular band moves preserving $\Sigma$ setwise}}$
    \end{tabular}}

\vspace{.1in}

\end{remark}

The restriction of bridge position to not include self-intersections in $X_2\cap X_3$ is merely a diagrammatic convenience from the viewpoint  of singular banded unlinks diagrams (SBUDs).

\begin{lemma}\label{nointsinx2x3}
Let $S$ be in bridge position with respect to $\tri=(X_1,X_2,X_3)$. There exists a sequence of perturbations of $S$ yielding a surface $S'$ in bridge position so that $S'$ has no self-intersections in $X_2\cap X_3$.
\end{lemma}

To inductively prove Lemma~\ref{nointsinx2x3}, it is clearly sufficient to prove the following proposition.
\begin{proposition}
Suppose there are $n>0$ self-intersections of $S$ in $X_2\cap X_3$. Then after $\tri$--regular isotopy of $S$, there is a surface $S'$ obtained from vertex perturbation on $S$ so that $S'$ has $n-1$ self-intersections in $X_2\cap X_3$.
\end{proposition}
\begin{proof}
Following from Definition~\ref{def:trivimmersedtangle} of a trivial immersed tangle, some $\tri$--regular isotopy of $S$ can arrange for the tangle $T=S\cap X_2\cap X_3$ to lie inside a collar neighborhood $\Sigma\times I\subset X_2\cap X_3$ of $\partial (X_2\cap X_3)=\Sigma$, so that projection to the $I$ factor is Morse on $T$ with one maximum on each arc component. Further isotope so that the self-intersections of $S$ in $\Sigma\times I$ lie at different values of the $I$ factor. In particular, one self-intersection $c$ lies strictly closest to $\Sigma$. Then by $\tri$--regular isotopy of $S$ near $\Sigma$ (sometimes called ``mutual braid transposition" when performed diagramatically), we can arrange for $c$ to have a neighborhood as in Figure~\ref{fig:surfacevertexperturb}, and thus apply a vertex perturbation to $S$ to obtain a surface $S'$ with one fewer self-intersection in $X_2\cap X_3$.

\end{proof}

\subsection{Uniqueness of bridge trisections of immersed surfaces}\label{sec:trisectionunique}
Perturbation of bridge trisections is conveniently very similar to perturbation of a banded link in bridge position. When perturbing a banded link $(L,B)$ with respect to a Heegaard surface $\Sigma$, we allow at most one band and one vertex to be between the intersection of $L$ and $\Sigma$ at which the perturbation is based and the two newly introduced intersections. See Figure~\ref{fig:perturbsingularband}.

\begin{lemma}\label{lemma:perturb}
Let $\tri=(X_1,X_2,X_3)$ be a trisection of a 4--manifold $X^4$. Let $h$ be a $\tri$--compatible Morse function on $X^4$, and $\K$ a Kirby diagram induced by $h$. Let $H_1:=X_3\cap X_1$ and $H_2:=X_1\cap X_2$ give the usual Heegaard splitting $(H_1,H_2)$ for $\K$, in which $\Sigma:=H_1\cap H_2$ is the Heegaard surface.

Suppose a singular banded unlink diagram $(\K,L,B)$ is in bridge position with respect to $\Sigma$. Let $(\K,L',B')$ be obtained from $(\K,L,B)$ by perturbation near $L\cap\Sigma$. Then $\Sigma(\K,L',B')$ can be obtained from $\Sigma(\K,L,B)$ by perturbation followed by $\tri$--regular isotopy.
\end{lemma}

\begin{figure}
    \centering
    \includegraphics[width=45mm,angle=90]{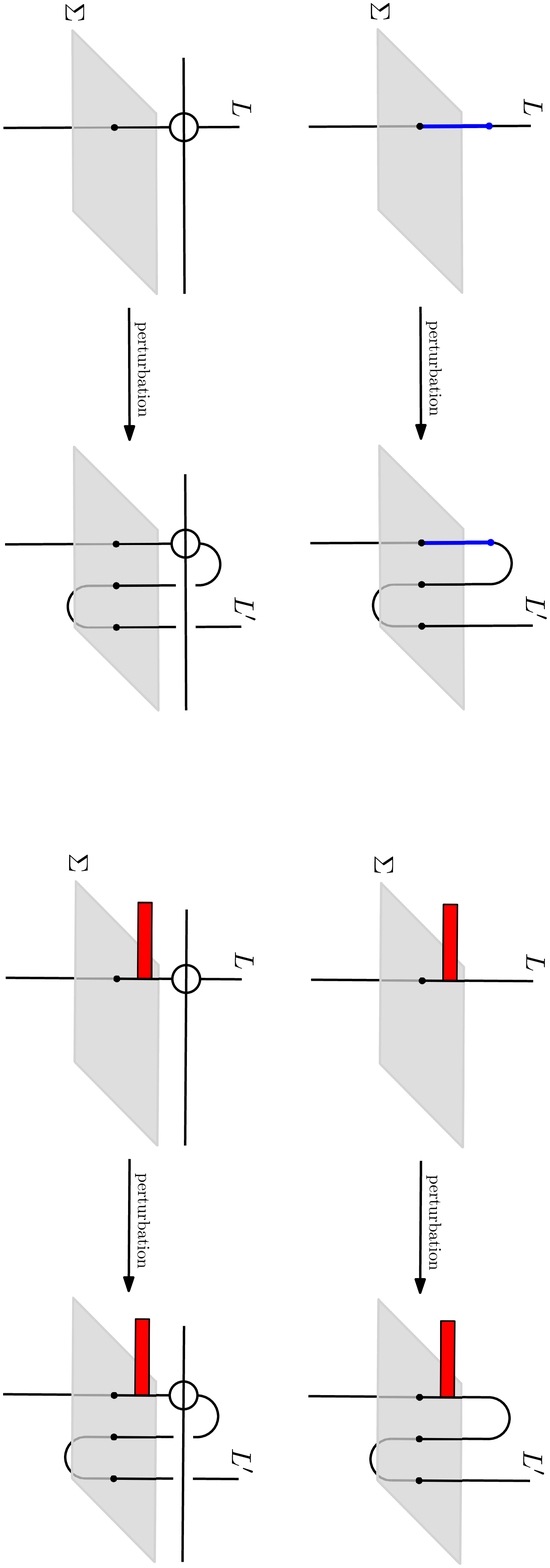}
    \caption{When performing a perturbation on the diagram in the top left we allow the blue arc to intersect at most one band and one vertex, as shown in the other three diagrams.}
    \label{fig:perturbsingularband}
\end{figure}
\begin{proof}
See Figure~\ref{fig:linkperturbvssurfaceperturb} (top).
\end{proof}

\begin{lemma}\label{lemma:perturb2}
Let $\tri=(X_1,X_2,X_3)$ be a trisection of a 4--manifold $X^4$. Let $h$ be a $\tri$--compatible Morse function on $X^4$, and $\K$ a Kirby diagram induced by $h$. Let $H_1:=X_3\cap X_1$ and $H_2:=X_1\cap X_2$ give the usual Heegaard splitting $(H_1,H_2)$ for $\K$, in which $\Sigma:=H_1\cap H_2$ is the Heegaard surface.

Suppose a singular banded unlink diagram $(\K,L,B)$ is in bridge position with respect to $\Sigma$ and that $v$ is a vertex of $L$ that is close to $\Sigma$ as in Figure~\ref{fig:vertexperturb}. Let $(\K,L',B')$ be obtained from $(\K,L,B)$ by isotoping $v$ through $\Sigma$. (We call this a {\emph{vertex perturbation of the banded link $(L,B)$}}). Then $\Sigma(\K,L',B')$ can be obtained from $\Sigma(\K,L,B)$ by one vertex perturbation followed by $\tri$--regular isotopy.
\end{lemma}

\begin{proof}
See Figure~\ref{fig:linkperturbvssurfaceperturb} (bottom).
\end{proof}

\begin{figure}
    \centering
    \includegraphics[width=0.8\textwidth]{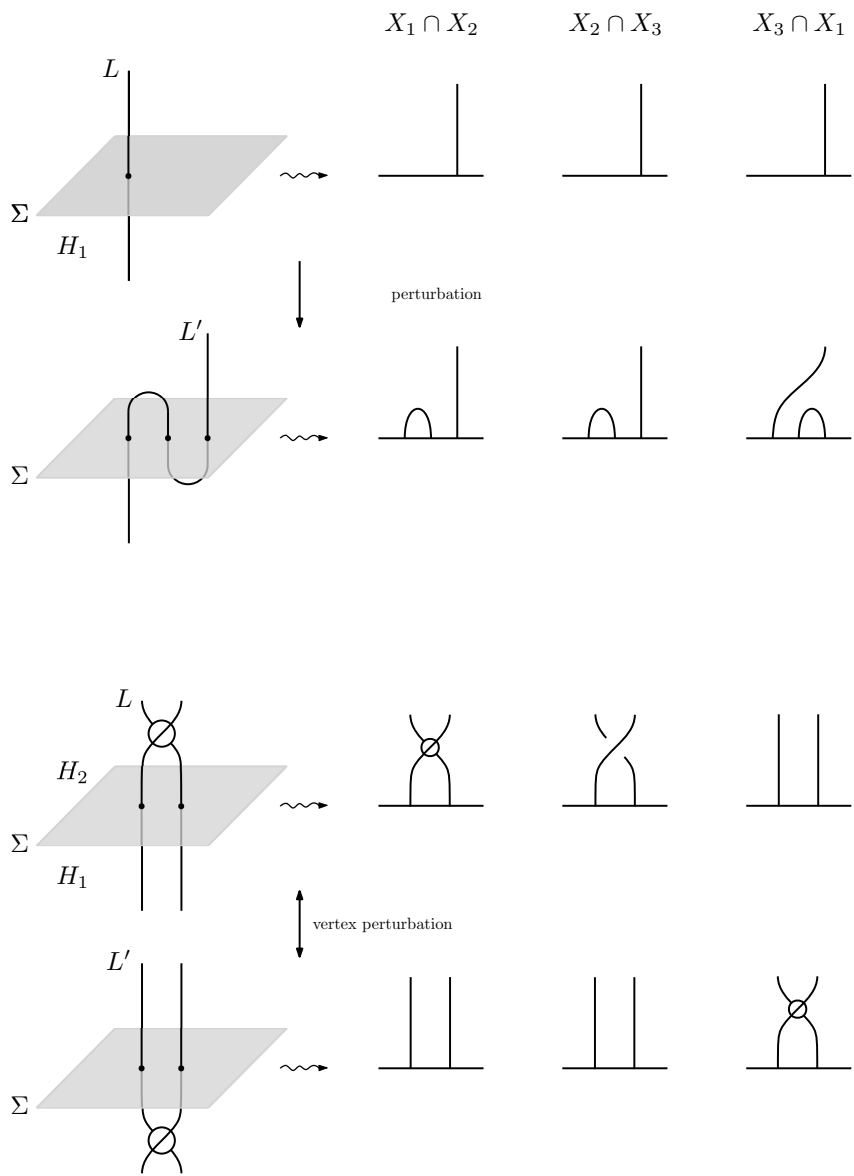}
    \caption{Perturbation of a singular banded unlink $(L,B)$ in bridge position induces perturbation of $\Sigma(L,B)$. {\bf{Top:}} Elementary perturbation. {\bf{Bottom:}} Vertex perturbation.}
    \label{fig:linkperturbvssurfaceperturb}
\end{figure}

The following uniqueness of bridge splittings of banded links motivates the uniqueness of bridge trisections.

\begin{theorem}\label{theorem:bandlinkisotopy}
Let $(L,B)$ and $(L',B')$ be isotopic banded singular marked links in a 3--manifold $M$ that has a Heegaard splitting $(H_1,H_2)$. Assume that both $(L,B)$ and $(L',B')$ are in bridge position with respect to $\Sigma:=H_1\cap H_2$, and that $B$ and $B'$ are both contained in $H_2$. Then there exists a banded singular marked link $(L'',B'')$ in bridge position with respect to $\Sigma$ that can be obtained from both $(L,B)$ and $(L',B')$ by  sequences of elementary perturbations, vertex perturbations, and isotopies that fix $\Sigma$ setwise.
\end{theorem}

Theorem \ref{theorem:bandlinkisotopy} is similar to a theorem for nonsingular banded links due to Meier and Zupan \cite{meier2017bridge,meier2018bridge}.

\begin{remark}
Meier and Zupan study banded links by viewing each band as a framed arc with endpoints on a link. They give moves to perturb a link in order to make these framed arcs parallel to a bridge surface with correct framing. In the setting of singular banded links, we are able to use their proof by viewing both self-intersections and bands as framed arcs, applying the theorem, and then contracting the self-intersection arcs to yield a singular link in bridge position.
\end{remark}

\begin{proof}
As in Theorem~\ref{theorem:bridgelinkisotopy}, there exist disjoint framed arcs $a_1,\ldots, a_n$ with endpoints on $L^+$ so that contracting $L^+$ along $a_1,\ldots, a_n$ yields $L$. Similarly, there exist framed arcs $a'_1,\ldots, a'_n$ with endpoints on ${L'}^+$ so that contracting ${L'}^+$ along $a'_1,\ldots, a'_n$ yields $L'$.

Now by Meier and Zupan \cite{meier2017bridge,meier2018bridge}, there exists a link $J$ that can be obtained from $L^+$ and from ${L'}^+$ by elementary perturbations and isotopies fixing $\Sigma$ setwise. Moreover, these isotopies and perturbations carry $a_i$ and $a'_i$ to framed arcs $b_i$ and $b'_i$ (respectively) with endpoints on $J$, so that $b_i,b'_i$ are each parallel to $\Sigma$ with surface framing, and either agree or could be isotoped to agree if the endpoints of $b'_i$ were allowed to pass through $\Sigma$ (i.e.\ $b_i$ and $b'_i$ are parallel and both lie close to $\Sigma$, but potentially on opposite sides). Moreover, during the perturbations and isotopies of $L^+$ (resp. ${L'}^+$), $a_i$ (resp. $a'_i$) never intersect $\Sigma$, so these perturbations and isotopies may be achieved by perturbations and isotopies of $L$ (resp. $L'$).

Meier--Zupan's proof allows us to not only control the framed arcs $a_i$, $a'_i$, but also the framed arcs that are the cores of the bands $B$ and $B'$. That is, by perhaps perturbing $J$ even further, we may also assume that $B$ and $B'$ are taken to bands $B_J$, and $B'_J$ whose $i$-th bands either agree or are parallel and close to $\Sigma$ but on opposite sides, and that $(J,B_J)$, $(J,B'_J)$ are both in bridge position. Let $\hat{J}$ and $\hat{J'}$ be the marked singular links obtained by contracting $J$ along $b_i$ and $b'_i$, respectively, and with markings induced by $L$ and $L'$. Then $\hat{J'}$ can be transformed into $\hat{J}$ by isotopy fixing $\Sigma$ and a vertex perturbation for each pair $a_i,a'_i$ in different components of $M\setminus\Sigma$. Therefore, the claim holds with $L''=\hat{J}$, and  $B''=B_J$.
\end{proof}

\begin{corollary}\label{cor:isotopicdiagrams}
If $\D=(L,B)$ and $\D'=(L',B')$ are isotopic banded unlink diagrams that are each in bridge position with respect to $\Sigma$, then $S:=\Sigma(\D)$ and $S':=\Sigma(\D')$ are related by elementary perturbation and deperturbation, vertex perturbation, and $\tri$--regular isotopy. 
\end{corollary}
\begin{proof}
By Theorem~\ref{theorem:bandlinkisotopy}, $\D$ and $\D'$ are related by a sequence of elementary perturbations and deperturbations, vertex perturbations, and isotopies fixing $\Sigma$ setwise. It is therefore sufficient to show that the claim is true if $\D'$ is obtained from $\D$ by a single one of these moves. We have already shown the claim to be true when $\D'$ is obtained from $\D$ by either a perturbation/deperturbation (Lemma~\ref{lemma:perturb}), or a vertex perturbation (Lemma~\ref{lemma:perturb2}). So suppose that $\D'$ is obtained from $\D$ by an isotopy $f_t$ of $M$ that fixes $\Sigma$ setwise.

The surface $\Sigma_{3/2}:=\Sigma$ is a separating surface in $M=h^{-1}(3/2)$. For every $t\in[0,4]$, there is a separating surface $\Sigma_t$ in $h^{-1}(t)$ that is vertically above or below $\Sigma$.
Then $f_t$ can be extended to a horizontal isotopy of the whole 4--manifold $X^4$ that fixes every $\Sigma_t$ horizontally. Since all index 2 critical points of $h$ are contained in one component of $X^4\setminus\cup_t\Sigma_t$, this isotopy can be chosen to take $S$ to $S'$. Since this isotopy is horizontal, it fixes $X_1=h^{-1}([0,3/2])$ and $X_2\cup X_3=h^{-1}([3/2,4])$ setwise. Since this isotopy fixes $X_2\cap X_3=\cup_{[3/2,4]}\Sigma_t$ setwise, it also fixes $X_2$ and $X_3$ setwise. Therefore, this is a $\tri$--regular isotopy.
\end{proof}

 The main theorem of this section is that bridge position and hence bridge trisection diagrams are essentially unique. The proof uses Theorem~\ref{maintheorem}.

\begin{theorem}\label{bridgethm}
Let $S$ and $S'$ be self-transverse immersed surfaces in bridge position with respect to a trisection $\tri=(X_1,X_2,X_3)$ of a closed $4$--manifold $X^4$. Suppose $S$ is ambiently isotopic to $S'$. Then $S$ can be taken to $S'$ by a sequence of elementary perturbations and deperturbations, vertex perturbations, and $\tri$--regular isotopy.
\end{theorem}

\begin{proof}

Let $h:X\to[0,4]$ be a $\tri$-compatible Morse function on $X^4$. Let $\K$ be a Kirby diagram for $X$ induced by $h$ and a fixed choice of $\nabla h$. As usual, we view $\Sigma:=X_1\cap X_2\cap X_3$ as a Heegaard surface for the ambient space of $\K$, with the dotted circles of $\K$ contained in one handlebody $H_1$ of this splitting and the 2--handle circles of $\K$ contained in the other handlebody $H_2$.

By Lemma~\ref{nointsinx2x3}, we may $\tri$--regularly isotope and perturb $S$ and $S'$ so that they do not include self-intersections in $X_2\cap X_3$. Then by Lemma~\ref{lemma:bridgegetbandedunlink}, there are banded unlink diagrams $\D:=(\K,L,B)$ and $\D':=(\K,L',B')$ so that $(L,B)$ and $(L',B')$ are in bridge position with respect to $\Sigma$ and so that $S$ and $S'$ are $\tri$--regular isotopic to $\Sigma(\D)$ and $\Sigma(\D')$, respectively.

By Theorem~\ref{maintheorem}, $\D$ and $\D'$ are related by a sequence of singular band moves. By Corollary~\ref{cor:isotopicdiagrams}, if $\D$ and $\D'$ are isotopic, then the theorem holds.

Assume that $\D'$ is obtained from $\D$ by one singular band move (other than isotopy). We will show that $S'$ and $S$ become $\tri$--regular isotopic after some sequence of perturbations and deperturbations. The theorem will then hold via induction on the length of a sequence of band moves relating $\D$ and $\D'$.

Meier and Zupan \cite{meier2017bridge} previously showed that the claim holds when the move turning $\D$ into $\D'$ is a cup, cap, band swim, or band slide. The authors of this paper \cite{bandpaper} showed the claim is true when the move is a 2-handle/band slide, 2-handle/band swim, or dotted circle slides. These arguments were technically only made for nonsingular banded unlinks, so we repeat them in the singular setting for clarity, often repeating Meier and Zupan's arguments. In the following paragraphs, we now consider every singular band move that might transform $\D$ into $\D'$.

{\bf{1. Intersection/band pass.}}

\begin{figure}
    \centering
    \includegraphics[width=.9\textwidth]{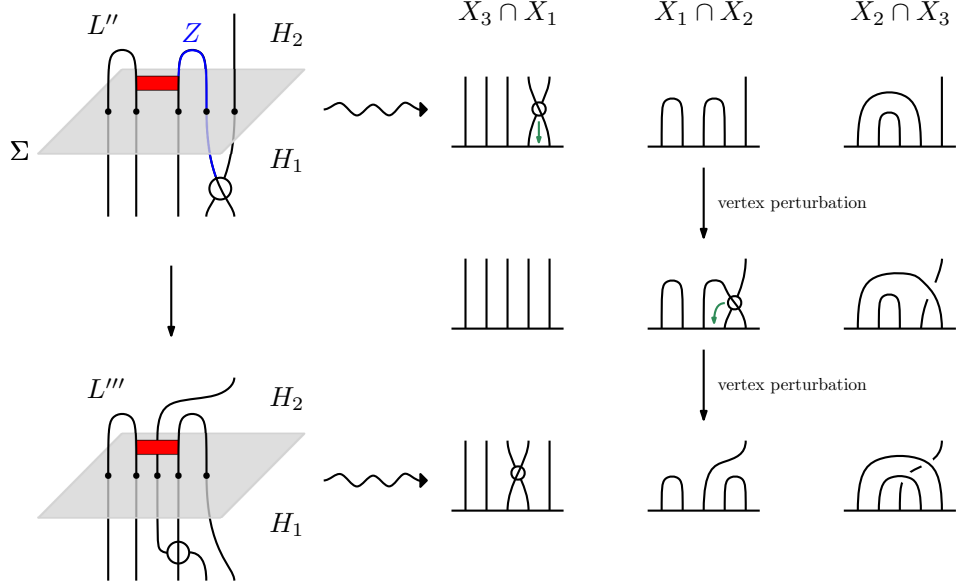}
    \caption{{\bf{Left:}} The singular banded unlink $(L''',B''')$ is obtained from $(L'',B'')$ by an intersection/band pass. {\bf{Right:}} We show that $\Sigma(L''',B''')$ (bottom) may be obtained from $\Sigma(L'',B'')$ (top) by two vertex perturbations and $\tri$--regular isotopy.}
    \label{fig:intbandpassbridge}
\end{figure}

Suppose $\D'$ is obtained from $\D$ by an intersection/band pass along a framed arc $z$ in $L$ between a vertex of $L$ and a band in $B$. Isotope $(L,B)$ so that $z$ is as in the top left of Figure~\ref{fig:intbandpassbridge}. Then isotope the rest of $L$ and $B$ outside a neighborhood of $z$ to obtain a banded link $(L'',B'')$ in bridge position. This banded singular link is isotopic to $(L,B)$, so by Corollary~\ref{cor:isotopicdiagrams} $S'':=\Sigma(L'',B'')$ is obtainable from $S$ by (de)perturbations and $\tri$--regular isotopy. Let $(L''',B''')$ be obtained from $(L'',B'')$ by performing the intersection/band pass along $z$, and let $S''':=\Sigma(L''',B''')$. Now the intersection of $S'''$ with each $X_i\cap X_j$ is isotopic rel boundary to the intersection of $S''$ with $X_i\cap X_j$, so $S'''$ is $\tri$--regular isotopic to $S''$. Finally, by Corollary~\ref{cor:isotopicdiagrams} we find that $S'''$ can be transformed into $S'$ by (de)perturbations and $\tri$--regular isotopy.

{\bf{2. Intersection/band slide.}}

Suppose $\D'$ is obtained from $\D$ by an intersection/band slide along a framed arc $z$ in $L$ between a vertex of $L$ and a band in $B$. Isotope $(L,B)$ so that $z$ is short and contained in $H_2$ in a neighborhood as in Figure~\ref{fig:intbandslidebridge}. Then isotope the rest of $L$ and $B$ outside this neighborhood to obtain a banded link $(L'',B'')$ in bridge position. This banded singular link is isotopic to $(L,B)$, so by Corollary~\ref{cor:isotopicdiagrams} $S'':=\Sigma(L'',B'')$ is obtainable from $S$ by (de)perturbations and $\tri$--regular isotopy. Let $(L''',B''')$ be obtained from $(L'',B'')$ by performing the intersection/band slide along $z$, and let $S''':=\Sigma(L''',B''')$. In Figure~\ref{fig:intbandslidebridge}, we show that $S'''$ can be obtained from $S''$ by perturbation and $\tri$--regular isotopy. Finally, by Corollary~\ref{cor:isotopicdiagrams} $S'''$ can be transformed into $S'$ by (de)perturbations and $\tri$--regular isotopy.

\begin{figure}
    \centering
    \includegraphics[width=.9\textwidth]{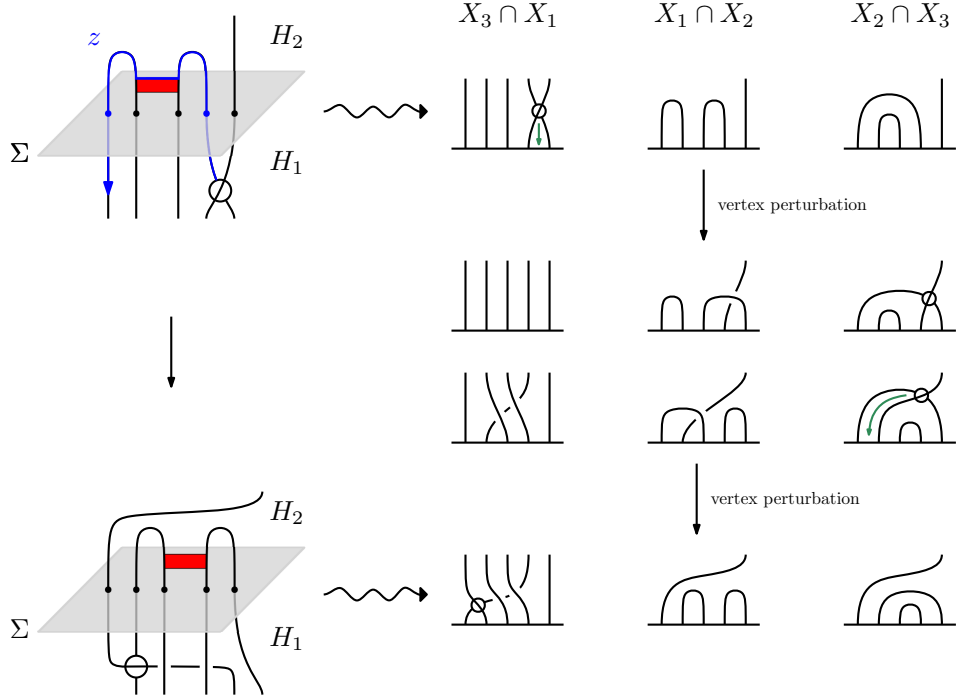}
    \caption{{\bf{Left:}} The singular banded unlink $(L''',B''')$ is obtained from $(L'',B'')$ by an intersection/band slide. {\bf{Right:}} We show that $\Sigma(L''',B''')$ (bottom) may be obtained from $\Sigma(L'',B'')$ (top) by two vertex perturbations and $\tri$--regular isotopy.}
    \label{fig:intbandslidebridge}
\end{figure}

{\bf{3. Intersection/band swim.}}

Suppose $\D'$ is obtained from $\D$ by an intersection/band swim along a framed arc $z$ between a vertex of $L$ and a band in $B$. Isotope $(L,B)$ so that $z$ is short and contained in $H_2$, as in Figure~\ref{fig:intbandswimbridge}. Then isotope the rest of $L$ and $B$ outside a neighborhood of $z$ to obtain a banded link $(L'',B'')$ in bridge position. This banded singular link is isotopic to $(L,B)$, so by Corollary~\ref{cor:isotopicdiagrams} $S'':=\Sigma(L'',B'')$ is obtainable from $S$ by (de)perturbations and $\tri$--regular isotopy. Let $(L''',B''')$ be obtained from $(L'',B'')$ by performing the intersection/band swim along $z$, and let $S''':=\Sigma(L''',B''')$. Now the intersection of $S'''$ with each $X_i\cap X_j$ is isotopic rel boundary to the intersection of $S''$ with $X_i\cap X_j$, so $S'''$ is $\tri$--regular isotopic to $S''$. Finally, by Corollary~\ref{cor:isotopicdiagrams} $S'''$ can be transformed into $S'$ by (de)perturbations and $\tri$--regular isotopy.

\begin{figure}
    \centering
    \includegraphics[width=.9\textwidth]{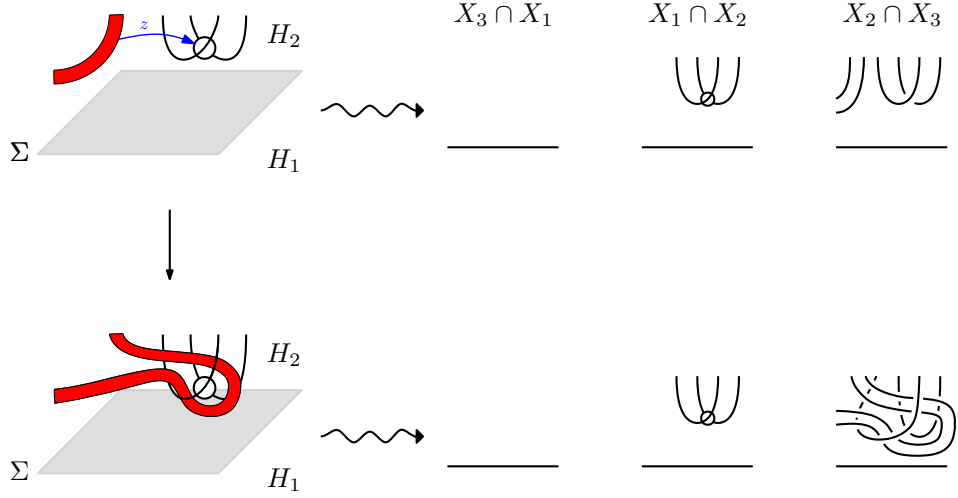}
    \caption{{\bf{Left:}} The singular banded unlink $(L''',B''')$ is obtained from $(L'',B'')$ by an intersection/band swim. {\bf{Right:}} We show that $\Sigma(L''',B''')$ (bottom) may be obtained from $\Sigma(L'',B'')$ (top) by $\tri$--regular isotopy.}
    \label{fig:intbandswimbridge}
\end{figure}

{\bf{4. Cup.}}

Suppose $\D'$ is obtained from $\D$ by a cup move. It does not matter in which direction we take the move, so assume that $L'$ is obtained from $L$ by adding a new unlink component $O$ contained in a ball not meeting $L$ or $B$, and $B'$ is obtained from $B$ by adding a trivial band $b_O$ from $L$ to $O$. By isotopy and intersection/band passes, we may take $O$ to be in 1--bridge position with respect to $\Sigma$, and $b_O$ to be in $H_2$, contained in a neighborhood as in Figure~\ref{fig:cupbridge}.  Performing the cup move yields a diagram $\D''$ that is related to $\D'$ by isotopy and intersection/band passes; by Corollary~\ref{cor:isotopicdiagrams} and Case 2, $\Sigma(\D'')$ can be transformed into $S'$ by perturbation and $\tri$--regular isotopy. Finally, we observe that $\Sigma(\D'')$ is obtained from the (perturbed) surface $S$ by perturbation (see Figure~\ref{fig:cupbridge}).

\begin{figure}
    \centering
    \includegraphics[width=.9\textwidth]{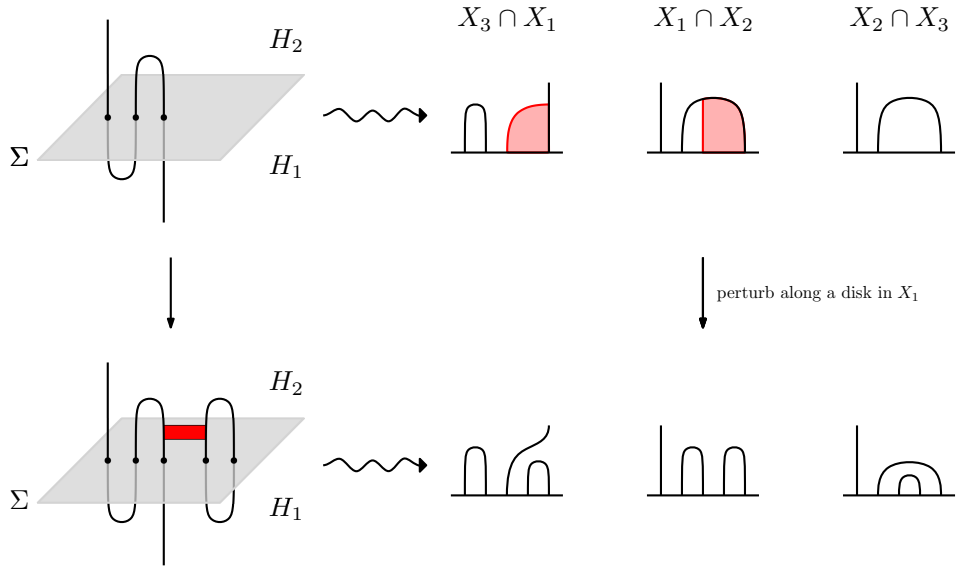}
    \caption{{\bf{Left:}} The singular banded unlink $(L''',B''')$ is obtained from $(L'',B'')$ by a cup move. {\bf{Right:}} We show that $\Sigma(L''',B''')$ (bottom) may be obtained from $\Sigma(L'',B'')$ (top) by an elementary perturbation and $\tri$--regular isotopy.}
    \label{fig:cupbridge}
\end{figure}

{\bf{5. Cap.}}

Suppose $\D'$ is obtained from $\D$ by a cap move. Again, it does not matter in which direction we take the move, so assume that $L'=L$ and $B'$ is obtained from $B$ by adding a trivial band $b$. By isotopy and intersection/band passes, we may take $b$ to have a neighborhood as in Figure~\ref{fig:capbridge}. Performing the cap move yields a diagram $\D''$ that is related to $\D'$ by isotopy and intersection/band passes; by Corollary~\ref{cor:isotopicdiagrams} and Case 2, $\Sigma(\D'')$ can be transformed into $S'$ by perturbation and $\tri$--regular isotopy. Finally, we observe that $\Sigma(\D'')$ is obtained from the (perturbed) surface $S$ by perturbation and deperturbation (see Figure~\ref{fig:capbridge}).

\begin{figure}
    \centering
    \includegraphics[width=.9\textwidth]{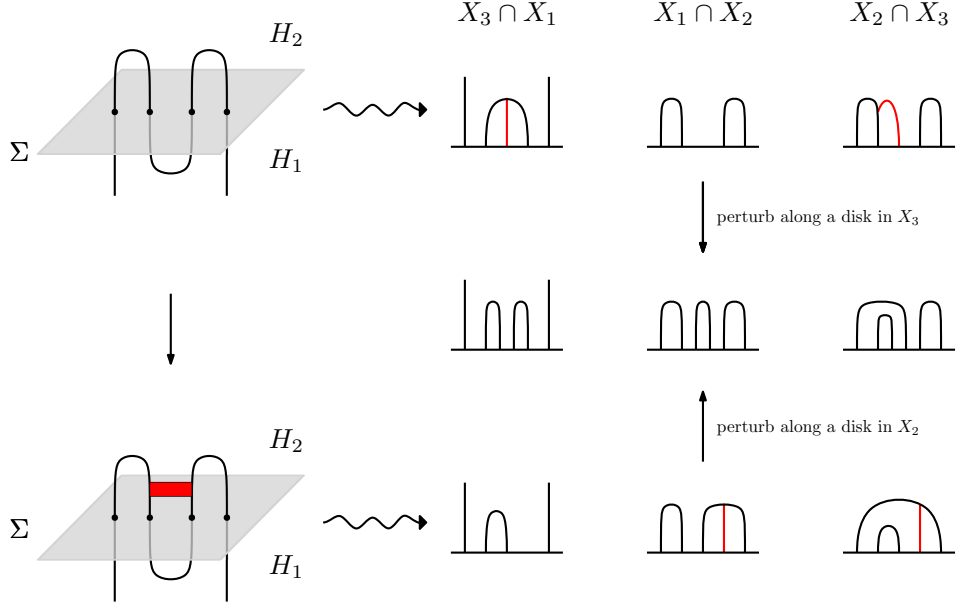}
    \caption{{\bf{Left:}} The singular banded unlink $\D''$ is obtained from $\D$ by a cap move. {\bf{Right:}} We show that $\Sigma(\D'')$ (bottom) may be obtained from $\Sigma(\D)$ (top) by an elementary perturbation and deperturbation and $\tri$--regular isotopy.}
    \label{fig:capbridge}
\end{figure}

{\bf{6. Band swim.}}

Suppose $\D'$ is obtained from $\D$ by a band swim. Isotope $\D$ to obtain a diagram in which the band swim looks as in Figure~\ref{fig:bandswimbridge}. Perform the band swim to obtain a diagram $\D''$ that is related to $\D'$ by isotopy; by Corollary~\ref{cor:isotopicdiagrams} and Case 2, $\Sigma(\D'')$ can be transformed into $S'$ by perturbation and $\tri$--regular isotopy. Finally, we observe that $\Sigma(\D'')$ is obtained from the (perturbed) surface $S$ by $\tri$--regular isotopy (see Figure~\ref{fig:bandswimbridge}).

\begin{figure}
    \centering
    \includegraphics[width=.9\textwidth]{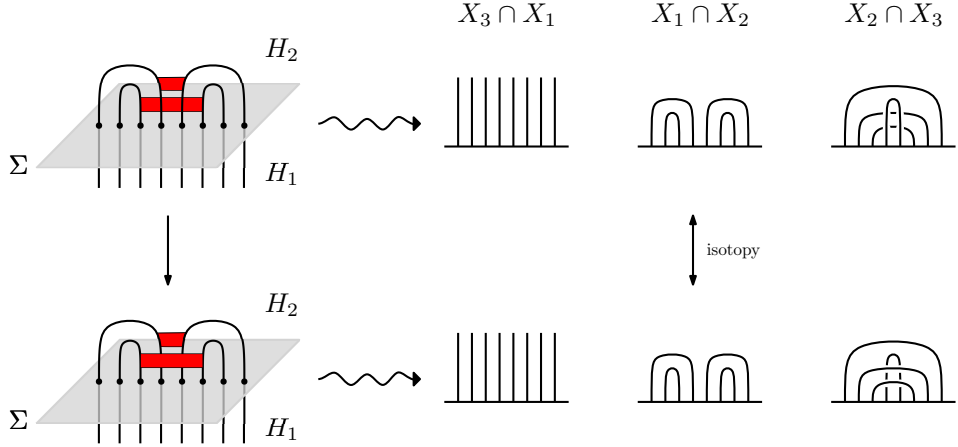}
    \caption{{\bf{Left:}} The singular banded unlink $\D''$ is obtained from $\D$ by a band swim. {\bf{Right:}} We show that $\Sigma(\D'')$ (bottom) may be obtained from $\Sigma(\D)$ (top) by $\tri$--regular isotopy.}
    \label{fig:bandswimbridge}
\end{figure}

{\bf{7. Band slide.}}

Suppose $\D'$ is obtained from $\D$ by a band slide. Isotope $\D$ to obtain a diagram in bridge position in which the desired band slide looks like Figure~\ref{fig:bandslidebridge}. By Corollary~\ref{cor:isotopicdiagrams}, the effect on $S$ can be achieved by (de)perturbation and $\tri$--regular isotopy. Call the result of the band slide $\D''$; by Corollary~\ref{cor:isotopicdiagrams} the surface $\Sigma(\D'')$ can be transformed into $S'$ by (de)perturbation and $\tri$--regular isotopy. In Figure~\ref{fig:bandslidebridge}, we observe that $\Sigma(\D'')$ is obtained from $S$ by perturbation and deperturbation.

\begin{figure}
    \centering
    \includegraphics[width=.9\textwidth]{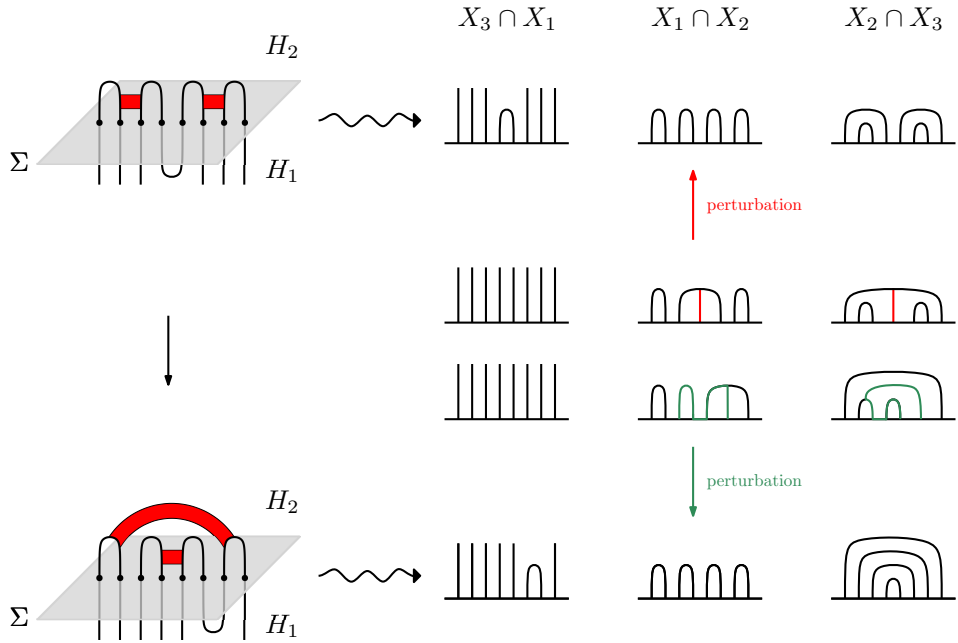}
    \caption{{\bf{Left:}} The singular banded unlink $\D''$ is obtained from $\D$ by a band slide. {\bf{Right:}} We show that $\Sigma(\D'')$ (bottom) may be obtained from $\Sigma(\D)$ (top) by an elementary perturbation and deperturbation and $\tri$--regular isotopy.}
    \label{fig:bandslidebridge}
\end{figure}

{\bf{8. 2--handle/band slide.}}

Suppose $\D'$ is obtained from $\D$ by sliding a band over a 2--handle via a framed arc $z$ between a band in $B$ and a 2--handle attaching circle in $\K$. As in Case 7, we may perturb $\D$ so that $z$ is contained in $H_2$ (See Figure~\ref{fig:2handlebandslidebridge}). Now performing the slide along $z$ yields a diagram $\D''$ that is related to $\D'$ by isotopy; by Corollary~\ref{cor:isotopicdiagrams} the surface $\Sigma(\D'')$ can be transformed into $S'$ by perturbation and $\tri$--regular isotopy. Finally, we observe that $\Sigma(\D'')$ is obtained from the (perturbed) surface $S$ by $\tri$--regular isotopy supported in $X_2$ and $X_3$.

\begin{figure}
    \centering
    \includegraphics[width=.9\textwidth]{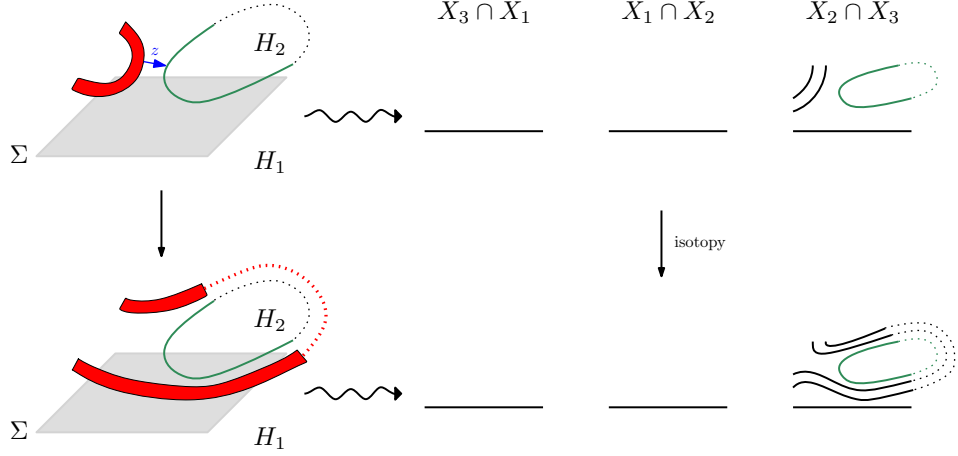}
    \caption{{\bf{Left:}} The singular banded unlink $\D''$ is obtained from $\D$ by a 2--handle/band slide. {\bf{Right:}} We show that $\Sigma(\D'')$ (bottom) may be obtained from $\Sigma(\D)$ (top) by $\tri$--regular isotopy.}
    \label{fig:2handlebandslidebridge}
\end{figure}

{\bf{9. 2--handle/band swim.}}

Suppose $\D'$ is obtained from $\D$ by swimming a 2--handle through a band. Isotope $\D'$ so that the swim looks like the one in Figure~\ref{fig:2handlebandswimbridge}. By Corollary~\ref{cor:isotopicdiagrams}, this can be achieved by (de)perturbations and $\tri$--regular isotopy of $S$. 
Now performing the swim along $z$ yields a diagram $\D''$ that is related to $\D'$ by isotopy; by Corollary~\ref{cor:isotopicdiagrams} the surface $\Sigma(\D'')$ can be transformed into $S'$ by perturbation and $\tri$--regular isotopy. Finally, we observe that $\Sigma(\D'')$ is obtained from the (perturbed) surface $S$ by $\tri$--regular isotopy supported in $X_2$ and $X_3$.

\begin{figure}
    \centering
    \includegraphics[width=.9\textwidth]{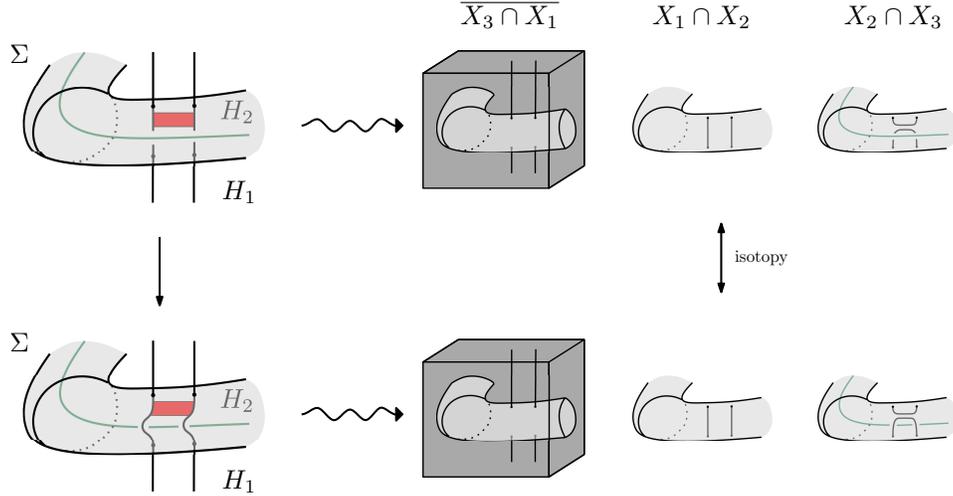}
    \caption{{\bf{Left:}} The singular banded unlink $\D''$ is obtained from $\D$ by a 2--handle/band swim. {\bf{Right:}} We show that $\Sigma(\D'')$ (bottom) may be obtained from $\Sigma(\D)$ (top) by $\tri$--regular isotopy.}
    \label{fig:2handlebandswimbridge}
\end{figure}

{\bf{10. Dotted circle/band slide}}
This follows from Theorem~\ref{theorem:bandlinkisotopy}, as slides over dotted circles are simply isotopies of the banded link $(L,B)$ in $M_{3/2}$.

{\bf{11. Dotted circle/link slide}}
Again, this follows from Theorem~\ref{theorem:bandlinkisotopy}, as slides over dotted circles are simply isotopies of the banded link $(L,B)$ in $M_{3/2}$.
\end{proof}

\FloatBarrier

\section{Some example applications}\label{sec:examples}
In this (comparatively short) section, we give a few sample applications of the diagrammatic theory of singular banded unlink diagrams.

\subsection{Calculating the Kirk invariant}\label{sec:kirk}
In \cite{schneiderman}, Schneiderman and Teichner classified all 2--component spherical links in $S^4$ up to link homotopy using the Kirk invariant $\sigma_i(F_1, F_2):= \lambda(F_i, F'_i)$. Here $i\in\{1,2\}$, $F'_i$ is a parallel push off of $F_i$, and $\lambda(F_i, F_i')$ is Wall's intersection invariant. Furthermore, $F_i$ denotes an  oriented immersed 2--sphere in $S^4$, with $F_1$ and  $F_2$  disjoint. The Kirk invariant takes values in $\mathbb{Z}[\mathbb{Z}]=\mathbb{Z}[x^{\pm}]$. 

Schneiderman--Teichner showed that the set of all 2--component spherical links in $S^4$ up to link homotopy is a free $R$--module, where $R = \mathbb{Z}[z_1, z_2]/(z_1 z_2)$ is freely generated by the Fenn--Rolfsen link $FR$ 
depicted in Figure~\ref{fig:FennRolfsen}. 

\begin{figure}
    \centering
    %% Creator: Inkscape 1.0.2-2 (e86c870879, 2021-01-15), www.inkscape.org
%% Accompanies image file 'FennRolfsen_orient.pdf' (pdf, eps, ps)
\begingroup%
  \makeatletter%
  \providecommand\color[2][]{%
    \errmessage{(Inkscape) Color is used for the text in Inkscape, but the package 'color.sty' is not loaded}%
    \renewcommand\color[2][]{}%
  }%
  \providecommand\transparent[1]{%
    \errmessage{(Inkscape) Transparency is used (non-zero) for the text in Inkscape, but the package 'transparent.sty' is not loaded}%
    \renewcommand\transparent[1]{}%
  }%
  \providecommand\rotatebox[2]{#2}%
  \newcommand*\fsize{\dimexpr\f@size pt\relax}%
  \newcommand*\lineheight[1]{\fontsize{\fsize}{#1\fsize}\selectfont}%
  \ifx\svgwidth\undefined%
    \setlength{\unitlength}{187.95090082bp}%
    \ifx\svgscale\undefined%
      \relax%
    \else%
      \setlength{\unitlength}{\unitlength * \real{\svgscale}}%
    \fi%
  \else%
    \setlength{\unitlength}{\svgwidth}%
  \fi%
  \global\let\svgwidth\undefined%
  \global\let\svgscale\undefined%
  \makeatother%
  \begin{picture}(1,0.70258739)%
    \lineheight{1}%
    \setlength\tabcolsep{0pt}%
    \put(0.7493042,0.00944695){\color[rgb]{0,0,0}\makebox(0,0)[lt]{\lineheight{1.25}\smash{\begin{tabular}[t]{l}$F_1$\end{tabular}}}}%
    \put(0.69270714,0.65454489){\color[rgb]{0,0,0}\makebox(0,0)[lt]{\lineheight{1.25}\smash{\begin{tabular}[t]{l}$F_2$\end{tabular}}}}%
    \put(0,0){\includegraphics[width=\unitlength,page=1]{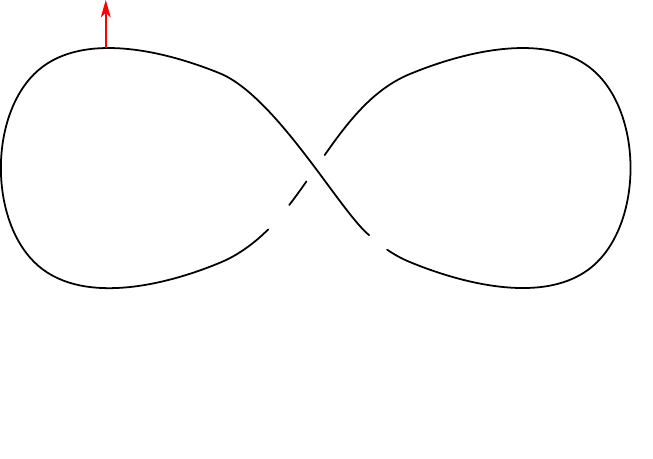}}%
    \put(0.17524399,0.68119459){\color[rgb]{1,0,0}\makebox(0,0)[lt]{\lineheight{1.25}\smash{\begin{tabular}[t]{l}$w_1$\end{tabular}}}}%
    \put(0.15762007,0.59012845){\color[rgb]{0,0,1}\makebox(0,0)[lt]{\lineheight{1.25}\smash{\begin{tabular}[t]{l}$w_2$\end{tabular}}}}%
    \put(0,0){\includegraphics[width=\unitlength,page=2]{FennRolfsen_orient.pdf}}%
    \put(0.21348181,0.49256871){\color[rgb]{1,0,0}\makebox(0,0)[lt]{\lineheight{1.25}\smash{\begin{tabular}[t]{l}$w_1$\end{tabular}}}}%
    \put(0.19702105,0.39400893){\color[rgb]{0,0,1}\makebox(0,0)[lt]{\lineheight{1.25}\smash{\begin{tabular}[t]{l}$w_2$\end{tabular}}}}%
    \put(0,0){\includegraphics[width=\unitlength,page=3]{FennRolfsen_orient.pdf}}%
  \end{picture}%
\endgroup%

    \caption{The Fenn--Rolfsen link. At the indicated points with arrows, a positive basis of the normal bundle is $(w_1,w_2)$, where $w_1$ is the drawn arrow pointing upward and $w_2$ points out of the page toward the reader.}
    \label{fig:FennRolfsen}
\end{figure}

In this subsection, we show how to compute the Kirk invariant of $FR$. This computation can be adapted to compute Wall's self-intersection invariant for general 2--component spherical links in arbitrary closed orientable 4--manifolds. Since $FR$ has a symmetry between its two components that reverses the orientation on one component, we have $\sigma_2=-\sigma_1$ and thus only compute $\sigma_1$.

Consider the singular banded unlink diagram of $FR=F_1\sqcup F_2$ as in Figure~\ref{fig:FennRolfsen}. Choose a basepoint $p$ far away from $FR$ and an arc $\gamma$ from $p$ to a point $q$ in $F_1$. Take a pushoff $F'_1$ of $F_1$ that transversely intersects $F_1$; simultaneously push off $\gamma$ to obtain an arc $\gamma'$ from $p$ to a point $q'$ of $F'_1$.

We thus have two parallel arcs $\gamma'$ and $\gamma$ from $p$ to $F'_1$ and from $p$ to $F_1$, respectively (as in Figure~\ref{fig:wall1}). Now delete a neighborhood of $F_2$ as in Figure~\ref{fig:wall1'}.

\begin{figure}
    \centering
    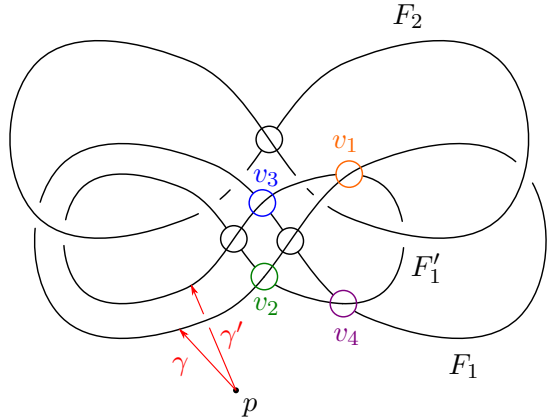
    \caption{A parallel pushoff $F'_1$ of $F_1$ that intersects $F_1$ transversely in 4 points yielding vertices $v_1, v_2, v_3, v_4$ in the singular banded unlink diagram. The intersections respectively have signs $s_{v_1}=1$, $s_{v_2}=-1$, $s_{v_3}=-1$, $s_{v_4}=1$.}\label{fig:wall1}
\end{figure}

\begin{figure}
    %% Creator: Inkscape 1.0.2-2 (e86c870879, 2021-01-15), www.inkscape.org
%% Accompanies image file 'wallkirby.pdf' (pdf, eps, ps)
\begingroup%
  \makeatletter%
  \providecommand\color[2][]{%
    \errmessage{(Inkscape) Color is used for the text in Inkscape, but the package 'color.sty' is not loaded}%
    \renewcommand\color[2][]{}%
  }%
  \providecommand\transparent[1]{%
    \errmessage{(Inkscape) Transparency is used (non-zero) for the text in Inkscape, but the package 'transparent.sty' is not loaded}%
    \renewcommand\transparent[1]{}%
  }%
  \providecommand\rotatebox[2]{#2}%
  \newcommand*\fsize{\dimexpr\f@size pt\relax}%
  \newcommand*\lineheight[1]{\fontsize{\fsize}{#1\fsize}\selectfont}%
  \ifx\svgwidth\undefined%
    \setlength{\unitlength}{185.09382834bp}%
    \ifx\svgscale\undefined%
      \relax%
    \else%
      \setlength{\unitlength}{\unitlength * \real{\svgscale}}%
    \fi%
  \else%
    \setlength{\unitlength}{\svgwidth}%
  \fi%
  \global\let\svgwidth\undefined%
  \global\let\svgscale\undefined%
  \makeatother%
  \begin{picture}(1,0.66526124)%
    \lineheight{1}%
    \setlength\tabcolsep{0pt}%
    \put(0,0){\includegraphics[width=\unitlength,page=1]{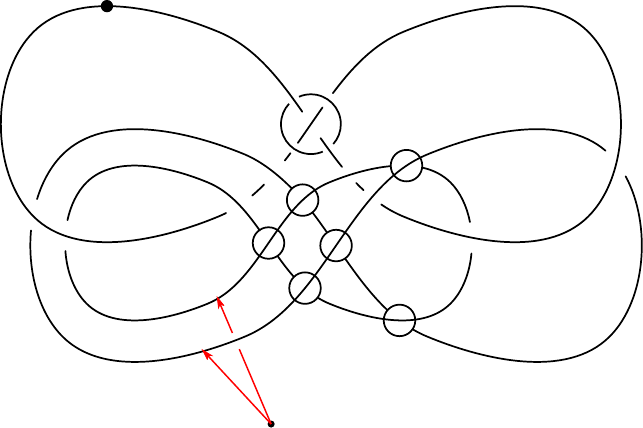}}%
    \put(0.54054262,0.4658748){\color[rgb]{0,0,0}\makebox(0,0)[lt]{\lineheight{1.25}\smash{\begin{tabular}[t]{l}$0$\end{tabular}}}}%
  \end{picture}%
\endgroup%

    \caption{We delete a neighborhood of $F_2$. The resulting singular banded unlink diagram of $F_1\cup F'_1$ is in a Kirby diagram with one 1--handle and one 2--handle.}\label{fig:wall1'}
\end{figure}

Pick a vertex $v$ between the diagrams of $F_1$ and $F'_1$, and choose arcs $\eta,\eta'$ contained in $F_1$ and $F'_1$ (respectively), from $q$ and $q'$ (respectively) to $v$. Let $C_v$ be the based loop obtained by concatenating $\gamma$, $\eta$, $-\eta'$, $-\gamma'$. There are four vertices $v_1,v_2,v_3,v_4$ shared between the diagrams of $F_1$, and $F'_1$; see Figure~\ref{fig:wall2} for potential loops $C_{v_i}$ for all $i=1,2,3,4$. Note that each loop might pass through the other intersections in the singular banded unlink diagram, but we always can perturb each loop a little bit on the actual surface $FR$ to miss the intersections. 

\begin{figure}
    \centering
    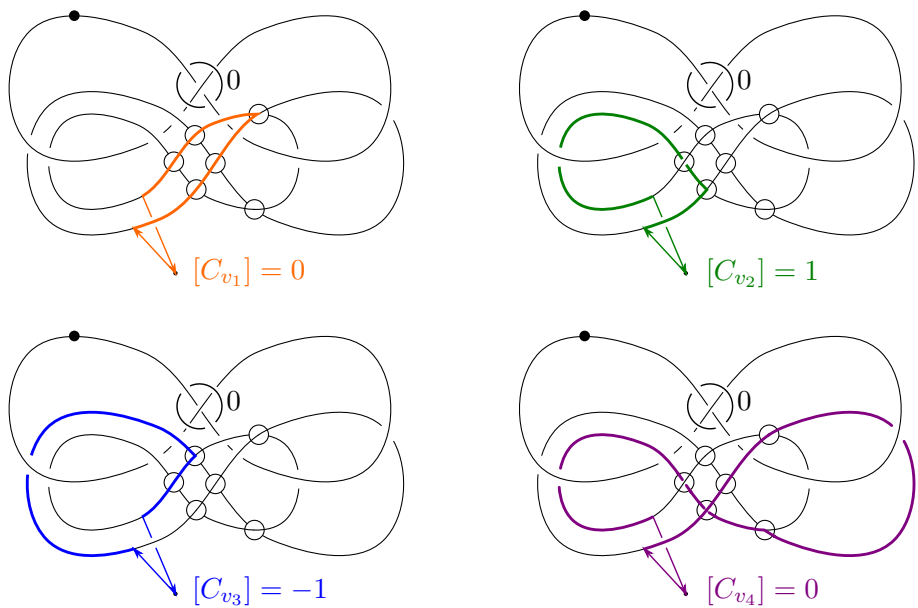
    \caption{The loops $C_{v_1},C_{v_2},C_{v_3}$, and $C_{v_4}$ respectively represent the elements $0,1,-1$, and $0$ in $H_1(S^4\setminus F_2)=\mathbb{Z}$.}\label{fig:wall2}
\end{figure}

Now each loop $C_{v_i}$ represents some element of $H_1(S^4 - F_2)=\mathbb{Z}$. In addition, each vertex has a sign $s_{v_i}\in\{-1,+1\}$ given by the sign of the corresponding intersection of $F_1$ and $F'_1$, which agrees with the sign of the crossing when the marking is resolved negatively. The values of $[C_{v_i}]$ and $s_{v_i}$ are as follows:

\hfil\begin{tabular}{c|rr}
i&$s_{v_i}$&$[C_{v_i}]$\\\hline
1&1&0\\
2&-1&1\\
3&-1&-1\\
4&1&0
\end{tabular}

The Kirk invariant $\sigma_1$ is then given by $$\sigma_1(FR)=\sum_{i=1}^4 s_{v_i}x^{[C_{v_i}]}=-x+2-x^{-1}.$$

The above computation generalizes for any singular banded unlink diagram of a 2--component spherical link $(F_1,F_2)$ in $S^4$; use whiskers from a basepoint $p$ to $F_1$ and a parallel pushoff $F'_1$ intersecting $F_1$ in $v_1,\ldots, v_n$ to form a loop $C_{v_i}$ for each $v_i$ representing $[C_{v_i}]\in H_1(S\setminus F_2)=\mathbb{Z}$. Then $\sigma_1(F_1,F_2)=\sum_{i=1}^n s_{v_i} x^{[C_{v_i}]}$.

\subsection{Immersed surfaces and stabilization}\label{sec:stabilization}
Hosokawa and Kawauchi \cite{hosokawa} showed that any pair of embedded oriented surfaces in $S^4$ become isotopic after some number of {\emph{stabilizations}}.

\begin{definition}\label{def:stabilizationsurface}
Let $F$ be a connected, self-transversely immersed genus $g$ oriented surface in $S^4$. Let $\gamma$ be an arc with endpoints on $F$ and which is normal to $F$ near $\partial \gamma$, but with the interior of $\gamma$ disjoint from $F$. Frame $\gamma$ so that $\gamma\times D^2$ is a 3--dimensional 1--handle with ends on $F$, and so that surgering $F$ along this 1--handle yields an oriented genus $(g+1)$ surface $F'$. Then we say $F'$ is obtained from $F$ by {\emph{stabilization.}}
\end{definition}

\begin{remark}
In Definition~\ref{def:stabilizationsurface}, there are two distinct ways to frame $\gamma$ to obtain a 3--dimensional 1--handle with ends on $F$. However, one of these choices will yield a non-orientable surface after surgery, so in fact the framing of $\gamma$ need not be specified.
\end{remark}

More generally, Baykur and Sunukjian \cite{baykur} extended this result for any pair of homologous embedded oriented surfaces in a closed orientable 4--manifold, and Kamada \cite{kamada1999unknotting} extended it to immersed oriented surfaces in $S^4$ using singular braid charts. In this subsection, we extend these above results in full generality, i.e.\ for any pair of homologous immersed surfaces in a closed orientable 4--manifold. 

\begin{theorem}\label{thm:stablyequiv}
Let $F$ and $F'$ be oriented self-transversely immersed surfaces in a closed, orientable 4--manifold $X$ which are homologous and have the same number of transverse double points of each sign. Then $F$ and $F'$ become isotopic after a sequence of stabilizations.
\end{theorem}
To prove Theorem~\ref{thm:stablyequiv}, we rely on the following diagrammatic lemma.

\begin{lemma}\label{lem:stablyembed}
Let $F$ be an oriented self-transversely immersed surface in a closed, orientable 4--manifold. Suppose $F$ has $p$ positive and $n$ negative self-intersections. After some number of stabilizations, $F$ becomes isotopic to the connected-sum of an embedded surface with $p$ copies of $U_+$ and $n$ copies of $U_-$, where $U_\pm$ denotes the result of performing a cusp move to the embedded unknotted 2--sphere to create a $\pm$ self-intersection.
\end{lemma}

\begin{proof}[Proof of Lemma~\ref{lem:stablyembed}]
Let $(\mathcal{K},L,B)$ be a singular banded unlink diagram of $F_0:=F$. Suppose that $F$ has $k=p+n>0$ self-intersections. Fix a vertex $v_0$ of $L$. Stabilize $F_0$ as in Figure~\ref{fig:tubing}, i.e.\ along an arc in $h^{-1}(3/2)$ that lies close to $v_0$. Call the resulting surface $G_0$. Now perform singular band moves as in Figure~\ref{fig:tubing} to see that $
G_0$ is isotopic to a connect sum $F_1\# U_{\epsilon^0}$, where $\epsilon^0$ is the sign of $v_0$, and $F_1$ is a self-transverse immersed surface with $k-1$ self-intersections.

If $k-1>0$ then repeat this argument on $F_1$ near another vertex $v_1$, stabilizing $F_1$ to obtain a surface $G_1$ that is isotopic to $F_2\#U_{\epsilon^1}$, where $F_2$ has $k-2$ self-intersections. Note $F$ is then stably isotopic to $F_2\#U_{\epsilon^1}\# U_{\epsilon^0}$.

Repeat inductively to find that $F$ is stably isotopic to $F_k\#(\#_p U_+)\#(\#_n U_-)$ for $F_k$ an embedded surface, as desired.
\end{proof}

\begin{figure}
    \centering
    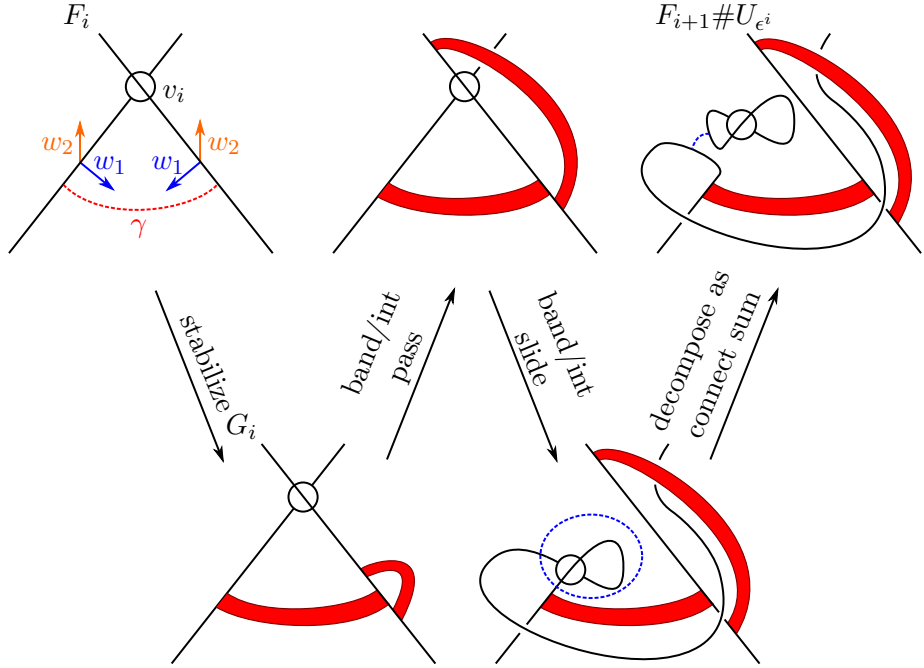
    \caption{{\bf{Top left:}} $F_i$ is an oriented surface with $k-i>0$ transverse self-intersections. Here we draw part of a singular banded unlink diagram for $F_i$ near a vertex $v_i$ representing a self-intersection of $F_i$. (In this drawing, it is a negative self-intersection. Changing the marking at $v_i$ yields a positive self-intersection.) We draw a positive normal basis $(w_1,w_2)$ along each local sheet of $F_i$ and indicate an arc $\gamma$ along which we may stabilize $F_i$. {\bf{From left to right following the arrows:}} We stabilize $F_i$ to obtain a surface $G_i$, and then isotope $G_i$ to realize a connect sum of a surface $F_{i+1}$ with $U_{\epsilon^i}$, where $\epsilon^i$ is the sign of the self-intersection represented by $v_i$.}
    \label{fig:tubing}
\end{figure}

\begin{proof}[Proof of Theorem~\ref{thm:stablyequiv}]
By Lemma~\ref{lem:stablyembed}, $F$ may be stabilized to a surface isotopic to $\hat{F}\#\left(\#_p U_+\right)\#\left(\#_n U_-\right)$ where $\hat{F}$ is an embedded surface and $p$ and $n$ are (respectively) the numbers of positive and negative  self-intersections of $F$. Applying the lemma also to $F'$ (recalling that $F'$ also has $p$ positive and $n$ negative self-intersections), we find that after suitable stabilizations $F'$ becomes isotopic to \[\hat{F'}\#\left(\#_p U_+\right)\#\left(\#_n U_-\right)\] for some embedded surface $\hat{F'}$. Since $U_\pm$ is nullhomologous, $\hat{F}$ and $\hat{F'}$ are homologous to $F$ and $F'$ and hence to each other. Then by \cite{baykur}, we know that $\hat{F}$ and $\hat{F'}$ (and hence $F$ and $F'$) are stably isotopic.
\end{proof}

\subsection{Unknotting 2--knots with regular homotopies}\label{sec:cassonwhitney}

In \cite{joseph2020unknotting}, Joseph, Klug, Ruppik, and Schwartz introduced the notion of the \emph{Casson--Whitney number} of a 2--knot, which is half the minimal number of finger and Whitney moves needed to change a given 2--knot to an unknot. They showed that the Casson--Whitney number of any non-trivial twist spin of a 2--bridge knot is one; i.e.\ that any non-trivial twist spin of a 2--bridge knot can be unknotted via one finger move followed by one Whitney move. In this subsection, we explicitly realize such a regular homotopy via singular banded unlink diagrams.

\begin{theorem}\cite{joseph2020unknotting}
The Casson--Whitney number of the $n$--twist spin ($|n|\neq 1$) $\tau^n K$ of a 2--bridge knot $K$ is one.
\end{theorem}

\begin{proof}
First, as in \cite{joseph2020unknotting}, we assume that the 2--bridge knot $K$ is in normal form \cite{conway1970enumeration} with the number of half-twists in each twist region even, as in Figure~\ref{fig:2-bridgecombine}. (That is, using the standard correspondence between 2--bridge link diagrams and continued fraction expansion, we arrange for a diagram of $K$ to correspond to a continued fraction $(a_1,b_1,\ldots,a_m,b_m)$ of all even integers. We write $K=K(a_1,b_1,\ldots, a_m,b_m)$.

Apply a finger move to the diagram of $\tau^n K$ in Figure~\ref{fig:2-bridgecombine} to obtain the first frame of Figure~\ref{fig:2-bridge1} (the visible twists are contained in the $\pm a_1$ twist boxes). In Figure~\ref{fig:2-bridge1} and Figure~\ref{fig:2-bridge2}, we show how to perform singular band moves with the result of decreasing $|a_1|$ by one. Repeating this sequence, we eventually arrange for $a_1$ to become $0$.

\begin{figure}
    \centering
    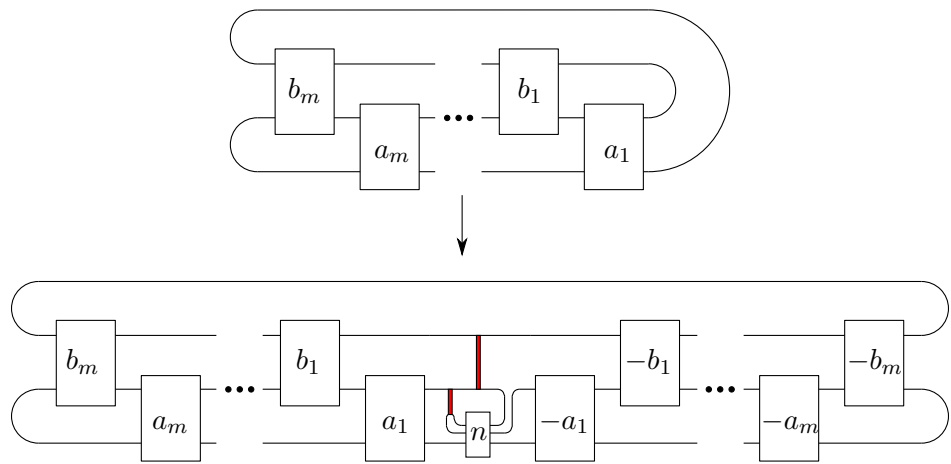
    \caption{\textbf{Top}: A 2--bridge knot $K$ in normal form. Here, $a_i$ and $b_i$ indicate signed numbers of whole twists (so each box has an even number of half-twists). \textbf{Bottom}: The $n$-twist spin $\tau^n K$ of $K$.}
   \label{fig:2-bridgecombine}
\end{figure}

\begin{figure}
    \centering
    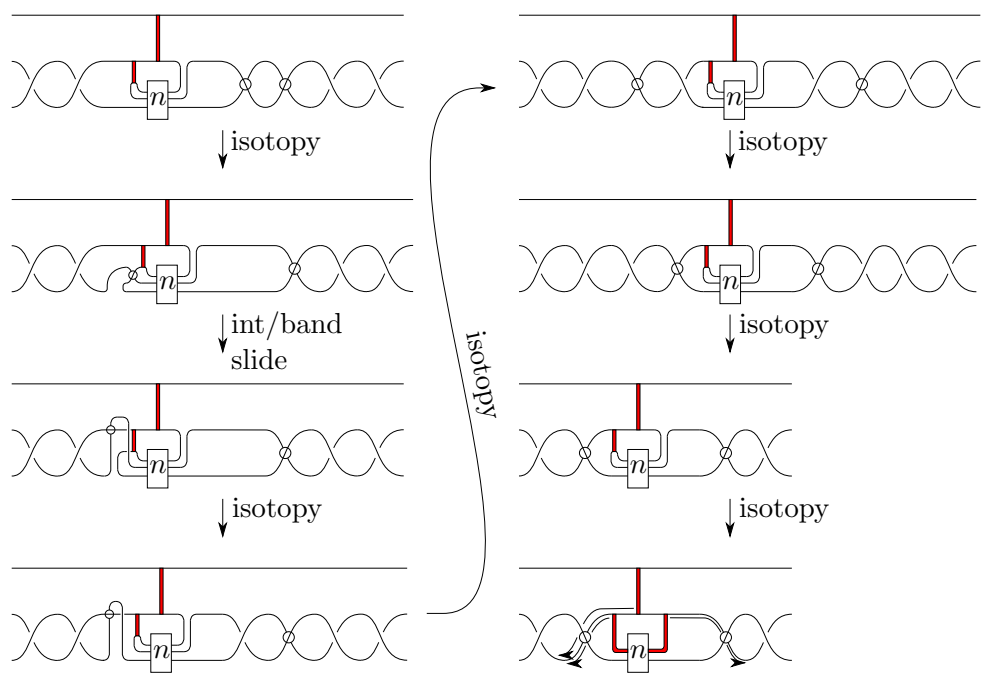
    \caption{The first frame is (a portion of the diagram) obtained from Figure~\ref{fig:2-bridgecombine} (bottom) by a finger move. We begin applying singular band moves with the goal of decreasing $|a_1|$ by one. In the last frame we indicate three band/intersection passes that yield the first frame of Figure~\ref{fig:2-bridge2}.}
   \label{fig:2-bridge1}
\end{figure}

In Figure~\ref{fig:2-bridge3}, we give another sequence of band moves (now assuming $a_1=0$) that decrease $|b_1|$ by one. Repeating this sequence, we eventually arrange for $a_1=b_1=0$.

\begin{figure}
    \centering
    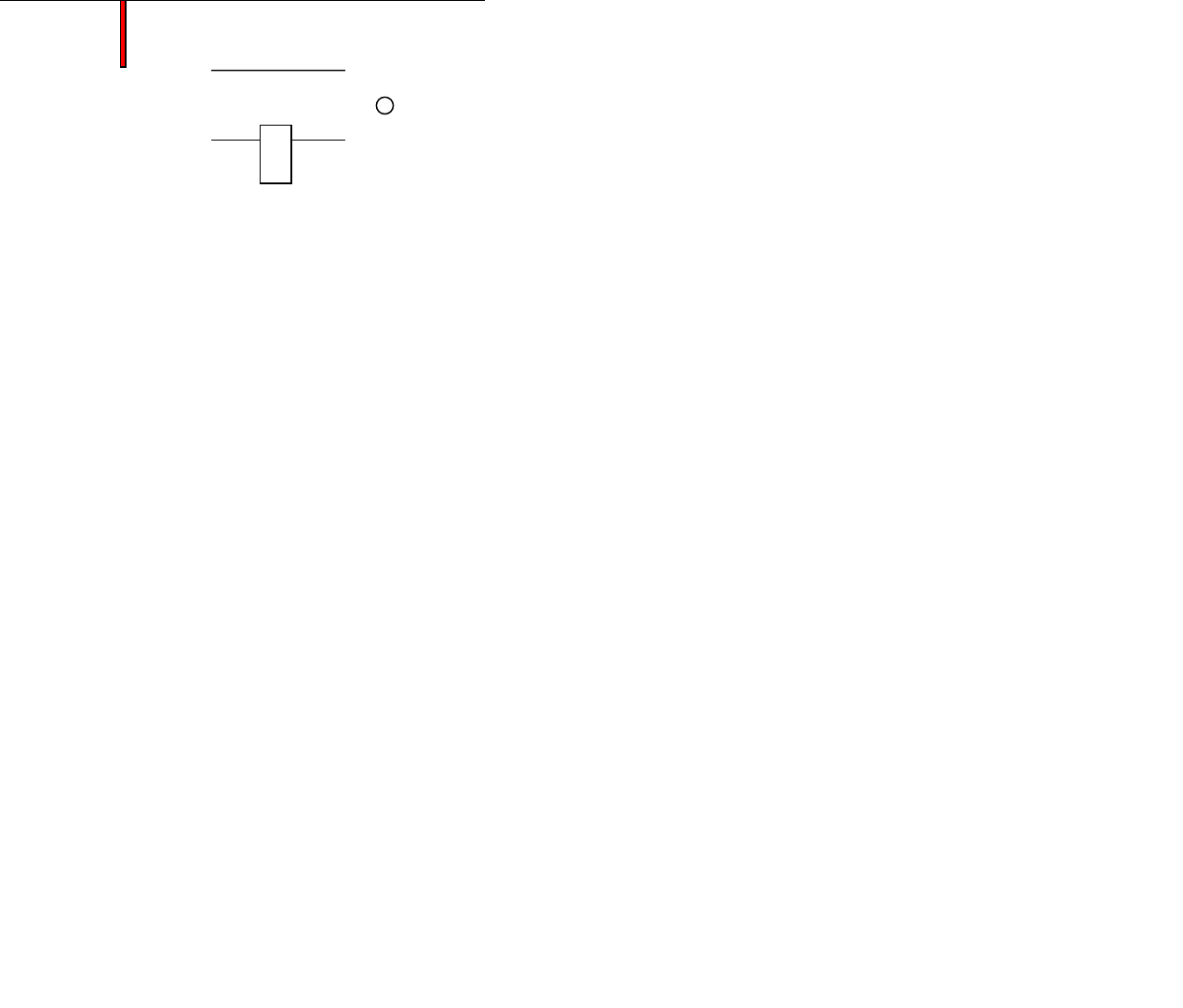
    \caption{Continuing from Figure~\ref{fig:2-bridge1}, we perform more singular band moves. In the last frame, the two vertices can be removed by a Whitney move, yielding the diagram from Figure~\ref{fig:2-bridgecombine} (bottom) but with $|a_1|$ decreased by one.}
   \label{fig:2-bridge2}
\end{figure}

We repeat these sequences of band moves to undo the twist boxes labelled $\pm a_2$, $\pm b_2$,$\ldots$, $\pm a_m$, $\pm b_m$, and then finally apply a Whitney move to remove the two vertices and obtain a singular banded unlink diagram for the $n$-twist spin of the unknot. This is an unknotted sphere, so we conclude that the Casson--Whitney number of $\tau^n K$ is one.
\begin{figure}
    \centering
    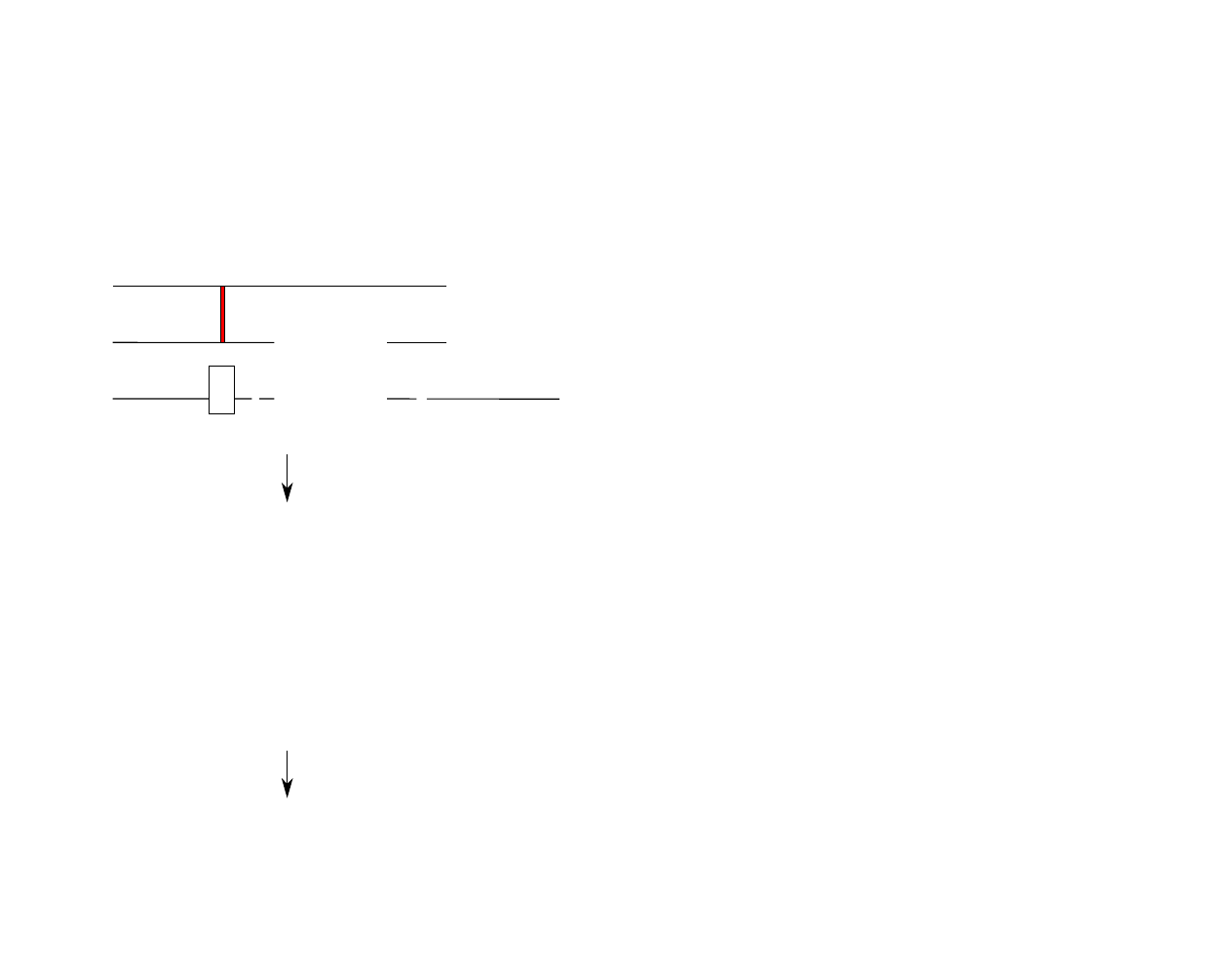
    \caption{The first frame agrees with the last frame of Figure~\ref{fig:2-bridge2} after $|a_1|$ is decreased to zero. We can then perform singular band moves to the diagram to decrease $|b_1|$ by one.}
   \label{fig:2-bridge3}
\end{figure}

\end{proof}

\bibliographystyle{plain}
\bibliography{biblio.bib}

\end{document}